\documentclass[12pt]{amsart}
\usepackage{amsfonts}
\usepackage{amsmath}
\usepackage{amssymb}
\usepackage{mathrsfs}
\usepackage{geometry}
\usepackage{graphicx}
\usepackage{subfigure}
\usepackage{cite}
\usepackage{hyperref}
\usepackage{microtype}
\usepackage{enumitem}
\usepackage{mathrsfs}
\usepackage{tikz}
\usepackage{tikz-cd}
\usepackage{color}
\allowdisplaybreaks[1]

\newcommand{\la}{\lambda}
\newcommand{\al}{\alpha}

\newcommand{\ga}{\gamma}

\newcommand{\ve}{\varepsilon}
\newcommand{\vp}{\varphi}

\newcommand{\R}{\mathbb{R}}

\newcommand{\T}{\mathbb{T}}

\newcommand{\f}{\forall}

\newcommand{\ccc}{\cdot\cdot\cdot}

\newcommand{\n}[1]{\Vert #1\Vert }
\newcommand{\bn}[1]{\big \Vert #1 \big \Vert }
\newcommand{\bbn}[1]{\Big\Vert #1 \Big \Vert }
\newcommand{\lr}[1]{\left\{ #1\right\} }
\newcommand{\lrc}[1]{\left[ #1\right] }
\newcommand{\lrs}[1]{\left( #1\right) }

\newcommand{\lra}[1]{\langle #1\rangle}

\newcommand{\babs}[1]{\big | #1 \big| }
\newcommand{\bbabs}[1]{\Big | #1 \Big| }
\newcommand{\wt}[1]{\widetilde{#1} }
\newcommand{\pa}{\partial}
\newcommand{\ol}{\overline}

\begin{document}

\newtheorem{theorem}{Theorem}[section]
\newtheorem{lemma}[theorem]{Lemma}

\theoremstyle{definition}
\newtheorem{definition}[theorem]{Definition}
\newtheorem{example}[theorem]{Example}
\newtheorem{remark}[theorem]{Remark}

\numberwithin{equation}{section}

\newtheorem{proposition}[theorem]{Proposition}
\newtheorem{corollary}[theorem]{Corollary}
\newtheorem*{notation}{Notation}

\renewcommand{\figurename}{Fig.}

\title[1D global derivation of Vlasov-Poisson]{Global derivation of the 1D Vlasov-Poisson equation from quantum many-body dynamics with screened Coulomb potential }
\author[X. Chen]{Xuwen Chen}
\address{Department of Mathematics, University of Rochester, Rochester, NY 14627, USA}
\email{xuwenmath@gmail.com}

\author[S. Shen]{Shunlin Shen}
\address{School of Mathematical Sciences, University of Science and Technology of China, Hefei 230026,
Anhui Province, China}
\email{slshen@ustc.edu.cn}

\author[P. Zhang]{Ping Zhang}
\address{Academy of Mathematics \& Systems Science and Hua Loo-Keng Center for Mathematical Sciences, Chinese Academy of Sciences, Beijing 100190, China, and School of Mathematical Sciences, University of Chinese Academy of Sciences, Beijing 100049, China
}

\email{zp@amss.ac.cn}

\author[Z. Zhang]{Zhifei Zhang}
\address{School of Mathematical Sciences, Peking University, Beijing, 100871, China}

\email{zfzhang@math.pku.edu.cn}

\subjclass[2010]{Primary 35Q55, 35Q83, 35D30; Secondary 35A01, 81V70, 82C70.}
\date{}

\dedicatory{}

\begin{abstract}
We study the 1D quantum many-body dynamics with a screened Coulomb potential in the mean-field setting. Combining the quantum mean-field, semiclassical, and Debye length limits, we prove the global derivation of the 1D Vlasov-Poisson equation. We tackle the difficulties brought by the pure state data, whose Wigner transforms converge to Wigner measures. We find new weighted uniform estimates around which we build the proof. As a result, we obtain, globally, stronger limits, and hence the global existence of solutions to the 1D Vlasov-Poisson equation subject to such Wigner measure data, which satisfy conservation laws of mass, momentum, and energy, despite being measure solutions. This happens to solve the 1D case of an open problem regarding the conservation law of the Vlasov-Poisson equation raised in \cite{DP88} by Diperna and Lions.

\end{abstract}
\keywords{Quantum Many-body Dynamics, Vlasov-Poisson Equation, Global Weak Solution, Quantum Mean-field Approximation, Semiclassical Limit}
\maketitle
\tableofcontents

\section{Introduction}
Per the superposition principle, the dynamics of $N$ quantum particles interacting through a two-body interaction potential
are governed by
the linear $N$-body Schr\"{o}dinger equation
\begin{equation}\label{equ:n-body NLS}
\left\{
\begin{aligned}
i\hbar \pa_{t}\Psi_{N,\hbar,\ve}=&H_{N,\hbar,\ve}\Psi_{N,\hbar,\ve},\\
\Psi_{N,\hbar,\ve}(0)=&\Psi_{N,\hbar}^{\mathrm{in}},
\end{aligned}
\right.
\end{equation}
where $\Psi_{N,\hbar,\ve}(t,x_{1},..,x_{N})\in \mathbb{C}$ is the $N$-particle wave function
 at time $t$ and the Hamiltonian operator is
\begin{align}
H_{N,\hbar,\ve}=\sum_{j=1}^{N}-\frac{1}{2}\hbar^{2}\Delta_{x_{j}}+\frac{1}{N}\sum_{1\leq j<k\leq N}V_{\ve}(x_{j}-x_{k}).
\end{align}

In many physical systems dealing with charges in which electro-magnetism is involved, an important physically observable phenomenon is the screening effect, which arises from the collective behaviors of charged particles and modifies the long-range Coulomb potential into an exponentially decaying form at a distance.
The concept of a screened Coulomb potential arises in the physics of many-body systems, particularly in plasma physics, condensed matter physics, and certain areas of molecular physics.
 For example, for an electrically neutral system, the distribution of charges gives rise to an electric potential $V(x)$ that satisfies Poisson's equation
\begin{align*}
\nabla^{2}V(x)=-\sum_{j=1}^{N}q_{j}n_{j}(x),
\end{align*}
where $q_{j}$ is the charge and $n_{j}(x)$ is the concentration at position $x$. Under suitable physical assumptions, one often reduces the Poisson's equation to a simpler one
\begin{align*}
(\nabla^{2}-\ve^{2})V(x)= \delta(x),
\end{align*}
where the parameter $\ve$ denotes the Debye length that characterizes different physical regimes. For more details on the derivation of a screened Coulomb potential, see also the standard monograph \cite{akhiezer2017plasma}. For more physical background on the screened Coulomb potential, see for instance \cite{CS04,HSE03,HC22,Liu21,RGH70,WS11}.

In the paper, we consider the 1D screened Coulomb potential
\begin{align}\label{equ:screened coulomb}
V_{\ve}(x)=\pm \frac{1}{2}|x|e^{-\ve|x|},
\end{align}
where the sign $\pm$ denotes defocusing/focusing.
Here, the form \eqref{equ:screened coulomb} is a version of approximate solution to the 1D Poisson's equation.

Totally different from the 3D Coulomb potential $\frac{1}{|x|}$ which has a slow decay at the infinity, the 1D Coulomb potential $|x|$ tends to infinity as $|x|\to \infty$. Hence, for the 1D interacting systems, it is reasonable to consider the screened Coulomb potential model, as it seems to be counterintuitive that the interaction force grows to be infinitely large with the distance between particles increasing to infinity. From the perspective of physics, the screening effect is widely present in many physical systems. In fact, the Debye length is an
experimentally observable parameter of $N$-body systems. Some people even use that to define the experimental regimes.

Taking into account the screening effect, the effective interaction range between particles, by which different physical regimes are characterized, is quantified by the Debye length. The most interesting regime might be the Debye length limit $\ve\to 0$, as the full 1D Coulomb potential is formally recovered in the limit. Thus, not only in the theoretical physics but also in the numerical computation, it is common to take the screened model as an approximation.

 Nevertheless, it is a challenge to provide a rigorous proof, as the 1D screened Coulomb potential is far from a perturbation or a regularized model for the Coulomb potential. Moreover,
a key goal in mathematical physics is to understand how
nonlinear equations of classical physics emerge as descriptions of quantum microscopic linear dynamics in
appropriate asymptotic regimes.
Staring from the quantum many-body dynamics \eqref{equ:n-body NLS}, we are concerned with the asymptotic limit of the $N$-body wave function
as the particle number $N\to \infty$, the Planck's constant $\hbar\to 0$, and the Debye length $\ve\to 0$, which leads to a kinetic equation, the Vlasov-Poisson equation
\begin{equation}\label{equ:vlasov-poisson,1d}
\left\{
\begin{aligned}
&\pa_{t}f+\xi\pa_{x}f+ E\pa_{\xi}f=0,\\
&\pa_{x}E=\pm \int_{\R}fd\xi,\\
&f(0)=f_{0}.
\end{aligned}
\right.
\end{equation}
The Vlasov-Poisson systems describe the evolution of the distribution function $f(t,x,\xi)$ of particle
under a self-consistent electric or gravitational field. There have been many developments such as \cite{DP88,DP88solutions,LP91,ZM94} on the global well-posedness problem of weak/measure solutions to the Vlasov-Poisson equation. Moreover, as pointed out in the review \cite[p.278]{DP88}, apart from the uniqueness and regularity, the conservation law for the weak solutions in the kinetic theory is an important open question.

The quantum many-body dynamics \eqref{equ:n-body NLS} and the kinetic equation \eqref{equ:vlasov-poisson,1d} are linked by the Wigner transform, which
takes the form that
\begin{align}\label{equ:wigner function}
f_{N,\hbar,\ve}^{(1)}(t, x, \xi)=W_{\hbar}[\ga_{N,\hbar,\ve}^{(1)}](t,x,\xi)=\frac{1}{2 \pi} \int_{\mathbb{R}} e^{-i \xi y} \ga_{N,\hbar,\ve}^{(1)}\left(t, x+\frac{\hbar y}{2}, x-\frac{\hbar y}{2}\right) d y,
\end{align}
where the first marginal density is
\begin{align}
\ga_{N,\hbar,\ve}^{(1)}(t,x,x')=\int_{\R^{N-1}}\Psi_{N,\hbar,\ve}(t,x,x_{2},..,x_{N})
\ol{\Psi_{N,\hbar,\ve}(t,x',x_{2},..,x_{N})}dx_{2}\ccc dx_{N}.
\end{align}
The Wigner function $f_{N,\hbar,\ve}^{(1)}(t, x, \xi)$ turns the spatial marginal density into a real-valued density on the phase space, and satisfies induced basic properties in kinetic theory such as the conservation laws of mass, momentum, and energy.

Our goal is to justify the limit process in which the Wigner function \eqref{equ:wigner function} from the quantum
many-body dynamics \eqref{equ:n-body NLS} tends to the Vlasov-Poisson equation \eqref{equ:vlasov-poisson,1d}.
\begin{theorem}[Main theorem]\label{thm:main theorem}
Let $\Psi_{N,\hbar,\ve}(t)$ be the solution to the $N$-body dynamics \eqref{equ:n-body NLS}, and $f_{N,\hbar,\ve}^{(1)}(t)$ be the Wigner transform of $\ga_{N,\hbar,\ve}^{(1)}(t)$.
Assume the initial data $\Psi_{N,\hbar,\ve}(0)$ is normalized and factorized in the sense that
\begin{align*}
\Psi_{N,\hbar}^{\mathrm{in}}=\prod_{j=1}^{N}\psi_{\hbar}^{\mathrm{in}}(x_{j}), \quad \n{\psi_{\hbar}^{\mathrm{in}}}_{L_{x}^{2}}=1,
\end{align*}
and $\psi_{\hbar}^{\mathrm{in}}$ satisfies the uniform bounds
\begin{align}\label{equ:uniform bounds,initial data}
\n{|x|\psi_{\hbar,\ve}^{\mathrm{in}}}_{L_{x}^{2}}\leq C,\quad \n{\hbar^{k}\pa_{x}^{k}\psi_{\hbar}^{\mathrm{in}}}_{L_{x}^{2}}\leq C^{k}k^{k},\quad k\geq 0.
\end{align}

Then there exist a subsequence of $\lr{f_{N,\hbar,\ve}^{(1)}}$, which we still denote by $\lr{f_{N,\hbar,\ve}^{(1)}}$, and a non-negative bounded Radon measure
$$f(t,dx,d\xi)\in C([0,\infty);\mathcal{M}^{+}(\R^{2})-w^{*}),$$
such that
\begin{align}\label{equ:convergence,N-body}
\lim_{(N,\hbar,\ve)\to (\infty,0,0)}\int_{0}^{T}\iint_{\R^{2}}\lrs{f_{N,\hbar,\ve}^{(1)}(t,x,\xi)-f(t,x,\xi)}\phi dxd\xi dt=0,\quad
\end{align}
for all $T>0$ and $\phi\in L_{t}^{1}([0,T];\mathcal{A})$, where the space $\mathcal{A}$ is defined in \eqref{equ:function space,A}.
The Wigner measure $f(t,dx,d\xi)$ is a weak solution to the Vlasov-Poisson equation \eqref{equ:vlasov-poisson,1d} with the initial measure datum $f(0,dx,d\xi)$ in the sense of Definition \ref{def:weak solution,vp}.
Moreover, the Wigner measure $f(t,dx,d\xi)$ satisfies the conservation laws of mass, momentum, and energy
\begin{align}
&\iint_{\R^{2}} f(t,dx,d\xi)=\iint_{\R^{2}} f(0,dx,d\xi),\label{equ:conservation,mass,theorem}\\
&\iint_{\R^{2}} \xi f(t,dx,d\xi)=\iint_{\R^{2}} \xi f(0,dx,d\xi),\label{equ:conservation,momentum,theorem}\\
&\iint_{\R^{2}} \xi^{2}f(t,dx,d\xi)\pm \frac{1}{2}\iint_{\R^{2}} |x-y|\rho(t,dx)\rho(t,dy)\label{equ:conservation,energy,theorem}\\
=&\iint_{\R^{2}} \xi^{2}f(0,dx,d\xi)\pm \frac{1}{2}\iint_{\R^{2}} |x-y|\rho(0,dx)\rho(0,dy),\notag
\end{align}
where $\rho(t,dx)=\int_{\R}f(t,dx,\xi)d\xi$.
\end{theorem}
\begin{remark}[Global existence and conservation laws]
One could also consider Theorem \ref{thm:main theorem} as proof of global existence of measure solutions to \eqref{equ:vlasov-poisson,1d} subject to such Wigner measure data with conversation of mass, momentum, and energy. Staring from the quantum many-body dynamics, we happen to solve the 1D case of an open problem regarding the conservation law of the Vlasov-Poisson equation raised in \cite{DP88} by Diperna and Lions.
\end{remark}
\begin{remark}[Existence of initial data]
One can choose the initial data as $\psi_{\hbar}^{\mathrm{in}}*j_{\hbar}$, where the mollifier $j_{\hbar}(x)=\hbar^{-1}j(x/\hbar)$ with $0\leq j(x)\in \mathcal{S}(\R)$, $\int_{\R}j(x)dx=1$, and $\int_{\R}|\pa_{x}^{k}j|dx\leq Ck^{k}$ for all $k\geq 0$. Then the uniform bounds \eqref{equ:uniform bounds,initial data} are satisfied. For example, one can choose $j(x)=\pi^{-1}e^{-x^{2}}$.
\end{remark}
\begin{remark}[The torus case]
With some modifications, Theorem \ref{thm:main theorem} can be extended to the torus case for the Coulomb potential, as the screening effect is more specialized for $\R$ and there is no essential difference between the unscreened and screened cases on $\T$ from the mathematical view.
\end{remark}
\begin{remark}[Fixed Debye length]
Our proof also works for any fixed Debye length. The limit equation would then be a Vlasov equation with a screened Coulomb potential characterized by the Debye length.
\end{remark}

Currently, there have been many nice developments \cite{BPSS16,Car08,CSZ23,EESY04,GMP16,GP17,GP22,GMP03,LS23,LP93,MM93,Ser20,Zha02,ZZM02} devoted to the derivation of the Vlasov-type equations from quantum systems. The semiclassical limit of the one-body Schr\"{o}dinger equation leading to the Vlasov equation
was first systematically studied in \cite{LP93}. For the Coulomb potential case, in \cite{LP93,MM93}, this problem was solved for a mixed state initial data
\begin{align}
\sum_{j=1}^{\infty}\la_{j}^{\hbar}\psi_{j}^{\hbar}(x)\ol{\psi_{j}^{\hbar}(x')},\quad \sum_{j=1}^{\infty}
\la_{j}^{\hbar}=1,
\end{align}
under the uniform bound condition
\begin{align}\label{equ:uniform bound condition}
\frac{1}{\hbar^{3}}\sum_{j=1}^{\infty}(\la_{j}^{\hbar})^{2}\leq C.
\end{align}
However, a pure state in which $j=1$, $\la_{j}^{\hbar}=1$ cannot satisfy \eqref{equ:uniform bound condition}. It was then solved in \cite{ZZM02} for the 1D case with general initial data including the pure state densities. For the higher dimensional case, apart from the local derivation for the monokinetic case such as \cite{GP22,Ser20,Zha02}, it remains an open problem for the global derivation of the 3D Vlasov-Poisson equation from the quantum and classical microscopic systems.

In our setting of justifying the global limit to the weak solution of the Vlasov-Poisson equation from the 1D quantum many-body dynamics,
there are also several hard problems which we list below.

\begin{enumerate}
\item
The problem of the pure state density.
The quantum mean-field problem in the $N\to \infty $ limit is closely related to the Bose-Einstein condensate, a physical phenomenon that all particles take the same quantum state. That is, the $N$-body wave function takes the product form that
\begin{align*}
\Psi_{N,\hbar,\ve}(t,x_{1},..,x_{N})\sim \prod_{j=1}^{N}\psi_{\hbar,\ve}(t,x_{j}),
\end{align*}
which yields a pure state marginal density
$$\ga_{N,\hbar,\ve}^{(1)}(t,x,x') \sim \psi_{\hbar,\ve}(t,x)\ol{\psi_{\hbar,\ve}}(t,x').$$
However, the Wigner transform of a pure state density is only known to converge to a Wigner measure as pointed by Lions and Paul in \cite{LP93}.
That is, to obtain the Vlasov-Poisson equation from a pure state density, we have to work in the non-smoothing setting and deal with a not only weak but also measure solution of the Vlasov-Poisson equation. Many exsiting strong-weak stability arguments such as the modulated energy method might not be valid, as the uniqueness of the limiting weak solution to the Vlasov-Poisson equation is unknown.

\item The non-smoothness of the potential at the origin. For the $C^{1,1}$ interaction potentials, the mean-field and semiclassical approximation to the Vlasov-type equation has been proven in \cite{GP17}. The singularity at the origin hinders the application of the method in \cite{GP17}. Hence, new ideas are required to deal with the singularity at the origin.

\item Weak convergence problem in the Debye length limit. To recover the Vlasov-Poisson equation with the full 1D Coulomb potential which is singular at both the origin and the infinity, we need to
establish the limit process, which only holds in the weak sense that
\begin{align*}
\lim_{\ve\to 0}\int_{\R}V_{\ve}(x) \vp(x)dx=\int_{\R} \frac{1}{2}|x|\vp(x)dx,\quad \f \vp\in C_{c}(\R).
\end{align*}
Therefore, two weak limits, the semiclassical and Debye length limits, entangle here.
Especially for the convergence of the nonlinear term, it requires new uniform estimates and a cancellation structure
to deal with these two weak limits at the same time.

\item The conservation laws for the limit measure solution. As we start from  basic physics model, the linear $N$-body dynamics, it is naturally desired that the limit solution satisfies more physical properties, such as the conservation laws. However, from the view of mathematics, like many other open problems such as the conservation of energy for the renormalized solution of the Boltzmann equation \cite{DP89} and the Vlasov-Poisson equation \cite{DP88,DP88solutions}, it is highly non-trivial to prove these conservation laws for the limit weak measure solution.
\end{enumerate}

\subsection{Outline of the Proof of the Main Theorem}

We divide the proof into the following five steps.

\noindent \textbf{Step 1. Preliminary reduction to a one-body nonlinear Schr\"{o}dinger equation.}

Since the first wave of work, for example \cite{AGT07,BEGMH02,EESY06,ES07,
ESY06,ESY07,ESY09,ESY10,EY01} and the references within on deriving the nonlinear Schr\"{o}dinger equations from the quantum many-body dynamics with the delta-type and Coulomb potentials,
there have been a large quantities of work on the study of the quantum mean-field limit using various methods, such as  \cite{BG24,BS19,BOS15,CP11,CP14,CH16on,CH19,CH22quantitative,CSWZ24,GM13,KSS11,KM08,Pic11,RS09}.
One of the crucial step of the paper is to take the quantum mean-field limit and reduce the $N$-body problem to the one-body nonlinear Schr\"{o}dinger equation
\begin{align*}
i\hbar\pa_{t}\psi_{\hbar,\ve}=&-\frac{1}{2}\hbar^{2}\pa_{x}^{2}\psi_{\hbar,\ve}+(V_{\ve}*|\psi_{\hbar,\ve}|^{2})\psi_{\hbar,\ve}.
\end{align*}
We use directly the result in \cite{BG24} by Ben Porat and Golse, and obtain
\begin{align}
\n{f_{N,\hbar,\ve}^{(1)}(t)-f_{\hbar,\ve}(t)}_{L_{x,\xi}^{2}}\leq 4\sqrt{\frac{1}{N\hbar}}\exp\lrs{\sqrt{\frac{Ct}{\hbar^{3}\ve}}},
\end{align}
where $f_{\hbar,\ve}(t)=W_{\hbar}[\psi_{\hbar,\ve}(t)]$ is the Wigner transform of the one-body wave function $\psi_{\hbar,\ve}(t)$.
With this key observation, it suffices to study the limit problem for $f_{\hbar,\ve}(t)$.

\vskip .1in

\noindent\textbf{Step 2. Weighted uniform energy estimates.}

In Section \ref{sec:Uniform Estimates}, we introduce new weighted uniform estimates
\begin{align}\label{equ:weighted,uniform,intro}
\n{\lra{x}^{\frac{1}{2}}\hbar^{k}\pa_{x}^{k} \psi_{\hbar,\ve}(t)}_{L_{x}^{2}}\leq& C(k,t),\quad \f k\geq 1,
\end{align}
based on which we set up
 \begin{align}\label{equ:weighted,uniform,intro,f}
 \bbn{\lra{x}\hbar^{\al}\pa_{x}^{\al}\int_{\R}\xi^{k}f_{\hbar,\ve}(t,x,\xi) d\xi}_{L_{x}^{1}}\leq& C(k,\al,t),
 \end{align}
 where $\lra{x}=\sqrt{1+x^{2}}$.
  The weighted uniform estimates are new, and 1D specific for our subsequent analysis including the compactness, convergence and the conservation laws for the limit solution. The proof and usage of \eqref{equ:weighted,uniform,intro} and \eqref{equ:weighted,uniform,intro,f}
  are the key.

\vskip .1in

\noindent\textbf{Step 3. Compactness and convergence.}
In Section \ref{section:Compactness and Narrow Convergence}, using the uniform estimates in Section \ref{sec:Uniform Estimates}, we are able to obtain higher moment difference estimates between the Wigner function and the Husimi function which is non-negative, and attain more properties. Then,
we prove the compactness of the sequence $\lr{f_{\hbar,\ve}(t,x,\xi)}$ and justify the weak convergence (up to a subsequence) to a non-negative bounded Radon measure
\begin{align*}
f(t,dx,d\xi)\in C([0,\infty);\mathcal{M}^{+}(\R^{2})-w^{*}),
\end{align*}
in the sense that for $\f T>0$, $k\geq 0$, there hold
\begin{align*}
\lim_{(\hbar,\ve)\to (0,0)}\int_{0}^{T}\iint_{\R^{2}}\lrs{\xi^{k}f_{\hbar,\ve}(t,x,\xi)-\xi^{k}f(t,x,\xi)}\phi dxd\xi dt=0,\ \f \phi\in
L_{t}^{1}([0,T];\mathcal{A}),
\end{align*}
where the test function space $\mathcal{A}$ is defined in \eqref{equ:function space,A}. Our method here enables a direct proof that the limit is non-negative.
Furthermore, with our method, for the convergence of the moment function $\int_{\R}\xi^{k}f_{\hbar,\ve}(t,x,\xi)d\xi$, we are able to prove the narrow convergence due to \eqref{equ:weighted,uniform,intro} and \eqref{equ:weighted,uniform,intro,f}.
That is,
\begin{align*}
\lim_{(\hbar,\ve)\to (0,0)}\int_{0}^{T}\int_{\R}\lrs{\int_{\R}\xi^{k}f_{\hbar,\ve}(t,x,\xi)d\xi-\int_{\R}\xi^{k}f(t,x,\xi)d\xi}\vp dx dt=0,\ \f \vp\in
L_{t}^{1}([0,T];C_{b}(\R)).
\end{align*}
The test functions belong to the space of the bounded continuous functions. This is the key to the conservation laws, as we can take the constant 1 as a test function now.

\vskip .1in
\noindent\textbf{Step 4. Conservation laws for the limit solution.}

In Section \ref{section:Conservation Laws for the Limit Measure}, we prove the conservation laws of mass, momentum, and energy for the limit measure solution as presented in \eqref{equ:conservation,mass,theorem}--\eqref{equ:conservation,energy,theorem}.
The mass, momentum, and kinetic energy parts follow from the narrow convergence.
The difficult one is the convergence of the interaction potential energy, as the narrow convergence we obtain in Step 3 remains too weak to deal with the limit problem for the nonlinear term.
 To circumvent this problem, based on the weighted uniform estimates \eqref{equ:weighted,uniform,intro} and \eqref{equ:weighted,uniform,intro,f}, we introduce a weighted transform to obtain the local strong convergence, which is the key to the convergence of the interaction potential energy.

\vskip .1in
\noindent\textbf{Step 5. Convergence to the Vlasov-Poisson equation}

The most intricate part is to verify the limit to the Vlasov-Poisson equation. In Sections \ref{section:momentum convergence}--\ref{section:Full Convergence to the Vlasov-Poisson Equation}, we follow the scheme in \cite{ZZM02} to prove the moment convergence to the Vlasov-Poisson equation and establish the exponential decay for the limit measure, which is used to obtain the full convergence.
More precisely, in Section \ref{section:momentum convergence}, for the test function $\vp(t,x) \xi^{k}$, we obtain
\begin{equation}\label{equ:moment convergence,intro}
\begin{aligned}
\int_{\Omega_T} \int_{\mathbb{R}}\left(\partial_t \vp+\xi \partial_x \vp\right) \xi^k f(t, d x, d \xi) dt -k \int_{\Omega_T} \vp \ol{E} \lrs{\int_{\mathbb{R}} \xi^{k-1} f(t,dx,d\xi)} d t=0,
\end{aligned}
\end{equation}
where $\Omega_{T}=(0,T)\times \R$ and $\ol{E}$ is the Vol$^\prime$pert's symmetric average defined in \eqref{equ:volpert average}. Then we prove the exponential decay estimate that
\begin{align*}
\iint_{\Omega_{T}}\int_{\R}e^{\delta|\xi|}f(t,dx,d\xi)dt\leq C_{\delta},
\end{align*}
based on which we prove
\begin{align*}
\pa_{t}f+\xi\pa_{x}f-\pa_{\xi}(\ol{E}f)=0,
\end{align*}
in the sense of distributions in Section \ref{section:Full Convergence to the Vlasov-Poisson Equation}. (Notice the difference of test functions in Sections \ref{section:momentum convergence} and \ref{section:Full Convergence to the Vlasov-Poisson Equation}.) Hence, we conclude that the limit measure $f(t,dx,d\xi)$ is a weak solution to the Vlasov-Poisson equation.

During the proof of the convergence, the main difficulties lie in dealing with the vanishing problem of the remainder term
\begin{align*}
\mathrm{R}_{\hbar,\ve}^{(k)}=i\sum_{2\leq \al\leq k}\binom{k}{\al}\frac{\hbar^{\al-1}}{2^{k}}(1-(-1)^{\al})D_{x}^{\al}(V_{\ve}*\rho_{\hbar,\ve})
\int_{\R}\xi^{k-\al}f_{\hbar,\ve}d\xi,
\end{align*}
 and the convergence problem of the nonlinear term
\begin{align*}
 \lrs{\pa_{x}V_{\ve}*\rho_{\hbar,\ve}} \lrs{\int_{\mathbb{R}} \xi^{k-1} f_{\hbar,\ve} d \xi},
\end{align*}
both of which yield a entangled product limit problem in the double semiclassical and Debye length limit. Adding more to the difficulty, the limit point is a measure instead of a locally integrable function, which is also known to be a stubborn technical point.
Nonetheless, using the new weighted uniform estimates \eqref{equ:weighted,uniform,intro}--\eqref{equ:weighted,uniform,intro,f} and an iteration scheme which reduces $\xi^{k+1}$ in \eqref{equ:moment convergence,intro} into $\xi^{k}$ and hence enables an induction, we can fortunately overcome these difficulties.
\vskip .1in

Finally, we put the definition of the weak solution to the
Vlasov-Poisson equation in Appendix \ref{section:The Measure Solution to the Vlasov-Poisson Equation} and include some basic properties of the bounded
variation functions in Appendix \ref{section:Basic Properties of BV Functions}.
Putting together the results in the above steps $1$--$5$, we conclude Theorem \ref{thm:main theorem}.

\section{Preliminary Reduction: Quantum Mean-field Limit}\label{section:Preliminary Reduction: Quantum Mean-field Limit}
In this section, we take the quantum mean-field limit and reduce the quantum $N$-body dynamics to the one-body nonlinear Schr\"{o}dinger equation (NLS)
\begin{equation}\label{equ:NLS,one-body,uniform estimate}
\left\{
\begin{aligned}
i\hbar\pa_{t}\psi_{\hbar,\ve}=&-\frac{1}{2}\hbar^{2}\pa_{x}^{2}\psi_{\hbar,\ve}+(V_{\ve}*|\psi_{\hbar,\ve}|^{2})\psi_{\hbar,\ve},\\
\psi_{\hbar,\ve}(0)=&\psi_{\hbar}^{\mathrm{in}}.
\end{aligned}
\right.
\end{equation}
Certainly, there have been many methods developed to establish a quantitative measurement between the $N$-body systems and the one-body NLS.
Here, to make our limit problem of the quantum $N$-body dynamics concise and clear,
we use directly the result in \cite{BG24}, which is inspired by \cite{Pic11} and gives a convergence rate estimate between
the $N$-body dynamics and the one-body NLS with an explicit $\hbar$-dependence.

\begin{theorem}[{\hspace{-0.05em}\cite[Corollary 4.2]{BG24}}]\label{lemma:quantum mean-field limit}
Let $\Psi_{N,\hbar,\ve}(t)$ be the solution to the $N$-body dynamics \eqref{equ:n-body NLS} with the factorized initial data, $\ga_{N,\hbar,\ve}^{(1)}(t)$ be the first marginal density. Then it holds that
\begin{align}\label{equ:N-body,one-body,convergence}
\operatorname{Tr}\bbabs{\ga_{N,\hbar,\ve}^{(1)}(t,x,x')-\psi_{\hbar,\ve}(t,x) \ol{\psi_{\hbar,\ve}}(t,x')}\leq 4 \sqrt{\frac{1}{N}} \exp \left(\frac{3}{\hbar} \int_0^t L_{\hbar,\ve}(s) d s\right),
\end{align}
where
\begin{align}
L_{\hbar,\ve}(t):=C\n{V_{\ve}}_{L_{x}^{\infty}}\n{\psi_{\hbar,\ve}(t)}_{H^{2}}.
\end{align}
\end{theorem}

Using Theorem \ref{lemma:quantum mean-field limit} and the energy estimate \eqref{equ:higher energy estimates} in Section \ref{sec:Uniform Estimates}, we immediately obtain the following corollary.

\begin{corollary}\label{lemma:convergence rate,N,h}
Let $f_{N,\hbar,\ve}^{(1)}(t)$, $f_{\hbar,\ve}(t)$ be the Wigner transform of $\ga_{N,\hbar,\ve}^{(1)}(t)$, $\psi_{\hbar,\ve}(t)$ respectively. We have
\begin{align}\label{equ:strong convergence,n-body}
\n{f_{N,\hbar,\ve}^{(1)}(t)-f_{\hbar,\ve}(t)}_{L_{x,\xi}^{2}}\leq 4\sqrt{\frac{1}{N\hbar}}\exp\lrs{\sqrt{\frac{Ct}{\hbar^{3}\ve}}}.
\end{align}
\end{corollary}
\begin{proof}
By Plancherel identity and the operator inequality that $\n{A}_{HS}\leq \operatorname{Tr}|A|$, we obtain
\begin{align}
\n{f_{N,\hbar,\ve}^{(1)}(t)-f_{\hbar,\ve}(t)}_{L_{x,\xi}^{2}}=&\n{W_{\hbar}[\ga_{N,\hbar,\ve}^{(1)}(t)]
-W_{\hbar}[\psi_{\hbar,\ve}(t)]}_{L_{x,\xi}^{2}}\\
=&\hbar^{-\frac{1}{2}}\bn{\ga_{N,\hbar,\ve}^{(1)}(t,x,x')-\psi_{\hbar,\ve}(t,x) \ol{\psi_{\hbar,\ve}}(t,x')}_{L_{x,x'}^{2}}\notag\\
\leq& \hbar^{-\frac{1}{2}}\operatorname{Tr}\bbabs{\ga_{N,\hbar,\ve}^{(1)}(t,x,x')-\psi_{\hbar,\ve}(t,x) \ol{\psi_{\hbar,\ve}}(t,x')}.\notag
\end{align}
Using \eqref{equ:N-body,one-body,convergence}, energy estimate \eqref{equ:higher energy estimates}, and $\n{V_{\ve}}_{L_{x}^{\infty}}\leq \ve^{-1}$, we arrive at \eqref{equ:strong convergence,n-body}.
\end{proof}
The convergence rate estimate \eqref{equ:strong convergence,n-body} is enough for the limit problem up to a subsequence. Therefore,
in the follow Sections \ref{sec:Uniform Estimates}--\ref{section:Full Convergence to the Vlasov-Poisson Equation}, we start from the nonlinear Schr\"{o}dinger equation \eqref{equ:NLS,one-body,uniform estimate} and justify its limit to the Vlasov-Poisson equation.
\section{Weighted Uniform Higher Energy Estimates}\label{sec:Uniform Estimates}
In this section, we set up the weighted uniform higher energy estimates on the one-body wave function $\psi_{\hbar,\ve}(t)$ of the nonlinear
Schr\"{o}dinger equation \eqref{equ:NLS,one-body,uniform estimate}.
Then using the weighted uniform estimates, we provide the higher derivative and weighted uniform estimates for the higher moments of $f_{\hbar,\ve}(t,x,\xi)$ in Lemma \ref{lemma:uniform estimate,momentum of wigner}.

\begin{lemma}[Weighted uniform estimates]
Let $\rho_{\hbar,\ve}(t)=|\psi_{\hbar,\ve}(t)|^{2}$. We have
\begin{align}
\n{\lra{x}^{2}\rho_{\hbar,\ve}(t)}_{L_{x}^{1}}\leq& C(t),\label{equ:weighted estimate,mass}\\
\n{\lra{x}^{\frac{1}{2}}\hbar^{k}\pa_{x}^{k} \psi_{\hbar,\ve}(t)}_{L_{x}^{2}}\leq& C(k,t),\quad \f k\geq 1, \label{equ:weighted estimate,higher}
\end{align}
where $\lra{x}=\sqrt{1+x^{2}}$.
\end{lemma}
\begin{proof}
Estimate \eqref{equ:weighted estimate,mass} is usually called a virial estimate, while estimate \eqref{equ:weighted estimate,higher} is a new weighted uniform estimate, which might not be true for the higher dimension case.

Before proving the weighted uniform estimates \eqref{equ:weighted estimate,mass}--\eqref{equ:weighted estimate,higher}, we set up the higher energy estimates
\begin{align}\label{equ:higher energy estimates}
\bn{\hbar^{k} \pa_{x}^{k}\psi_{\hbar,\ve}(t)}_{L_{x}^{2}}\leq C(k,t).
\end{align}

For $k=0$, estimate \eqref{equ:higher energy estimates} just follows from the mass conservation law of \eqref{equ:NLS,one-body,uniform estimate}. For the defousing case,
we can also obtain \eqref{equ:higher energy estimates} with $k=1$ by using the energy conservation law of \eqref{equ:NLS,one-body,uniform estimate}, as the potential energy is positive. However, such an argument is not valid for the focusing case. To provide a unified proof, we take the induction argument to deal with the general case $k\geq 1$ for both defocusing and focusing cases.

We assume that \eqref{equ:higher energy estimates} holds for $n\leq k-1$, and we prove it for $n=k$.
Using the nonlinear Schr\"{o}dinger equation \eqref{equ:NLS,one-body,uniform estimate}, we obtain
\begin{align*}
&\frac{1}{2}\frac{d}{dt}\n{\hbar^{k} \pa_{x}^{k}\psi_{\hbar,\ve}}_{L_{x}^{2}}^{2} \\
=&\operatorname{Re}\int \ol{\pa_{t}\hbar^{k}\pa_{x}^{k}\psi_{\hbar,\ve}}  \hbar^{k}\pa_{x}^{k}\psi_{\hbar,\ve}dx\\
=&-\operatorname{Im}\int \ol{\hbar^{k-1}\pa_{x}^{k}\lrc{(V_{\ve}*\rho_{\hbar,\ve})\psi_{\hbar,\ve}}}  \hbar^{k}\pa_{x}^{k}\psi_{\hbar,\ve}dx\\
=&-\operatorname{Im}\int \ol{(\hbar^{k-1}\pa_{x}^{k}\lrc{(V_{\ve}*\rho_{\hbar,\ve})\psi_{\hbar,\ve}}-(V_{\ve}*\rho_{\hbar,\ve})\hbar^{k-1}\pa_{x}^{k}\psi_{\hbar,\ve})}  \hbar^{k}\pa_{x}^{k}\psi_{\hbar,\ve}dx.
\end{align*}
Then by H\"{o}lder's inequality, Leibniz rule, and Young's inequality, we get
\begin{align}
&\frac{1}{2}\frac{d}{dt}\n{\hbar^{k} \pa_{x}^{k}\psi_{\hbar,\ve}}_{L_{x}^{2}}^{2}\label{equ:higher energy,dt}\\
\leq& \bn{\hbar^{k-1}\pa_{x}^{k}\lrc{(V_{\ve}*\rho_{\hbar,\ve})\psi_{\hbar,\ve}}-(V_{\ve}*\rho_{\hbar,\ve})\hbar^{k-1}\pa_{x}^{k}\psi_{\hbar,\ve}}_{L_{x}^{2}}
\n{\hbar^{k}\pa_{x}^{k}\psi_{\hbar,\ve}}_{L_{x}^{2}}\notag\\
\leq&\sum_{j=1}^{k}\binom{k}{j}\n{\hbar^{j-1}\pa_{x}^{j}V_{\ve}*\rho_{\hbar,\ve}}_{L_{x}^{\infty}}
\n{\hbar^{k-j} \pa_{x}^{k-j}\psi_{\hbar,\ve}}_{L_{x}^{2}} \n{\hbar^{k}\pa_{x}^{k}\psi_{\hbar,\ve}}_{L_{x}^{2}}\notag\\
\leq&\sum_{j=1}^{k}\binom{k}{j}\n{\pa_{x}V_{\ve}}_{L_{x}^{\infty}}\hbar^{j-1}\n{\pa_{x}^{j-1}\rho_{\hbar,\ve}}_{L_{x}^{1}}
\n{\hbar^{k-j} \pa_{x}^{k-j}\psi_{\hbar,\ve}}_{L_{x}^{2}} \n{\hbar^{k}\pa_{x}^{k}\psi_{\hbar,\ve}}_{L_{x}^{2}}\notag\\
\leq&C\sum_{j=1}^{k}\binom{k}{j}\hbar^{j-1}\n{\pa_{x}^{j-1}\rho_{\hbar,\ve}}_{L_{x}^{1}}
\n{\hbar^{k-j} \pa_{x}^{k-j}\psi_{\hbar,\ve}}_{L_{x}^{2}} \n{\hbar^{k}\pa_{x}^{k}\psi_{\hbar,\ve}}_{L_{x}^{2}},\notag
\end{align}
where in the last inequality we have used
$\n{\pa_{x}V_{\ve}}_{L_{x}^{\infty}}\leq C$.
Using again Leibniz rule and H\"{o}lder's inequality, we have
\begin{align}
\hbar^{j-1}\n{\pa_{x}^{j-1}\rho_{\hbar,\ve}}_{L_{x}^{2}}\leq& \sum_{j_{1}+j_{2}=j-1}\binom{j-1}{j_{1}}\n{\hbar^{j_{1}}\pa_{x}^{j_{1}}\psi_{\hbar,\ve}}_{L_{x}^{2}}
\n{\hbar^{j_{2}}\pa_{x}^{j_{2}}\psi_{\hbar,\ve}}_{L_{x}^{2}}.\label{equ:higher energy,density}
\end{align}
Plugging \eqref{equ:higher energy,density} into \eqref{equ:higher energy,dt}, we use \eqref{equ:higher energy estimates} for the case $n<k$ to obtain
\begin{align}\label{equ:higher energy,dt,inequality}
\frac{d}{dt}\n{\hbar^{k} \pa_{x}^{k}\psi_{\hbar,\ve}}_{L_{x}^{2}}^{2}\leq C(k,t).
\end{align}
Noticing that the initial datum satisfies
\begin{align*}
\n{\hbar^{k} \pa_{x}^{k}\psi_{\hbar,\ve}(0)}_{L_{x}^{2}}\leq C^{k}k^{k},
\end{align*}
by \eqref{equ:higher energy,dt,inequality} we arrive at
\begin{align*}
\n{\hbar^{k} \pa_{x}^{k}\psi_{\hbar,\ve}(t)}_{L_{x}^{2}}\leq C(k,t),
\end{align*}
which completes the proof of \eqref{equ:higher energy estimates}.

Now, we get into the proof of estimate \eqref{equ:weighted estimate,mass}.
For $t\in[0,T]$, we have
\begin{align*}
&\frac{d}{dt}\int_{\R} e^{-2\delta |x|^{2}}\lra{x}^{2}\rho_{\hbar,\ve}(t,x)dx\\
=&-\int_{\R} e^{-2\delta|x|^{2}}\lra{x}^{2}\pa_{x}\lrs{\operatorname{Im}\lrs{\ol{\psi_{\hbar,\ve}}\hbar \pa_{x}\psi_{\hbar,\ve}}}dx\\
=&\int_{\R} e^{-2\delta|x|^{2}}2x(1-2\delta\lra{x}^{2}) \operatorname{Im}\lrs{\ol{\psi_{\hbar,\ve}}\hbar \pa_{x}\psi_{\hbar,\ve}} dx\\
\lesssim& \n{e^{-\delta|x|^{2}}(\delta\lra{x}^{2}+1)}_{L_{x}^{\infty}}
\n{e^{-\delta|x|^{2}}|x|\psi_{\hbar,\ve}(t,x)}_{L_{x}^{2}}\n{\hbar\pa_{x}\psi_{\hbar,\ve}(t,x)}_{L_{x}^{2}}\\
\lesssim&C(T)\int_{\R} e^{-2\delta|x|^{2}}\lra{x}^{2}\rho_{\hbar,\ve}(t,x)dx,
\end{align*}
where in the last inequality we used the energy estimate \eqref{equ:higher energy estimates} with $k=1$.
By Gronwall's inequality, we have
\begin{align*}
\n{e^{-2\delta|x|^{2}}\lra{x}^{2}\rho_{\hbar,\ve}(t,x)}_{L_{x}^{1}}\leq C(t).
\end{align*}
Letting $\delta\to 0$ and using Fatou's lemma, we arrive at \eqref{equ:weighted estimate,mass}.

Next, we can prove the weighted uniform estimate \eqref{equ:weighted estimate,higher}.
 By integration by parts, H\"{o}lder's inequality, virial estimate \eqref{equ:weighted estimate,mass}, and the higher energy estimates \eqref{equ:higher energy estimates}, we have
\begin{align*}
&\bbn{\lrs{\chi(\frac{x}{R})\lra{x}}^{\frac{1}{2}}\hbar^{k}\pa_{x}^{k} \psi_{\hbar,\ve}}_{L_{x}^{2}}^{2}\\
=&\hbar^{2k}\bbabs{\int_{\R} \chi(\frac{x}{R})\lra{x}\pa_{x}^{k}\psi_{\hbar,\ve}\ol{\pa_{x}^{k} \psi_{\hbar,\ve}}dx}\\
\leq &\hbar^{2k}\sum_{\al=0}^{k}\binom{k}{\al}\bbabs{\int_{\R} \psi_{\hbar,\ve} \pa_{x}^{\al}\lrs{\chi(\frac{x}{R})\lra{x}}\ol{\pa_{x}^{2k-\al} \psi_{\hbar,\ve}}dx}\\
\leq&  \n{\lra{x}\psi_{\hbar,\ve}}_{L_{x}^{2}} \n{\chi(\frac{x}{R})}_{L_{x}^{\infty}}\n{\hbar^{2k}\pa_{x}^{2k}\psi_{\hbar,\ve}}_{L_{x}^{2}}\\
&+\sum_{\al=1}^{k}\binom{k}{\al}\hbar^{\al}\n{\psi_{\hbar,\ve}}_{L_{x}^{2}}
\bbn{\pa_{x}^{\al}\lrs{\chi(\frac{x}{R})\lra{x}}}_{L_{x}^{\infty}}
\n{\hbar^{2k-\al}\pa_{x}^{2k-\al}\psi_{\hbar,\ve}}_{L_{x}^{2}}\\
\leq & C(k,t),
\end{align*}
where in the last inequality we have used that
\begin{align*}
\hbar^{\al}\leq 1, \quad \bbn{\pa_{x}^{\al}\lrs{\chi(\frac{x}{R})\lra{x}}}_{L_{x}^{\infty}}
\lesssim 1,\quad \f \al\geq 1.
\end{align*}
Sending $R\to \infty$ and using Fatou's lemma, we arrive at \eqref{equ:weighted estimate,higher}.
\end{proof}

Now, we are able to provide the higher derivative and weighted uniform estimates for the higher moments of $f_{\hbar,\ve}(t,x,\xi)$.
For simplicity, we define the Wigner function
\begin{align}
W_{\hbar}[u_{1},u_{2}]:=\frac{1}{2\pi}\int_{\R}e^{-iz\xi}u_{1}(x+\frac{\hbar z}{2})\ol{u_{2}(x-\frac{\hbar z}{2})}dz,
\end{align}
and use the shorthand $W_{\hbar}[u]$ to denote $W_{\hbar}[u,u]$.
Next, we give two formulas of the Wigner function.
Let $D_{x}:=\frac{1}{i}\pa_{x}$.
\begin{itemize}
\item The Wigner function with a weight function $\xi^{k}$ satisfies
\begin{align}
\xi^{k}W_{\hbar}[u_{1},u_{2}]=&\frac{1}{2\pi}\frac{\hbar^{k}}{2^{k}}\sum_{\al=0}^{k}\binom{k}{\al}(-1)^{k-\al}
\int e^{-iz\xi}D_{x}^{\al}u_{1}(x+\frac{\hbar z}{2})\ol{D_{x}^{k-\al}u_{2}(x-\frac{\hbar z}{2})}dz\label{equ:momentum,wigner,form}\\
=&\frac{1}{2\pi}\frac{\hbar^{k}}{2^{k}}\sum_{\al=0}^{k}\binom{k}{\al}(-1)^{k-\al} W_{\hbar}[D^{\al}u_{1},D^{k-\al}u_{2}].\notag
\end{align}
\item The integral of the Wigner function with a weight function $\xi^{k}$ satisfies
\begin{align}\label{equ:momentum,wigner,integral form}
\int_{\R} \xi^{k}W_{\hbar}[u_{1},u_{2}]d\xi=\frac{\hbar^{k}}{2^{k}}\sum_{\al=0}^{k}\binom{k}{\al}(-1)^{k-\al} D_{x}^{\al}u_{1}\ol{D_{x}^{k-\al}u_{2}}.
\end{align}
\end{itemize}

\begin{lemma}\label{lemma:uniform estimate,momentum of wigner}
For $t\in [0,T]$, we have
\begin{align}
\bbn{\lra{x}\hbar^{\al+1}\pa_{x}^{\al}\int_{\R}\xi^{k}f_{\hbar,\ve} d\xi}_{L_{x}^{\infty}}\leq& C(k,\al,t),\quad \f k\geq0,\ \al\geq 0, \label{equ:higer,momentum of wigner,Lmax}\\
\bbn{\lra{x}\hbar^{\al}\pa_{x}^{\al}\int_{\R}\xi^{k}f_{\hbar,\ve} d\xi}_{L_{x}^{1}}\leq& C(k,\al,t),\quad \f k\geq 0,\ \al\geq 0, \label{equ:weighted,momentum of wigner}\\
\n{\hbar^{k}\pa_{x}^{k}E_{\hbar,\ve}}_{L_{x}^{\infty}}\leq& C(k,t),\quad \f k\geq 0, \label{equ:weighted,eletron part}
\end{align}
where $E_{\hbar,\ve}=\pa_{x}V_{\ve}*\rho_{\hbar,\ve}$.
\end{lemma}
\begin{proof}
For \eqref{equ:higer,momentum of wigner,Lmax}, by formula \eqref{equ:momentum,wigner,integral form}, H\"{o}lder's inequality, Leibniz rule, we have
\begin{align*}
&\bbn{\lra{x}\hbar^{\al+1}\pa_{x}^{\al}\int_{\R}\xi^{k}f_{\hbar,\ve} d\xi}_{L_{x}^{\infty}}\\
=&
\bbn{\lra{x}\hbar^{\al+1}\pa_{x}^{\al}\int_{\R}\xi^{k}W_{\hbar}[\psi_{\hbar,\ve}]d\xi}_{L_{x}^{\infty}}\\
\leq&\sum_{k_{1}+k_{2}=k}\binom{k}{k_{1}}\bbn{\lra{x}\hbar^{\al+1}\pa_{x}^{\al}
\lrs{\hbar^{k_{1}}\pa_{x}^{k_{1}}\psi_{\hbar,\ve}\ol{\hbar^{k_{2}}\pa_{x}^{k_{2}}
\psi_{\hbar,\ve}}}}_{L_{x}^{\infty}}\\
\leq&\sum_{k_{1}+k_{2}=k}\sum_{\al_{1}+\al_{2}=\al}\binom{k}{k_{1}}\binom{\al}{\al_{1}}
\n{\lra{x}^{\frac{1}{2}}\hbar^{\al_{1}+k_{1}+\frac{1}{2}}\pa_{x}^{\al_{1}+k_{1}}\psi_{\hbar,\ve}}_{L_{x}^{\infty}}
\n{\lra{x}^{\frac{1}{2}}\hbar^{\al_{2}+k_{2}+\frac{1}{2}}\pa_{x}^{\al_{2}+k_{2}}\psi_{\hbar,\ve}}_{L_{x}^{\infty}}\\
\leq& C(k,\al,t),
\end{align*}
where in the last inequality we have used the weighted uniform estimates \eqref{equ:weighted estimate,higher} and the interpolation inequality that
\begin{align*}
\n{\lra{x}^{\frac{1}{2}}\hbar^{j+\frac{1}{2}}\pa_{x}^{j}\psi_{\hbar,\ve}}_{L_{x}^{\infty}}\leq \bbn{\hbar\pa_{x}\lrs{\lra{x}^{\frac{1}{2}}\hbar^{j}\pa_{x}^{j}\psi_{\hbar,\ve}}}_{L_{x}^{2}}^{\frac{1}{2}}
\n{\lra{x}^{\frac{1}{2}}\hbar^{j}\pa_{x}^{j}\psi_{\hbar,\ve}}_{L_{x}^{2}}^{\frac{1}{2}}\leq C(j,t).
\end{align*}
For \eqref{equ:weighted,momentum of wigner}, similarly we have
\begin{align*}
&\bbn{\lra{x}\hbar^{\al}\pa_{x}^{\al}\int_{\R}\xi^{k}f_{\hbar,\ve} d\xi}_{L_{x}^{1}}\\
\leq&\sum_{k_{1}+k_{2}=k}\sum_{\al_{1}+\al_{2}=\al}\binom{k}{k_{1}}\binom{\al}{\al_{1}}
\n{\lra{x}^{\frac{1}{2}}\hbar^{\al_{1}+k_{1}}\pa_{x}^{\al_{1}+k_{1}}\psi_{\hbar,\ve}}_{L_{x}^{2}}
\n{\lra{x}^{\frac{1}{2}}\hbar^{\al_{2}+k_{2}}\pa_{x}^{\al_{2}+k_{2}}\psi_{\hbar,\ve}}_{L_{x}^{2}}\\
\leq& C(k,\al,t).
\end{align*}
For \eqref{equ:weighted,eletron part}, by Young's inequality and \eqref{equ:weighted,momentum of wigner}, we get
\begin{align*}
\n{\hbar^{k}\pa_{x}^{k}E_{\hbar,\ve}}_{L_{x}^{\infty}}=\bn{\pa_{x}V_{\ve}*\lrs{\hbar^{k}\pa_{x}^{k} \rho_{\hbar,\ve}}}_{L_{x}^{\infty}}
\leq \n{\pa_{x}V_{\ve}}_{L_{x}^{\infty}}\n{\hbar^{k}\pa_{x}^{k}\rho_{\hbar,\ve}}_{L_{x}^{1}}\leq C(k,t).
\end{align*}
Hence, we have completed the proof of \eqref{equ:higer,momentum of wigner,Lmax}--\eqref{equ:weighted,eletron part}.

\end{proof}

\section{Compactness and Narrow Convergence}\label{section:Compactness and Narrow Convergence}

In this section, with respect to the weak$^{*}$ topology of the dual space of $\mathcal{A}$ defined in \eqref{equ:function space,A}, we prove the compactness of the sequence $\lr{f_{\hbar,\ve}(t,x,\xi)}$ justify a weak convergence to a non-negative Radon measure $f(t,dx,d\xi)$.
In general, the Wigner transform of the wave function is only a real-valued function and may change sign. To ensure the non-negativity of the limit measure, we need to use the Husimi transform of the wave function, which fixes a non-negative sign.

In Section \ref{sec:Husimi Transform}, we estimate the higher moment differences between the Wigner function and Husimi function, which is used to show that they have the same convergence and limit. Then in Section \ref{sec:Convergence to the Non-negative Radon Measure}, it suffices to prove the convergence of the Husimi function to a non-negative measure, the proof of which relies on the weighted uniform estimates established in Section \ref{sec:Uniform Estimates}.

More specifically, for the convergence of the sequence $\lr{f_{\hbar,\ve}(t,x,\xi)}$, we use the test function space introduced in \cite{LP93}
\begin{align}\label{equ:function space,A}
\mathcal{A}=\left\{\phi \in C_c^{\infty}(\R^{2}):\left(\mathcal{F}_{\xi} \phi\right)(x, \eta) \in L^1\left(\mathbb{R}_\eta, C_c\left(\mathbb{R}_x\right)\right)\right\},
\end{align}
equipped with the norm
$$
\|\phi(x, \xi)\|_{\mathcal{A}}=\int_{\mathbb{R}} \sup _x\left|\left(\mathcal{F}_{\xi} \phi\right)(x, \eta)\right| d \eta,
$$
where $\left(\mathcal{F}_{\xi} \phi\right)(x, \eta)$ is the Fourier transform of $\phi(x, \xi)$ with respect to $\xi$.

Furthermore,
for the convergence of the moment function $\int_{\R}\xi^{k}f_{\hbar,\ve}(t,x,\xi)d\xi$, we are able to prove the narrow convergence.
That is, the test functions belong to the space of the bounded continuous functions, which we denote by $C_{b}(\R)$. The stronger narrow convergence is the key to the conservation laws for the limit measure in Section \ref{section:Conservation Laws for the Limit Measure} and the moment convergence to the Vlasov-Poisson equation in Section \ref{section:momentum convergence}.

\subsection{Higher Moment Estimates between the Wigner and Husimi function}\label{sec:Husimi Transform}
We first define the Husimi transform. For more details, see for instance \cite{LP93,Zha08}.
\begin{definition}
Given $u\in L^{2}$, the Husimi transform of $u$ is defined by
\begin{align}
\wt{W}_{\hbar}\lrc{u}=W_{\hbar}\lrc{u}*_{(x,\xi)}G_{\hbar},
\end{align}
with
\begin{align*}
G_{\hbar}(x,\xi)=(\pi \hbar)^{-1}e^{-\frac{|x|^{2}}{\hbar}}e^{-\frac{|\xi|^{2}}{\hbar}}:=g_{\hbar}(x)g_{\hbar}(\xi).
\end{align*}
\end{definition}
An important property of the Husimi function is the non-negativity, that is, $\wt{W}_{\hbar}\lrc{u}\geq 0$.
To make use of the non-negativity of the Husimi function, we need to provide the higher moment estimates between the Wigner function and Husimi function as follows.
\begin{lemma}\label{lemma:wigner,husimi}
For $\phi(x,\xi)\in C_{c}^{\infty}(\R^{2})$, there holds that
\begin{align}\label{equ:wigner,L2}
\bbabs{\iint_{\R^{2}}\xi^{k}W_{\hbar}[u_{1},u_{2}] \phi dxd\xi}\leq  \sum_{\al=0}^{k}\binom{k}{\al} \n{\hbar^{\al}\pa_{x}^{\al}u_{1}}_{L_{x}^{2}}
\n{\hbar^{k-\al}\pa_{x}^{k-\al}u_{2}}_{L_{x}^{2}}\n{\phi}_{\mathcal{A}},
\end{align}
In particular, we obtain
\begin{align}\label{equ:wigner,f,h,estimate}
\bbabs{\iint_{\R^{2}}\xi^{k}f_{\hbar,\ve}(t,x,\xi) \phi dxd\xi}\leq  C(k,t)\n{\phi}_{\mathcal{A}}.
\end{align}
Moreover, for $\phi(x,\xi)\in C_{c}^{\infty}(\R^{2})$ and $\vp(x)\in C_{c}^{\infty}(\R)$, we have the estimates on the difference between the Wigner function and Husimi function that
\begin{align}\label{equ:wigner,husimi}
&\bbabs{\iint_{\R^{2}}\lrs{\xi^{k}\wt{W}_{\hbar}[u]-\xi^{k}W_{\hbar}[u]} \phi dxd\xi}\\
\leq& \n{u}_{L_{x}^{2}}^{2}\n{\phi*_{(x,\xi)}G_{\hbar}-\phi}_{\mathcal{A}}+C(k)\sum_{\al=0}^{k-1}
\sum_{\al_{1}+\al_{2}=\al}\hbar^{k-\al} \n{\hbar^{\al_{1}}\pa_{x}^{\al_{1}}u}_{L_{x}^{2}}
\n{\hbar^{\al_{2}}\pa_{x}^{\al_{2}}u}_{L_{x}^{2}} \n{\phi}_{\mathcal{A}},\notag
\end{align}
and
\begin{align}\label{equ:wigner,husimi,integral}
&\int_{\R} \lrs{\int_{\R} \xi^{k}\wt{W}_{\hbar}[u]d\xi-\int_{\R} \xi^{k}W_{\hbar}[u]d\xi} \vp dx\\
\leq&
\bbn{\int_{\R}\xi^{k}W_{\hbar}[u]d\xi}_{L_{x}^{1}}\n{\vp*_{x}g_{\hbar}-\vp}_{L_{x}^{\infty}}+C(k)\sum_{\al=0}^{k-1}\hbar^{k-\al}
\bbn{\int_{\R}\xi^{\al}W_{\hbar}[u]d\xi}_{L_{x}^{1}}\n{\vp}_{L_{x}^{\infty}}.\notag
\end{align}

\end{lemma}
\begin{proof}
For \eqref{equ:wigner,L2} with $k=0$, we have
\begin{align*}
&\bbabs{\iint_{\R^{2}}W_{\hbar}[u_{1},u_{2}] \phi(x,\xi) dxd\xi}\\
\leq& \lrs{\int_{\R}\sup_{x}|\mathcal{F}_{\xi}\phi(x,y)|dy} \lrs{\sup_{y}\int_{\R}\bbabs{u_{1}(x+\frac{\hbar y}{2})\ol{u_{2}
(x-\frac{\hbar y}{2})}}dx}\\
\leq&\n{u_{1}}_{L_{x}^{2}}\n{u_{2}}_{L_{x}^{2}}\n{\phi}_{\mathcal{A}}.
\end{align*}
For the $k\geq 1$ case, by formula \eqref{equ:momentum,wigner,form}, we obtain
\begin{align*}
\bbabs{\iint_{\R^{2}}\xi^{k}W_{\hbar}[u_{1},u_{2}] \phi dxd\xi}\leq \sum_{\al=0}^{k}\binom{k}{\al} \n{\hbar^{\al}\pa_{x}^{\al}u_{1}}_{L_{x}^{2}}
\n{\hbar^{k-\al}\pa_{x}^{k-\al}u_{2}}_{L_{x}^{2}}\n{\phi}_{\mathcal{A}},
\end{align*}
which completes the proof of \eqref{equ:wigner,L2}.
Then \eqref{equ:wigner,f,h,estimate} follows from \eqref{equ:wigner,L2} and the uniform estimate \eqref{equ:higher energy estimates}.

For \eqref{equ:wigner,husimi}, we notice that
\begin{align}\label{equ:husimi,momentum}
\xi^{k}\wt{W}_{\hbar}[u]=\xi^{k}(W_{\hbar}[u]*_{(x,\xi)}G_{\hbar})=\sum_{\al=0}^{k}\binom{k}{\al}(\xi^{\al}
W_{\hbar}[u])*_{(x,\xi)}(\xi^{k-\al}G_{\hbar}),
\end{align}
and hence get
\begin{align*}
&\xi^{k}\wt{W}_{\hbar}[u]-\xi^{k}W_{\hbar}[u]\\
=&(\xi^{k}W_{\hbar}[u])*_{(x,\xi)}G_{\hbar}-\xi^{k}W_{\hbar}[u]+
\sum_{\al=0}^{k-1}\binom{k}{\al}(\xi^{\al}
W_{\hbar}[u])*_{(x,\xi)}(\xi^{k-\al}G_{\hbar}).
\end{align*}
Using \eqref{equ:wigner,L2}, we obtain
\begin{align*}
&\bbabs{\iint_{\R^{2}}\lrs{\xi^{k}\wt{W}_{\hbar}[u]-\xi^{k}W_{\hbar}[u]} \phi dxd\xi}\\
\leq&\bbabs{\iint_{\R^{2}}\xi^{k}W_{\hbar}[u] \lrs{\phi*_{(x,\xi)}G_{\hbar}-\phi}dxd\xi}+\sum_{\al=0}^{k-1}\binom{k}{\al}\bbabs{\iint_{\R^{2}}  (\xi^{\al}
W_{\hbar}[u])*(\xi^{k-\al}G_{\hbar}) \phi dxd\xi}\\
\leq&\n{u}_{L_{x}^{2}}^{2}\n{\phi*_{(x,\xi)}G_{\hbar}-\phi}_{\mathcal{A}}\\
&+\sum_{\al=0}^{k-1}
\sum_{\al_{1}+\al_{2}=\al}\binom{k}{\al}\binom{\al}{\al_{1}}\n{\hbar^{\al_{1}}\pa_{x}^{\al_{1}}u}_{L_{x}^{2}}
\n{\hbar^{\al_{2}}\pa_{x}^{\al_{2}}u}_{L_{x}^{2}}
\n{(\xi^{k-\al}G_{\hbar})*_{(x,\xi)}\phi}_{\mathcal{A}}\\
\leq&  \n{u}_{L_{x}^{2}}^{2}\n{\phi*_{(x,\xi)}G_{\hbar}-\phi}_{\mathcal{A}}+C(k)\sum_{\al=0}^{k-1}
\sum_{\al_{1}+\al_{2}=\al}\n{\hbar^{\al_{1}}\pa_{x}^{\al_{1}}u}_{L_{x}^{2}}
\n{\hbar^{\al_{2}}\pa_{x}^{\al_{2}}u}_{L_{x}^{2}}\n{\phi}_{\mathcal{A}}\n{\xi^{k-\al}G_{\hbar}}_{L_{x}^{1}L_{\xi}^{1}}\\
\lesssim&\n{u}_{L_{x}^{2}}^{2}\n{\phi*_{(x,\xi)}G_{\hbar}-\phi}_{\mathcal{A}}+C(k)\sum_{\al=0}^{k-1}
\sum_{\al_{1}+\al_{2}=\al}\n{\hbar^{\al_{1}}\pa_{x}^{\al_{1}}u}_{L_{x}^{2}}
\n{\hbar^{\al_{2}}\pa_{x}^{\al_{2}}u}_{L_{x}^{2}}\hbar^{k-\al} \n{\phi}_{\mathcal{A}},
\end{align*}
where in the last two inequalities we have used that
\begin{align*}
\n{g*_{(x,\xi)}\phi}_{\mathcal{A}}\leq& \n{(\mathcal{F}_{\xi}g)*_{x}(\mathcal{F}_{\xi}\phi)}_{L_{\xi}^{1}L_{x}^{\infty}}\leq \n{\mathcal{F}_{\xi}g}_{L_{\eta}^{\infty}L_{x}^{1}} \n{\mathcal{F}_{\xi}\phi}_{L_{\eta}^{1}L_{x}^{\infty}}
\leq \n{g}_{L_{x}^{1}L_{\xi}^{1}}\n{\phi}_{\mathcal{A}},\\
\n{\xi^{k-\al}G_{\hbar}}_{L_{x}^{1}L_{\xi}^{1}}\lesssim & \hbar^{k-\al}.
\end{align*}
Therefore, we complete the proof of \eqref{equ:wigner,husimi}.
Estimate \eqref{equ:wigner,husimi,integral} follows from a way in which we obtain \eqref{equ:wigner,husimi}.

\end{proof}

\subsection{Convergence to a Non-negative Radon Measure}\label{sec:Convergence to the Non-negative Radon Measure}
 As we have established the difference estimates between the Wigner function and the Husimi function, which shows that they have same limit, we can use the non-negativity of the Husimi function to conclude the convergence of $\lr{f_{\hbar,\ve}(t,x,\xi)}$ to a non-negative Radon measure $f(t,dx,d\xi)$.

\begin{notation}
Here, for the convenience, we also use the notation $f(t,x,\xi)dx d\xi$ to denote the measure $f(t,dx,d\xi)$. Hence, one should keep in mind that $f(t,x,\xi)$ is not an $L_{loc}^{1}$ function.
\end{notation}

\begin{lemma}\label{lemma:convergence}
There exists a subsequence of $\lr{f_{\hbar,\ve}(t,x,\xi)}$, which we still denote by $\lr{f_{\hbar,\ve}(t,x,\xi)}$, and a bounded non-negative Radon measure
\begin{align}\label{equ:convergence,limit point,measure}
f(t,dx,d\xi)\in C([0,\infty);\mathcal{M}^{+}(\R^{2})-w^{*}),
\end{align}
such that for $\f T>0$, $k\geq 0$, there hold
\begin{align}\label{equ:convergence,f,h}
\lim_{(\hbar,\ve)\to (0,0)}\int_{0}^{T}\iint_{\R^{2}}\lrs{\xi^{k}f_{\hbar,\ve}(t,x,\xi)-\xi^{k}f(t,x,\xi)}\phi dxd\xi dt=0,\ \f \phi\in
L_{t}^{1}([0,T];\mathcal{A}),
\end{align}
and
\begin{align}\label{equ:convergence,momentum,f,h}
\lim_{(\hbar,\ve)\to (0,0)}\int_{0}^{T}\int_{\R}\lrs{\int_{\R}\xi^{k}f_{\hbar,\ve}(t,x,\xi)d\xi-\int_{\R}\xi^{k}f(t,x,\xi)d\xi}\vp dx dt=0,\ \f \vp\in
L_{t}^{1}([0,T];C_{b}(\R)),
\end{align}
and
\begin{align}\label{equ:convergence,momentum,f,h,pointwise}
\lim_{(\hbar,\ve)\to (0,0)}\int_{\R}\lrs{\int_{\R}\xi^{k}f_{\hbar,\ve}(t,x,\xi)d\xi-\int_{\R}\xi^{k}f(t,x,\xi)d\xi}\vp dx=0,\ \f t\geq 0,\ \vp\in
C_{b}(\R).
\end{align}

Moreover, the limit measure satisfies the weighted estimates that
\begin{align}
\iint_{\R^{2}}\lra{x}^{2}f(t,dx,d\xi)\leq& C(t),\label{equ:measure,weighted,density}\\
\iint_{\R^{2}} \lra{x}|\xi|^{k}f(t,dx,d\xi)\leq& C(k,t).\label{equ:measure,weighted,momentum}
\end{align}
\end{lemma}
\begin{proof}
As we need to prove the non-negativity of the limit point, we take the Husimi transform of the wave function $\psi_{\hbar,\ve}$, which
we denote by $\wt{f}_{\hbar,\ve}(t,x,\xi):=\wt{W}_{\hbar}[\psi_{\hbar,\ve}]\geq 0$.

We first prove \eqref{equ:convergence,limit point,measure} and \eqref{equ:convergence,f,h}. The proof can be divided into the following three steps.

\textbf{Step 1. $L_{x,\xi}^{1}$ Uniform Bounds.}

By estimate \eqref{equ:wigner,husimi} on the Wigner function and the Husimi function in Lemma \ref{lemma:convergence} and the uniform estimate \eqref{equ:higher energy estimates}, we obtain
\begin{align*}
&\bbabs{\int_{0}^{T}\iint_{\R^{2}}\lrs{\xi^{k}f_{\hbar,\ve}(t,x,\xi)-\xi^{k}\wt{f}_{\hbar,\ve}(t,x,\xi)}\phi dxd\xi dt}\\
\leq& \n{\psi_{\hbar,\ve}}_{L_{t}^{\infty}L_{x}^{2}}^{2}\n{\phi*_{(x,\xi)}G_{\hbar}-\phi}_{L_{t}^{1}\mathcal{A}}\\
&+C(k)\sum_{\al=0}^{k-1}
\sum_{\al_{1}+\al_{2}=\al}\hbar^{k-\al} \n{\hbar^{\al_{1}}\pa_{x}^{\al_{1}}\psi_{\hbar,\ve}}_{L_{t}^{\infty}L_{x}^{2}}
\n{\hbar^{\al_{2}}\pa_{x}^{\al_{2}}\psi_{\hbar,\ve}}_{L_{t}^{\infty}L_{x}^{2}} \n{\phi}_{L_{t}^{1}\mathcal{A}}\\
\lesssim& \n{\phi*_{(x,\xi)}G_{\hbar}-\phi}_{L_{t}^{1}\mathcal{A}}+\hbar\n{\phi}_{L_{t}^{1}\mathcal{A}}\to 0.
\end{align*}
Thus, the Wigner function $f_{\hbar,\ve}(t,x,\xi)$ and the Husimi function $\wt{f}_{\hbar,\ve}(t,x,\xi)$ have the same convergence and limit. We are left to prove the convergence and the limit of the Husimi function $\wt{f}_{\hbar,\ve}(t,x,\xi)$.

We establish the $L_{x,\xi}^{1}$ uniform bound for $\wt{f}_{\hbar,\ve}(t,x,\xi)$.
By formula \eqref{equ:husimi,momentum}, we use the uniform estimate \eqref{equ:weighted,momentum of wigner} to get
\begin{align*}
\n{\xi^{2k}\wt{f}_{\hbar,\ve}(t,x,\xi)}_{L_{x,\xi}^{1}}=&\iint_{\R^{2}}\xi^{2k}\wt{f}_{\hbar,\ve} dxd\xi\\
\leq& \sum_{\al=0}^{2k}\binom{2k}{\al}\bbabs{\iint_{\R^{2}}
(\xi^{\al} f_{\hbar,\ve})*_{(x,\xi)}(\xi^{2k-\al}G_{\hbar})dxd\xi}\\
=&\sum_{\al=0}^{2k}\binom{2k}{\al}\bbabs{\lrs{\iint_{\R^{2}}\xi^{\al}f_{\hbar,\ve}dx d\xi} \lrs{\iint_{\R^{2}}\xi^{2k-\al}G_{\hbar}dxd\xi}}\\
\leq& C(2k,t).
\end{align*}
By H\"{o}lder's inequality, we arrive at
\begin{align}
\n{\xi^{k}\wt{f}_{\hbar,\ve}(t,x,\xi)}_{L_{x,\xi}^{1}}\leq \n{\xi^{2k}\wt{f}_{\hbar,\ve}(t,x,\xi)}_{L_{x,\xi}^{1}}^{\frac{1}{2}}
\n{\wt{f}_{\hbar,\ve}(t,x,\xi)}_{L_{x,\xi}^{1}}^{\frac{1}{2}}\leq  C(k,t).
\end{align}

\textbf{Step 2. Equicontinuity.}

To obtain the equicontinuity of $\wt{f}_{\hbar,\ve}(t,x,\xi)$ and apply the compactness argument, we prove the uniform estimates for the time-derivative of $\wt{f}_{\hbar,\ve}$.
We take the duality argument and notice that
 \begin{align}\label{equ:wigner,husimi,duality}
 \iint_{\R^{2}} \xi^{k}\pa_{t}\wt{f}_{\hbar,\ve}\phi dxd\xi=
  \iint_{\R^{2}} \xi^{k}\lrs{\pa_{t}f_{\hbar,\ve}*_{(x,\xi)}G_{\hbar}}\phi dxd\xi.
 \end{align}
 From the nonlinear Schr\"{o}dinger equation \eqref{equ:NLS,one-body,uniform estimate}, we have
\begin{align}\label{equ:wigner transform,evolution}
\pa_{t}f_{\hbar,\ve}+\xi\pa_{x}f_{\hbar,\ve}+\Theta[V_{\ve},f_{\hbar,\ve}]=0,
\end{align}
where the nonlinear term is
\begin{align}\label{equ:nonlinear term,wigner function}
\Theta[V_{\ve},f_{\hbar,\ve}]=\frac{i}{2\pi} \iint_{\R^{2}} \frac{V_{\ve}*\rho_{\hbar,\ve}(x+\frac{\hbar y}{2})-V_{\ve}*\rho_{\hbar,\ve}(x-\frac{\hbar y}{2})}{\hbar}f_{\hbar,\ve}(t,x,\eta)e^{-i(\xi-\eta)y}d\eta dy.
\end{align}
Putting \eqref{equ:wigner transform,evolution} into \eqref{equ:wigner,husimi,duality}, we obtain
\begin{align*}
\iint_{\R^{2}} \xi^{k}\pa_{t}\wt{f}_{\hbar,\ve}\phi dxd\xi=&
- \iint_{\R^{2}} \xi^{k}\lrs{(\xi\pa_{x}f_{\hbar,\ve})*_{(x,\xi)}G_{\hbar}}\phi dxd\xi\\
&- \iint_{\R^{2}} \xi^{k}\lrs{\Theta[V_{\ve},f_{\hbar,\ve}]*_{(x,\xi)}G_{\hbar}}\phi dxd\xi\\
:=&I_{1}+I_{2}.
\end{align*}

For the linear term $I_{1}$, we rewrite
\begin{align*}
I_{1}=&\iint_{\R^{2}} \xi^{k}\lrs{(\xi f_{\hbar,\ve})*_{(x,\xi)}G_{\hbar}}\lrs{\pa_{x}\phi} dxd\xi\\
=&\sum_{\al=0}^{k}\binom{k}{\al}\iint_{\R^{2}} \lrs{(\xi^{\al+1} f_{\hbar,\ve})*_{(x,\xi)}(\xi^{k-\al}G_{\hbar})}\lrs{\pa_{x}\phi} dxd\xi\\
=&\sum_{\al=0}^{k}\binom{k}{\al}\iint_{\R^{2}}\xi^{\al+1} f_{\hbar,\ve} \lrs{(\xi^{k-\al}G_{\hbar})*_{(x,\xi)}\lrs{\pa_{x}\phi}} dxd\xi.
\end{align*}
Using estimate \eqref{equ:wigner,f,h,estimate}, we get
\begin{align}\label{equ:time,f,h,I1}
|I_{1}|\lesssim \sum_{\al=0}^{k} \n{(\xi^{k-\al}G_{\hbar})*_{(x,\xi)}\pa_{x}\phi}_{\mathcal{A}}\lesssim \n{\pa_{x}\phi}_{\mathcal{A}}\lesssim \n{\phi}_{H_{x,\xi}^{2}}.
\end{align}

For the nonlinear term $I_{2}$, we rewrite
\begin{align}\label{equ:theta,I2,rewrite}
I_{2}=&-\sum_{\al=0}^{k}\binom{k}{\al}\iint_{\R^{2}} \lrs{\lrs{\xi^{\al}\Theta[V_{\ve},f_{\hbar,\ve}]}*_{(x,\xi)}\lrs{\xi^{k-\al}G_{\hbar}}}\phi dxd\xi\\
=&-\sum_{\al=0}^{k}\binom{k}{\al}\iint_{\R^{2}}\xi^{\al}\Theta[V_{\ve},f_{\hbar,\ve}]\lrs{\lrs{\xi^{k-\al}G_{\hbar}}*_{(x,\xi)}\phi} dxd\xi.\notag
\end{align}
Using again estimate \eqref{equ:wigner,f,h,estimate}, we have
\begin{align}\label{equ:nonlinear term,estimate,duality}
&\iint_{\R^{2}}\Theta[V_{\ve},f_{\hbar,\ve}] \phi dx d\xi\\
=&\iint_{\R^{2}}f_{\hbar,\ve}(t,x,\eta) \mathcal{F}_{\eta}^{-1}\lrc{\mathcal{F}_{y}(\phi)
\frac{V_{\ve}*\rho_{\hbar,\ve}(x+\frac{\hbar y}{2})-V_{\ve}*\rho_{\hbar,\ve}(x-\frac{\hbar y}{2})}{\hbar}}dx d\eta\notag\\
\lesssim & \bbn{ \mathcal{F}_{\eta}^{-1}\lrc{\mathcal{F}_{y}(\phi)
\frac{V_{\ve}*\rho_{\hbar,\ve}(x+\frac{\hbar y}{2})-V_{\ve}*\rho_{\hbar,\ve}(x-\frac{\hbar y}{2})}{\hbar}}}_{\mathcal{A}}\notag\\
\leq& \bbn{\mathcal{F}_{y}(\phi)
\frac{V_{\ve}*\rho_{\hbar,\ve}(x+\frac{\hbar y}{2})-V_{\ve}*\rho_{\hbar,\ve}(x-\frac{\hbar y}{2})}{\hbar}}_{L_{y}^{1}L_{x}^{\infty}}\notag\\
\leq& \n{|y|\mathcal{F}_{y}(\phi)}_{L_{y}^{1}L_{x}^{\infty}}\n{\pa_{x}V_{\ve}*\rho_{\hbar,\ve}}_{L_{x}^{\infty}}\notag\\
\lesssim&\n{\phi}_{H_{x,\xi}^{3}}.\notag
\end{align}
By the definition of $\Theta[V_{\ve},f_{\hbar,\ve}]$ in \eqref{equ:nonlinear term,wigner function}, we have
\begin{align}\label{equ:weight,theta}
\xi^{\al}\Theta[V_{\ve},f_{\hbar,\ve}]=\sum_{\al_{1}+\al_{2}=\al}\binom{\al}{\al_{1}}\Theta^{(\al_{1})}[V_{\ve},\xi^{\al_{2}}f_{\hbar,\ve}],
\end{align}
where
\begin{align*}
&\Theta^{(\al_{1})}[V_{\ve},\xi^{\al_{2}}f_{\hbar,\ve}]\\
=&\frac{i}{2\pi} \iint_{\R^{2}} D_{y}^{\al_{1}}\lrs{\frac{V_{\ve}*\rho_{\hbar,\ve}(x+\frac{\hbar y}{2})-V_{\ve}*\rho_{\hbar,\ve}(x-\frac{\hbar y}{2})}{\hbar}}\lrs{\eta^{\al_{2}}f_{\hbar,\ve}(t,x,\eta)}e^{-i(\xi-\eta)y}d\eta dy.
\end{align*}
Putting \eqref{equ:weight,theta} into \eqref{equ:theta,I2,rewrite}, in the same way as \eqref{equ:nonlinear term,estimate,duality}, we obtain
\begin{align}\label{equ:time,f,h,I2}
|I_{2}|\leq& \sum_{\al=0}^{k}\sum_{\al_{1}+\al_{2}=\al}\binom{k}{\al}\binom{\al}{\al_{1}}
\bbabs{\iint_{\R^{2}}\Theta^{(\al_{1})}[V_{\ve},\xi^{\al_{2}}f_{\hbar,\ve}]\lrs{\lrs{\xi^{k-\al}G_{\hbar}}*_{(x,\xi)}\phi} dxd\xi}\\
\lesssim&\sum_{\al=0}^{k}\sum_{\al_{1}+\al_{2}=\al}\bn{\lrs{\xi^{k-\al}G_{\hbar}}*_{(x,\xi)}\phi}_{H_{x,\xi}^{3}}
\n{\hbar^{\al_{1}}\pa_{x}^{\al_{1}+1}V_{\ve}*\rho_{\hbar,\ve}}_{L_{x}^{\infty}}\notag\\
\lesssim& C(k,t)\n{\phi}_{H_{x,\xi}^{3}},\notag
\end{align}
where in the last inequality we have used Young's inequality and uniform estimate \eqref{equ:weighted,eletron part}.

Combining estimates \eqref{equ:time,f,h,I1} and \eqref{equ:time,f,h,I2} on the terms $I_{1}$ and $I_{2}$, we arrive at
\begin{align*}
\n{\pa_{t}\xi^{k}\wt{f}_{\hbar,\ve}}_{H_{x,\xi}^{-3}}\leq C(k,t).
\end{align*}

\textbf{Step 3. Compactness Argument.}

By Arzel$\grave{\text{a}}$-Ascoli compactness lemma and a diagonal argument, for all $k\geq 1$ there exist a subsequence of $\lr{\wt{f}_{\hbar,\ve}}$, which we still denote by $\lr{\wt{f}_{\hbar,\ve}}$, and a limit point
$$f_{k}(t,x,\xi)\in C([0,T];H_{x,\xi}^{-3})$$
such that
\begin{align}\label{equ:f,h,strong,convergence,H}
\lim_{(\hbar,\ve)\to (0,0)}\n{\xi^{k}\wt{f}_{\hbar,\ve}(t,x,\xi)-f_{k}(t,x,\xi)}_{C([0,T];H_{x,\xi}^{-3})}=0.
\end{align}
Actually, we have that $f_{k}(t,x,\xi)=\xi^{k}f(t,x,\xi)$ due to that
\begin{align*}
&\lim_{(\hbar,\ve)\to(0,0)}\int_{0}^{T}\iint_{\R^{2}}\xi^{k}\wt{f}_{\hbar,\ve}(t,x,\xi) \phi(t,x,\xi)dtdxd\xi\\
=&\lim_{(\hbar,\ve)\to(0,0)}\int_{0}^{T}\iint_{\R^{2}}\wt{f}_{\hbar,\ve}(t,x,\xi) (\xi^{k}\phi(t,x,\xi))dtdxd\xi\\
=&\int_{0}^{T}\iint_{\R^{2}}f(t,x,\xi) \xi^{k}\phi(t,x,\xi)dtdxd\xi.
\end{align*}
Moreover, by the non-negativity of the Husimi function and the $L_{x,\xi}^{1}$ uniform bound for $\wt{f}_{\hbar,\ve}$, we get
\begin{align*}
&f(t,x,\xi)dxd\xi\in C([0,T];\mathcal{M}^{+}(\R^{2})-w^{*}).
\end{align*}
Therefore, we have completed the proof of \eqref{equ:convergence,limit point,measure} and \eqref{equ:convergence,f,h}.

Next, we prove estimate \eqref{equ:convergence,momentum,f,h}. By \eqref{equ:wigner,husimi,integral} in Lemma \ref{lemma:wigner,husimi} and the uniform estimate \eqref{equ:weighted,momentum of wigner}, for $\vp\in L_{t}^{1}([0,T];C_{c}^{\infty}(\R))$ we have
\begin{align*}
&\bbabs{\int_{0}^{T}\int_{\R}\lrs{\int_{\R}\xi^{k}f_{\hbar,\ve}(t,x,\xi)d\xi-\int_{\R}\xi^{k}\wt{f}_{\hbar,\ve}(t,x,\xi)d\xi}\vp dx dt}\\
\leq&\bbn{\int_{\R}\xi^{k}f_{\hbar,\ve}d\xi}_{L_{x}^{1}}\n{\vp*_{x}g_{\hbar}-\vp}_{L_{x}^{\infty}}+C(k)\sum_{\al=0}^{k-1}\hbar^{k-\al}
\bbn{\int_{\R}\xi^{\al}f_{\hbar,\ve} d\xi}_{L_{x}^{1}}\n{\vp}_{L_{x}^{\infty}}\\
\leq& \n{\vp*_{x}g_{\hbar}-\vp}_{L_{t}^{1}L_{x}^{\infty}}+\hbar\n{\vp}_{L_{t}^{1}L_{x}^{\infty}}\to 0.
\end{align*}
Thus, it suffices to deal with $\int_{\R}\xi^{k}\wt{f}_{\hbar,\ve}(t,x,\xi)d\xi$.
In the same way in which we obtain \eqref{equ:convergence,f,h}, we also have
\begin{align*}
\bbn{\pa_{t}\int_{\R}\xi^{k}\wt{f}_{\hbar,\ve}d\xi}_{H_{x}^{-3}}\leq C(k,t),
\end{align*}
which implies that there exists
a limit point $F_{k}(t,x)\in C([0,T];H_{x}^{-3})$ such that
\begin{align*}
\lim_{(\hbar,\ve)\to (0,0)}\bbn{\int_{\R}\xi^{k}\wt{f}_{\hbar,\ve}(t,x,\xi)d\xi-F_{k}(t,x)}_{C([0,T];H_{x}^{-3})}=0.
\end{align*}
We claim that $F_{k}(t,x)=\int_{\R}\xi^{k}f(t,x,d\xi)$. Indeed,
\begin{align*}
\int_{0}^{T}\int_{\R}F_{k}\vp dxdt=&\lim_{(\hbar,\ve)\to (0,0)}\int_{0}^{T}\int_{\R}
\lrs{\int_{\R}\xi^{k}\wt{f}_{\hbar,\ve}(t,x,\xi)d\xi}\vp dxdt\\
=&\lim_{(\hbar,\ve)\to (0,0)}\int_{0}^{T}
\int_{\R^{2}}(1+\xi^{2})\xi^{k}\wt{f}_{\hbar,\ve}(t,x,\xi)\frac{\vp}{1+\xi^{2}} d\xi dxdt\\
=&\int_{0}^{T}
\int_{\R^{2}}(1+\xi^{2})\xi^{k}f(t,x,\xi)\frac{\vp}{1+\xi^{2}} d\xi dxdt\\
=&\int_{0}^{T}\int_{\R}\lrs{\int_{\R}\xi^{k}f(t,x,\xi)d\xi}\vp dxdt,
\end{align*}
where in the second-to-last equality we have used convergence \eqref{equ:f,h,strong,convergence,H} and the fact that
\begin{align*}
\frac{1}{1+\xi^{2}}\vp(x)\in H_{x,\xi}^{3}.
\end{align*}
Hence, we conclude \eqref{equ:convergence,momentum,f,h} for $\vp\in L_{t}^{1}([0,T];C_{c}^{\infty}(\R))$. By the weighted uniform estimate \eqref{equ:weighted,momentum of wigner} and the fact that $C_{c}^{\infty}(\R)$ is dense in $C_{c}(\R)$, we arrive at \eqref{equ:convergence,momentum,f,h} for $\vp\in L_{t}^{1}([0,T];C_{c}(\R))$. Moreover,
for $\vp\in L_{t}^{1}([0,T];C_{b}(\R))$, we write
\begin{align}\label{equ:convergence,cb}
&\int_{0}^{T}\int_{\R}\lrs{\int_{\R}\xi^{k}f_{\hbar,\ve}(t,x,\xi) d\xi }\vp dx dt\\
=&\int_{0}^{T}\int_{\R}\lrs{\int_{\R}\xi^{k}f_{\hbar,\ve}(t,x,\xi) d\xi }\vp \chi(\frac{x}{R})dx dt\notag\\
&+\int_{0}^{T}\int_{\R}\lrs{\int_{\R}\xi^{k}f_{\hbar,\ve}(t,x,\xi) d\xi }\vp (1-\chi(\frac{x}{R}))dx dt.\notag
\end{align}
By the weighted uniform estimate \eqref{equ:weighted,momentum of wigner}, we get
\begin{align*}
&\bbabs{\int_{0}^{T}\int_{\R}\lrs{\int_{\R}\xi^{k}f_{\hbar,\ve}(t,x,\xi) d\xi }\vp (1-\chi(\frac{x}{R}))dx dt}\\
\leq& \frac{1}{R}\bbn{\lra{x}\int_{\R}\xi^{k}f_{\hbar,\ve}(t,x,\xi) d\xi }_{L_{t}^{\infty}L_{x}^{1}}\n{\vp}_{L_{t}^{1}L_{x}^{\infty}}\lesssim
\frac{1}{R}\to 0.
\end{align*}
Taking $(\hbar,\ve)\to (0,0)$ and then sending $R\to \infty$, \eqref{equ:convergence,cb} becomes
\begin{align*}
&\lim_{(\hbar,\ve)\to(0,0)}\int_{0}^{T}\int_{\R}\lrs{\int_{\R}\xi^{k}f_{\hbar,\ve}(t,x,\xi) d\xi }\vp dx dt\\
=&\lim_{R\to \infty}\int_{0}^{T}\int_{\R}\lrs{\int_{\R}\xi^{k}f(t,x,\xi) d\xi }\vp \chi(\frac{x}{R})dx dt\\
=&\int_{0}^{T}\int_{\R}\lrs{\int_{\R}\xi^{k}f(t,x,\xi) d\xi }\vp dx dt,
\end{align*}
where in the last equality we have used the dominated convergence theorem. Hence, we complete the proof of \eqref{equ:convergence,momentum,f,h}.

Next, we handle estimate \eqref{equ:convergence,momentum,f,h,pointwise}. As we have proven that
\begin{align*}
\lim_{(\hbar,\ve)\to (0,0)}\bbn{\int_{\R}\xi^{k}\wt{f}_{\hbar,\ve}(t,x,\xi)d\xi-\int_{\R}\xi^{k}f(t,x,\xi)d\xi}_{C([0,T];H_{x}^{-3})}=0,
\end{align*}
by the weighted uniform bound \eqref{equ:weighted,momentum of wigner}, we can improve it to the narrow convergence in a similar way in which we obtain \eqref{equ:convergence,momentum,f,h} and hence complete the proof of \eqref{equ:convergence,momentum,f,h,pointwise}.

Finally, we deal with the estimates \eqref{equ:measure,weighted,density}--\eqref{equ:measure,weighted,momentum}. It suffices to prove estimate \eqref{equ:measure,weighted,momentum}, as estimate \eqref{equ:measure,weighted,density} follows similarly.
 By the non-negativity of $f(t,dx,d\xi)$ and H\"{o}lder's inequality, we again only need to prove \eqref{equ:measure,weighted,momentum} with the weight function $\xi^{2k}$. By the weighted uniform estimate \eqref{equ:weighted,momentum of wigner}, we have
\begin{align*}
\iint_{\R^{2}} \chi(\frac{x}{R})\lra{x}\xi^{2k}f(t,dx,d\xi)=&\lim_{(\hbar,\ve)\to (0,0)}
\int_{\R}\chi(\frac{x}{R})\lra{x}\lrs{\int_{\R} \xi^{2k}f_{\hbar,\ve}(t,x,\xi)d\xi} dx\\
\leq&\sup_{(\hbar,\ve)}\bbn{\lra{x}\int_{\R}\xi^{2k}f_{\hbar,\ve}(t,x,\xi)d\xi}_{L_{x}^{1}}\leq C(2k,t).
\end{align*}
Sending $R\to \infty$, by Fatou's lemma, we arrive at \eqref{equ:measure,weighted,momentum}.
\end{proof}

\section{Conservation Laws for the Limit Measure}\label{section:Conservation Laws for the Limit Measure}
The conservation laws of mass, momentum and energy for the Wigner function $f_{\hbar,\ve}(t,x,\xi)$ are given by
\begin{align}
&\iint_{\R^{2}} f_{\hbar,\ve}(t,x,\xi)d\xi dx=\iint_{\R^{2}} f_{\hbar,\ve}(0,x,\xi)d\xi dx,\label{equ:convservation,mass,f,h}\\
&\iint_{\R^{2}}\xi f_{\hbar,\ve}(t,x,\xi)d\xi dx=\iint_{\R^{2}}\xi f_{\hbar,\ve}(0,x,\xi)d\xi dx,\label{equ:convservation,momentum,f,h}\\
&\iint_{\R^{2}}\xi^{2}f_{\hbar,\ve}(t,x,\xi)d\xi dx+\iint_{\R^{2}} V_{\ve}(x-y)\rho_{\hbar,\ve}(t,x)\rho_{\hbar,\ve}(t,y)dxdy\\
=&\iint_{\R^{2}}\xi^{2}f_{\hbar,\ve}(0,x,\xi)d\xi dx+\iint_{\R^{2}} V_{\ve}(x-y)\rho_{\hbar,\ve}(0,x)\rho_{\hbar,\ve}(0,y)dxdy.\notag
\end{align}
In the section, we prove the conservation laws for the limit measure $f(t,dx,d\xi)$.
\begin{lemma}\label{lemma:conservation laws}
The limit measure $f(t,dx,d\xi)$ satisfies the conservation laws of mass, momentum and energy
\begin{align}
&\iint_{\R^{2}} f(t,dx,d\xi)=\iint_{\R^{2}} f(0,dx,d\xi),\label{equ:convservation,mass}\\
&\iint_{\R^{2}} \xi f(t,dx,d\xi)=\iint_{\R^{2}} \xi f(0,dx,d\xi)\label{equ:convservation,momentum},\\
&\iint_{\R^{2}} \xi^{2}f(t,dx,d\xi)dxd\xi+ \frac{1}{2}\int_{\R^{2}} |x-y|\rho(t,dx)\rho(t,dy)\label{equ:convservation,energy}\\
=&\iint_{\R^{2}} \xi^{2}f(0,dx,d\xi)dxd\xi+ \frac{1}{2}\int_{\R^{2}} |x-y|\rho(0,dx)\rho(0,dy),\notag
\end{align}
where $\rho(t,dx)=\int_{\R}f(t,dx,\xi)d\xi$.
\end{lemma}
\begin{proof}
For \eqref{equ:convservation,mass}, by the narrow convergence \eqref{equ:convergence,momentum,f,h,pointwise} in Lemma \ref{lemma:convergence} and the conservation law of mass \eqref{equ:convservation,mass,f,h} for $f_{\hbar,\ve}(t,x,\xi)$, we have
\begin{align*}
\iint_{\R^{2}} f(t,dx,d\xi)=&\lim_{(\hbar,\ve)\to(0,0)}\iint_{\R^{2}} f_{\hbar,\ve}(t,x,\xi)dxd\xi\\
=&\lim_{(\hbar,\ve)\to(0,0)}\iint_{\R^{2}} f_{\hbar,\ve}(0,x,\xi)dxd\xi\\
=&\iint_{\R^{2}} f(0,dx,d\xi).
\end{align*}
In the same way, we also have the conservation law of momentum \eqref{equ:convservation,momentum}.

For the conservation law of energy \eqref{equ:convservation,energy}, using again the narrow convergence \eqref{equ:convergence,momentum,f,h,pointwise} in Lemma \ref{lemma:convergence}, we obtain the convergence for the kinetic energy part
\begin{align*}
\iint_{\R^{2}} \xi^{2}f(t,dx,d\xi)=&\lim_{(\hbar,\ve)\to(0,0)}\iint_{\R^{2}} \xi^{2}f_{\hbar,\ve}(t,x,\xi)dxd\xi,\\
\iint_{\R^{2}} \xi^{2}f(0,dx,d\xi)=&\lim_{(\hbar,\ve)\to(0,0)}\iint_{\R^{2}} \xi^{2}f_{\hbar,\ve}(0,x,\xi)dxd\xi.
\end{align*}

Next, we deal with the potential energy part. For simplicity, we omit the time variable and rewrite
\begin{align*}
&\iint_{\R^{2}} |x-y|e^{-\ve|x-y|}\rho_{\hbar,\ve}(x)\rho_{\hbar,\ve}(y)dxdy\\
=&\iint_{\R^{2}} |x-y|\rho_{\hbar,\ve}(x)\rho_{\hbar,\ve}(y)dxdy+\iint_{\R^{2}} |x-y|(1-e^{-\ve|x-y|})\rho_{\hbar,\ve}(x)\rho_{\hbar,\ve}(y)dxdy.
\end{align*}
By the weighted uniform estimate \eqref{equ:weighted estimate,mass}, we have
\begin{align*}
&\iint_{\R^{2}} |x-y|(1-e^{-\ve|x-y|})\rho_{\hbar,\ve}(x)\rho_{\hbar,\ve}(y)dxdy\\
\leq&\ve \iint_{\R^{2}} |x-y|^{2}\rho_{\hbar,\ve}(x)\rho_{\hbar,\ve}(y)dxdy\\
\lesssim &\ve\lrs{\int_{\R}|x|^{2}\rho_{\hbar,\ve}(x)dx}\lrs{\int_{\R} \rho_{\hbar,\ve}(y)dy}\lesssim \ve\to 0.
\end{align*}
Thus, we are left to consider the convergence of the term
\begin{align}\label{equ:potential part,convergence}
\iint_{\R^{2}} |x-y|\rho_{\hbar,\ve}(x)\rho_{\hbar,\ve}(y)dxdy,
\end{align}
To do this, we introduce a weighted transform. Set
\begin{align}
F_{\hbar,\ve}(x)=&\int_{\R}\frac{|x-y|}{\lra{x}^{1+\delta}}\rho_{\hbar,\ve}(y)dy,\quad F_{\hbar,\ve}(\pm \infty)=0,\\
G_{\hbar,\ve}(x)=&\int_{-\infty}^{x} \lra{y}^{1+\delta}\rho_{\hbar,\ve}(y)dy,
\end{align}
where $\delta\in (0,1)$ is a fixed constant.
After the weighted transform, we can use the integration by parts to get
\begin{align*}
\iint_{\R^{2}} |x-y|\rho_{\hbar,\ve}(x)\rho_{\hbar,\ve}(y)dxdy
=&\int_{\R} F_{\hbar,\ve}(x)\pa_{x}G_{\hbar,\ve}(x)dx\\
=&F_{\hbar,\ve}(x)G_{\hbar,\ve}(x)|^{+\infty}_{-\infty}-\int_{\R} G_{\hbar,\ve}(x)\pa_{x}F_{\hbar,\ve}(x)dx\\
=&-\int_{\R} G_{\hbar,\ve}(x)\pa_{x}F_{\hbar,\ve}(x)dx.
\end{align*}

Next, we get into the analysis of $G_{\hbar,\ve}$ and $F_{\hbar,\ve}$.
Using the weighted uniform estimate \eqref{equ:weighted estimate,mass}, we have
\begin{align*}
\n{G_{\hbar,\ve}}_{L_{x}^{\infty}}\leq C,\quad \n{\pa_{x}G_{\hbar,\ve}}_{L_{x}^{1}}\leq C.
\end{align*}
Together with the $L^{p}$ compactness criteria for $1 \leq p<\infty$, we conclude that there exist a subsequence of $\lr{G_{\hbar,\ve}}$ and an $L_{loc}^{p}$ function $G(x)$ such that
\begin{align}\label{equ:convergence,Gx}
G_{\hbar,\ve}\stackrel{L_{loc}^{p}}{\longrightarrow} G.
\end{align}
In a similar way, noticing that
\begin{align}\label{equ:derivative,F}
\pa_{x}F_{\hbar,\ve}=\int_{\R}\frac{1}{\lra{x}^{1+\delta}}\frac{x-y}{|x-y|}\rho_{\hbar,\ve}(y)dy-(1+\delta)\int_{\R}
\frac{x}{\lra{x}^{3+\delta}}|x-y|\rho_{\hbar,\ve}(y)dy,
\end{align}
and
\begin{align*}
\pa_{x}^{2}F_{\hbar,\ve}=I_{\hbar,\ve}^{(1)}+I_{\hbar,\ve}^{(2)}+I_{\hbar,\ve}^{(3)}+I_{\hbar,\ve}^{(4)},
\end{align*}
where
\begin{align*}
I_{\hbar,\ve}^{(1)}=&\frac{2\rho_{\hbar,\ve}(x)}{\lra{x}^{1+\delta}},\\
I_{\hbar,\ve}^{(2)}=&-(1+\delta)\int_{\R}
\frac{x}{\lra{x}^{3+\delta}}\frac{x-y}{|x-y|}\rho_{\hbar,\ve}(y)dy,\\
I_{\hbar,\ve}^{(3)}=&-(1+\delta)\int_{\R}\frac{x}{\lra{x}^{3+\delta}}\frac{x-y}{|x-y|}\rho_{\hbar,\ve}(y)dy,\\
I_{\hbar,\ve}^{(4)}=&-(1+\delta)\int_{\R}\lrs{\frac{1}{\lra{x}^{3+\delta}}-(3+\delta)\frac{x^{2}}{\lra{x}^{5+\delta}}}
\frac{x-y}{|x-y|}\rho_{\hbar,\ve}(y)dy,\\
\end{align*}
we use again the weighted uniform estimate \eqref{equ:weighted estimate,mass} to get
\begin{align*}
\n{\pa_{x}F_{\hbar,\ve}}_{L_{x}^{1}\cap L_{x}^{\infty}}\leq C,\quad \n{\pa_{x}^{2}F_{\hbar,\ve}}_{L_{x}^{1}}\leq C,
\end{align*}
which implies that there exist a subsequence of $\lr{\pa_{x}F_{\hbar,\ve}}$ and an $L_{loc}^{p}$ function which we denote by $\pa_{x}F(x)$ such that
\begin{align}\label{equ:convergence,partial,Fx}
\pa_{x}F_{\hbar,\ve}\stackrel{L_{loc}^{p}}{\longrightarrow} \pa_{x}F.
\end{align}

In the following, we identify the limits in \eqref{equ:convergence,Gx} and \eqref{equ:convergence,partial,Fx} by
\begin{align}
G(x)=&\int_{-\infty}^{x}\lra{y}^{1+\delta}\rho(dy),\quad a.e.,\label{equ:limiting,G}\\  \pa_{x}F(x)=&\pa_{x}\int_{\R}\frac{|x-y|}{\lra{x}^{1+\delta}}\rho(dy),\quad a.e..\label{equ:limiting,F}
\end{align}
By the weighted estimate \eqref{equ:measure,weighted,density} and the fact that $\int_{\R}\frac{|x-y|}{\lra{x}^{1+\delta}}\rho(dy)$ is Lipshitz continuous, \eqref{equ:limiting,G} and \eqref{equ:limiting,F} are indeed well-defined.

Take a test function $\vp(x)\in C_{c}^{\infty}(\R)$, we have
\begin{align*}
\int_{\R} \vp(x)G_{\hbar,\ve}(x)  dx=&\int_{\R}\lrs{\int_{y}^{+\infty}\vp(x)dx} \lra{y}^{1+\delta}\rho_{\hbar,\ve}(y)dy\\
=&A_{1}+A_{2},\\
\end{align*}
where
\begin{align*}
A_{1}=&\int_{\R}\chi(\frac{y}{R})\lrs{\int_{y}^{+\infty}\vp(x)dx} \lra{y}^{1+\delta}\rho_{\hbar,\ve}(y)dy,\\
A_{2}=&\int_{\R}(1-\chi(\frac{y}{R}))\lrs{\int_{y}^{+\infty}\vp(x)dx} \lra{y}^{1+\delta}\rho_{\hbar,\ve}(y)dy.
\end{align*}
Using the weighted uniform estimate \eqref{equ:weighted estimate,mass}, we obtain
\begin{align*}
|A_{2}|\leq \frac{1}{R^{1-\delta}}\n{\vp}_{L_{x}^{1}}\n{\lra{y}^{2}\rho_{\hbar,\ve}}_{L_{y}^{1}}\to 0.
\end{align*}
Therefore, letting first $(\hbar,\ve)\to (0,0)$ and then $R\to \infty$, by the convergence \eqref{equ:convergence,momentum,f,h,pointwise}, we get
\begin{align}\label{equ:double limit,G}
\lim_{(\hbar,\ve)\to(0,0)}\int_{\R} \vp(x)G_{\hbar,\ve}(x)  dx=&\lim_{R\to\infty}
\int_{\R}\chi(\frac{y}{R})\lrs{\int_{y}^{+\infty}\vp(x)dx} \lra{y}^{1+\delta}\rho(dy)\\
=&
\int_{\R}\lrs{\int_{y}^{+\infty}\vp(x)dx} \lra{y}^{1+\delta}\rho(dy)\notag\\
=&\int_{\R}\vp(x) \lrs{\int_{-\infty}^{x}\lra{y}^{1+\delta}\rho(dy)} dx,\notag
\end{align}
where in the second and last equalities we have used the dominated convergence theorem and Fubini's theorem based on the weighted estimate \eqref{equ:measure,weighted,density} that
\begin{align*}
\int_{\R}\lra{y}^{2}\rho(t,dy)=\iint_{\R^{2}}\lra{y}^{2}f(t,dy,d\xi)\leq C(t).
\end{align*}
Hence, we complete the proof of \eqref{equ:limiting,G} for $G(x)$.

For \eqref{equ:limiting,F}, by \eqref{equ:derivative,F} we have
\begin{align*}
&\int_{\R}\vp(x) \pa_{x}F_{\hbar,\ve}(x)  dx\\
=&\int_{\R} \lrs{\int_{\R} \lrs{\frac{x-y}{|x-y|}\frac{\vp(x)}{\lra{x}^{1+\delta}}-(1+\delta)|x-y|\frac{x\vp(x)}{\lra{x}^{3+\delta}}}dx}\rho_{\hbar,\ve}(y) dy.
\end{align*}
In a similar fashion as in \eqref{equ:double limit,G}, we get
\begin{align*}
&\lim_{(\hbar,\ve)\to(0,0)}\int_{\R}\vp(x) \pa_{x}F_{\hbar,\ve}(x)  dx\\
=&\int_{\R} \lrs{\int_{\R} \lrs{\frac{x-y}{|x-y|}\frac{\vp(x)}{\lra{x}^{1+\delta}}-(1+\delta)|x-y|\frac{x\vp(x)}{\lra{x}^{3+\delta}}}dx}\rho(dy)\\
=&\int_{\R}\vp(x)\lrs{\int_{\R}\frac{1}{\lra{x}^{1+\delta}}\frac{x-y}{|x-y|}\rho(dy)-(1+\delta)\int_{\R}
\frac{x}{\lra{x}^{3+\delta}}|x-y|\rho(dy)}  dx\\
=&\int_{\R}\vp(x)\lrs{\pa_{x}\int_{\R}\frac{|x-y|}{\lra{x}^{1+\delta}}\rho(dy)}dx,
\end{align*}
where in the last inequality we have used Leibniz rule for the Lipschitz continuous function.
This completes the proof of \eqref{equ:limiting,F}.

Finally, we prove the convergence of the potential energy part to the desired form. By \eqref{equ:derivative,F} and the weighted uniform estimate \eqref{equ:weighted estimate,mass}, we notice that
\begin{align*}
\n{\lra{x}^{\frac{\delta}{2}}\pa_{x}F_{\hbar,\ve}}_{L_{x}^{1}}\lesssim \int_{\R}\frac{1}{\lra{x}^{1+\frac{\delta}{2}}}dx\n{\lra{x}\rho_{\hbar,\ve}}_{L_{x}^{1}}\leq C,
\end{align*}
and hence obtain
\begin{align}\label{equ:smallness,GF,R}
\bbabs{\int_{|x|\geq R} G_{\hbar,\ve}(x)\pa_{x}F_{\hbar,\ve}(x)dx}\leq \frac{1}{R^{\frac{\delta}{2}}}\n{G_{\hbar,\ve}}_{L_{x}^{\infty}}
\n{\lra{x}^{\frac{\delta}{2}}\pa_{x}F_{\hbar,\ve}}_{L_{x}^{1}}\leq \frac{C}{R^{\frac{\delta}{2}}}\to 0.
\end{align}
As $G_{\hbar,\ve}(x)\stackrel{L_{loc}^{2}}{\to} G$, $\pa_{x}F_{\hbar,\ve}\stackrel{L_{loc}^{2}}{\to}\pa_{x}F$, letting first $(\hbar,\ve)\to (0,0)$ and then $R\to \infty$, we use \eqref{equ:smallness,GF,R} and the dominated convergence theorem to get
\begin{align*}
&\lim_{(\hbar,\ve)\to(0,0)}\int_{\R} G_{\hbar,\ve}(x)\pa_{x}F_{\hbar,\ve}(x)dx\\
=&\lim_{R\to \infty}\lim_{(\hbar,\ve)\to(0,0)}\int_{\R} \chi(\frac{x}{R})G_{\hbar,\ve}(x)\pa_{x}F_{\hbar,\ve}(x)dx\\
&+\lim_{R\to \infty}\lim_{(\hbar,\ve)\to(0,0)}\int_{\R} (1-\chi(\frac{x}{R}))G_{\hbar,\ve}(x)\pa_{x}F_{\hbar,\ve}(x)dx\\
=&\lim_{R\to \infty}\int_{\R} \chi(\frac{x}{R})G(x)\pa_{x}F(x)dx\\
=&\int_{\R} G(x)\pa_{x}F(x)dx.
\end{align*}
By Fubini's theorem, we obtain
\begin{align*}
\int_{\R} G(x)\pa_{x}F(x)dx=&\int_{\R}\lrs{\int_{\R}1_{\lr{y\leq x}}\pa_{x}F(x)dx}\lra{y}^{1+\delta}\rho(dy)\\
=&-\int_{\R}\int_{\R}\frac{|x-y|}{\lra{y}^{1+\delta}}\rho(dx)\lra{y}^{1+\delta}\rho(dy)\\
=&-\iint_{\R^{2}}|x-y|\rho(dx)\rho(dy).
\end{align*}
Therefore, together with the convergence of the kinetic energy part, we complete the proof of the conservation law of energy \eqref{equ:convservation,energy}.

\end{proof}

\section{Moment Convergence to the Vlasov-Poisson Equation}\label{section:momentum convergence}
In the section, our goal is to prove the convergence of some subsequence of $f_{\hbar,\ve}(t,x,\xi)$ to the Vlasov-Poisson equation for the test functions of the moment form
\begin{align}\label{equ:test function,moment form}
\phi(t,x,\xi)=\vp(t,x)\xi^{k}.
\end{align}
That is, the limit point $f(t,x,\xi)$ satisfies the Vlasov-Poisson equation for the test functions of the form \eqref{equ:test function,moment form}, based on which we extend it to all test function $\phi\in C_{c}^{\infty}((0,T)\times \R^{2})$ in Section \ref{section:Full Convergence to the Vlasov-Poisson Equation}.

 The main result of this section is Lemma \ref{lemma:momentum convergence,vp} below.
\begin{lemma}\label{lemma:momentum convergence,vp}
Let $T>0$ and $k\geq 0$. For $\vp \in C_{c}^{\infty}(\Omega_{T})$, there holds that
\begin{equation}
\begin{aligned}
\int_{\Omega_T} \int_{\mathbb{R}}\left(\partial_t \vp+\xi \partial_x \vp\right) \xi^k f(t, d x, d \xi) dt -k \int_{\Omega_T} \vp \ol{E} \lrs{\int_{\mathbb{R}} \xi^{k-1} f(t,dx,d\xi)} d t=0,
\end{aligned}
\end{equation}
where $\Omega_{T}=(0,T)\times \R$ and $\ol{E}$ is the Vol$^\prime$pert's symmetric average defined in \eqref{equ:volpert average}.
\end{lemma}
 To motivate the proof of Lemma \ref{lemma:momentum convergence,vp},
from the equation \eqref{equ:wigner transform,evolution} of $f_{\hbar,\ve}(t,x,\xi)$, we observe that the moment equation
\begin{align}\label{equ:momentum,f}
&\pa_{t}\int_{\R}\xi^{k}f_{\hbar,\ve}d\xi+\pa_{x}\int_{\R}\xi^{k+1}f_{\hbar,\ve}d\xi+
kE_{\hbar,\ve}\int_{\R}\xi^{k-1}f_{\hbar,\ve}d\xi+\mathrm{R}_{\hbar,\ve}^{(k)}=0,
\end{align}
where $E_{\hbar,\ve}=\pa_{x}V_{\ve}*\rho_{\hbar,\ve}$ and the remainder term is
\begin{equation}\label{equ:remainder term}
\mathrm{R}_{\hbar,\ve}^{(k)}=
\left\{
\begin{aligned}
&i\sum_{2\leq \al\leq k}\binom{k}{\al}\frac{\hbar^{\al-1}}{2^{k}}(1-(-1)^{\al})D_{x}^{\al}(V_{\ve}*\rho_{\hbar,\ve})
\int_{\R}\xi^{k-\al}f_{\hbar,\ve}d\xi,\quad &k\geq 3,\\
&0,\quad &k=0,1,2.
\end{aligned}
\right.
\end{equation}

By the convergence result in Lemma \ref{lemma:convergence}, we have the convergence for the linear term
\begin{align*}
\lim_{(\hbar,\ve)\to (0,0)}\int_{\Omega_T} \int_{\mathbb{R}}\left(\partial_t \vp+\xi \partial_x \vp\right) \xi^k f_{\hbar,\ve} d x d \xi d t
=&\int_{\Omega_T} \int_{\mathbb{R}}\left(\partial_t \vp+\xi \partial_x \vp\right) \xi^k f(t,dx,d\xi) d t.
\end{align*}
Therefore, to conclude Lemma \ref{lemma:momentum convergence,vp}, we are left to prove the vanishing of the remainder term
and the convergence of the nonlinear term, that is,
\begin{align}
&\lim_{(\hbar,\ve)\to (0,0)}\bbabs{\int_{\Omega_{T}}\vp \mathrm{R}_{\hbar,\ve}^{(k)} dxdt}=0,\label{equ:vanishing remainder,intro}\\
&\lim _{(\hbar,\ve)\to 0} \int_{\Omega_T} \vp E_{\hbar,\ve} \lrs{\int_{\mathbb{R}} \xi^{k-1} f_{\hbar,\ve} d \xi} d x d t=\int_{\Omega_T} \vp \ol{E} \lrs{\int_{\mathbb{R}} \xi^{k-1} f(t,dx,d\xi)} dt. \label{equ:convergence,nonlinear,intro}
\end{align}

We deal with the remainder term and prove \eqref{equ:vanishing remainder,intro} in Section \ref{section:Vanishing via a cancellation structure}.
The convergence of the nonlinear term is usually one of the main difficulties, as it is actually a problem of convergence of the product form in the mixed limit. We prove \eqref{equ:convergence,nonlinear,intro} for the $k=1,2$  case in Section \ref{section:Convergence of the Nonlinear Term,k=1,2}, and the general $k\geq 3$ case in Section \ref{section:Convergence of the Nonlinear Term,k>=3}.

\subsection{Vanishing Remainder Terms via a Cancellation Structure}\label{section:Vanishing via a cancellation structure}
In the space of the strong topology, the remainder term $\mathrm{R}_{\hbar,\ve}^{(m)}$ in \eqref{equ:remainder term} is only uniformly bounded in $L^{\infty}([0,T];L_{x}^{1})$.
We follow the idea of \cite{ZZM02} to prove that the remainder term would vanish in the weak sense by an iteration scheme using a cancellation structure.

First, we provide the weighted uniform bound for the remainder term.
\begin{lemma}
For $T>0$ and $m\geq 3$, we have
\begin{align}\label{equ:uniform estimate,remainder}
\n{\lra{x}\mathrm{R}_{\hbar,\ve}^{(m)}}_{L_{t}^{\infty}([0,T];L_{x}^{1})}\leq C(m,T).
\end{align}
\end{lemma}
\begin{proof}
By the weighted uniform estimate \eqref{equ:weighted,momentum of wigner} and the uniform bound \eqref{equ:weighted,eletron part}, we use H\"{o}lder's inequality to get
\begin{align*}
\n{\lra{x}\mathrm{R}_{\hbar,\ve}^{(m)}}_{L_{t}^{\infty}L_{x}^{1}}\lesssim&
\sum_{\al=0}^{m}\bbn{\lra{x}\hbar^{\al-1}\pa_{x}^{\al}\lrs{V_{\ve}*\rho_{\hbar,\ve}}\int_{\R}\xi^{m-\al}f_{\hbar,\ve}d\xi}_{L_{t}^{\infty}L_{x}^{1}}\\
\leq&\n{\hbar^{\al-1}\pa_{x}^{\al}V_{\ve}*\rho_{\hbar,\ve}}_{L_{t}^{\infty}L_{x}^{\infty}}
\bbn{\lra{x}\int_{\R}\xi^{m-\al}f_{\hbar,\ve}d\xi}_{L_{t}^{\infty}L_{x}^{1}}\\
\leq&C(m,T).
\end{align*}
\end{proof}

Next, we get into the analysis of the remainder term.
\begin{lemma}\label{lemma:remainder,vanishing}
Let $T>0$ and $k\geq 3$. For $\vp(t,x)\in C_{c}^{1}(\Omega_{T})$ and $\al=2n+1\leq k$ with $n\geq 1$, we have
\begin{align}\label{equ:remainder,estimate}
\hbar^{\al-1}\int_{\Omega_{T}}\vp \pa_{x}^{\al}(V_{\ve}*\rho_{\hbar,\ve})\lrs{\int_{\R}\xi^{k-\al}f_{\hbar,\ve}d\xi} dx dt= \mathcal{E}(\hbar,\ve),
\end{align}
with
\begin{align}\label{equ:remainder,error}
|\mathcal{E}(\hbar,\ve)|\leq C(k,\al)\lrs{\hbar \n{\nabla_{t,x} \vp}_{L_{t}^{1}L_{x}^{\infty}}+\hbar \n{\vp}_{L_{t}^{1}L_{x}^{\infty}}+\ve \n{\vp}_{L_{t}^{1}L_{x}^{\infty}}}.
\end{align}
In particular, we have the quantitative estimate that
\begin{align}\label{equ:rate,remainder}
\bbabs{\int_{\Omega_{T}}\vp \mathrm{R}_{\hbar,\ve}^{(k)} dxdt}\leq C(k)\lrs{\hbar \n{\nabla_{t,x} \vp}_{L_{t}^{1}L_{x}^{\infty}}+\hbar \n{\vp}_{L_{t}^{1}L_{x}^{\infty}}+\ve \n{\vp}_{L_{t}^{1}L_{x}^{\infty}}},\quad \f \vp(t,x)\in C_{c}^{1}(\Omega_{T}),
\end{align}
and have the qualitative convergence that
\begin{align}\label{equ:convergence,remainder}
\lim_{(\hbar,\ve)\to(0,0)}\bbabs{\int_{\Omega_{T}}\vp \mathrm{R}_{\hbar,\ve}^{(k)} dxdt}=0,\quad \f \vp(t,x)\in L_{t}^{1}([0,T];C_{b}(\R)).
\end{align}
\end{lemma}
\begin{proof}
For convenience, we take up the notation
\begin{align}
U_{\ve}(x):=&\frac{1}{2}|x|-V_{\ve}(x)=\frac{1}{2}|x|(1-e^{-\ve |x|}),
\end{align}
and hence rewrite
\begin{align}\label{equ:remainder term,I1,I2}
\hbar^{\al-1}\int_{\Omega_{T}}\vp \pa_{x}^{\al}(V_{\ve}*\rho_{\hbar,\ve})\lrs{\int_{\R}\xi^{k-\al}f_{\hbar,\ve}d\xi} dx dt=I_{1}-I_{2},
\end{align}
where
\begin{align*}
I_{1}=&\hbar^{\al-1}\int_{\Omega_{T}}\vp \lrs{\pa_{x}^{\al}\frac{|x|}{2}*\rho_{\hbar,\ve}}\lrs{\int_{\R}\xi^{k-\al}f_{\hbar,\ve}d\xi} dx dt,\\
I_{2}=&\hbar^{\al-1}\int_{\Omega_{T}}\vp \lrs{\pa_{x}^{\al}U_{\ve}*\rho_{\hbar,\ve}}\lrs{\int_{\R}\xi^{k-\al}f_{\hbar,\ve}d\xi} dx dt.
\end{align*}

We first deal with the term $I_{2}$. Noting that
\begin{align*}
|\pa_{x}U_{\ve}(x)|\leq \ve|x|\quad a.e.,
\end{align*}
we have the pointwise bound
\begin{align}\label{equ:U,pointwise estimate}
\hbar^{\al-1}|\pa_{x}^{\al}(U_{\ve}*\rho_{\hbar,\ve})|=&\hbar^{\al-1}|\pa_{x}U_{\ve}*\pa_{x}^{\al-1}\rho_{\hbar,\ve}|\\
\leq& \hbar^{\al-1}\ve \int_{\R} |x-y||\pa_{x}^{\al-1}\rho_{\hbar,\ve}(y)|dy\notag\\
\leq& \ve \lra{x}\n{\lra{x}\hbar^{\al-1}\pa_{x}^{\al-1}\rho_{\hbar,\ve}}_{L_{x}^{1}}\notag\\
\lesssim& \ve \lra{x},\notag
\end{align}
where in the last inequality we have used the uniform estimate \eqref{equ:weighted,momentum of wigner}.
By \eqref{equ:U,pointwise estimate}, we then use H\"{o}lder's inequality and the uniform estimate \eqref{equ:weighted,momentum of wigner} to obtain
\begin{align}\label{equ:remainder term,I2}
I_{2}\lesssim \ve \n{\vp}_{L_{t}^{1}L_{x}^{\infty}}\bbn{\lra{x}\int_{\R}\xi^{k-\al}f_{\hbar,\ve}d\xi}_{L_{t}^{\infty}L_{x}^{1}}\lesssim \ve \n{\vp}_{L_{t}^{1}L_{x}^{\infty}}.
\end{align}

Next, we handle the term $I_{1}$ via an iteration scheme.
For $\al=2n+1\leq k$ with $n\geq 1$, we set the notation
\begin{align}
M_{\vp}^{(k,\al,j)}=\hbar^{\al-1}\int_{\Omega_{T}}\vp \lrs{\pa_{x}^{\al-2} \int_{\R} \xi^{j} f_{\hbar,\ve}d\xi}  \lrs{\int_{\R} \xi^{k-\al-j} f_{\hbar,\ve}d\xi} dxdt.
\end{align}
In particular, noticing that $\pa_{x}^{2}(\frac{|x|}{2})=\delta(x)$, we have
$$I_{1}=\hbar^{\al-1}\int_{\Omega_{T}}\vp \pa_{x}^{\al-2}((\pa_{x}^{2}\frac{|x|}{2})*\rho_{\hbar,\ve})\lrs{\int_{\R}\xi^{k-\al}f_{\hbar,\ve}d\xi} dx dt=
M_{\vp}^{(k,\al,0)}.$$

In the following, we get into the analysis of $M_{\vp}^{(k,\al,j)}$.
By integration by parts, we have
\begin{align*}
M_{\vp}^{(k,\al,j)}=&-\hbar^{\al-1}\int_{\Omega_{T}}\vp \lrs{\pa_{x}^{\al-3} \int_{\R} \xi^{j} f_{\hbar,\ve}d\xi}  \lrs{\pa_{x}\int_{\R} \xi^{k-\al-j} f_{\hbar,\ve}d\xi} dxdt\\
&-\hbar^{\al-1}\int_{\Omega_{T}}\pa_{x}\vp \lrs{\pa_{x}^{\al-3} \int_{\R} \xi^{j} f_{\hbar,\ve}d\xi}  \lrs{\int_{\R} \xi^{k-\al-j} f_{\hbar,\ve}d\xi} dxdt.
\end{align*}
Using the moment equation \eqref{equ:momentum,f} in which we take $m=k-\al-j-1$, we have
\begin{align}\label{equ:M,k,j,expand}
M_{\vp}^{(k,\al,j)}=A_{1}+A_{2}+A_{3}+A_{4},
\end{align}
where
\begin{align*}
A_{1}=&\hbar^{\al-1}\int_{\Omega_{T}}\vp \lrs{\pa_{x}^{\al-3} \int_{\R} \xi^{j} f_{\hbar,\ve}d\xi}  \lrs{\pa_{t} \int_{\R} \xi^{m} f_{\hbar,\ve}d\xi} dxdt,\\
A_{2}=&\hbar^{\al-1}\int_{\Omega_{T}}\vp \lrs{\pa_{x}^{\al-3} \int_{\R} \xi^{j} f_{\hbar,\ve}d\xi}  \lrs{mE_{\hbar,\ve} \int_{\R} \xi^{m-1} f_{\hbar,\ve}d\xi} dxdt,\\
A_{3}=&\hbar^{\al-1}\int_{\Omega_{T}}\vp \lrs{\pa_{x}^{\al-3} \int_{\R} \xi^{j} f_{\hbar,\ve}d\xi}  \lrs{\mathrm{R}_{\hbar,\ve}^{(m)}} dxdt,\\
A_{4}=&-\hbar^{\al-1}\int_{\Omega_{T}}\pa_{x}\vp \lrs{\pa_{x}^{\al-3} \int_{\R} \xi^{j} f_{\hbar,\ve}d\xi}  \lrs{\int_{\R} \xi^{k-\al-j} f_{\hbar,\ve}d\xi} dxdt.
\end{align*}

We can directly bound the terms $A_{2}$, $A_{3}$ and $A_{4}$. For $A_{2}$, by H\"{o}lder's equality, the uniform estimates \eqref{equ:higer,momentum of wigner,Lmax} and \eqref{equ:weighted,momentum of wigner}, we have
\begin{align*}
|A_{2}|\lesssim& \hbar\n{\vp}_{L_{t}^{1}L_{x}^{\infty}}\bbn{\hbar^{\al-2}\pa_{x}^{\al-3} \int_{\R} \xi^{j} f_{\hbar,\ve}d\xi}_{L_{t,x}^{\infty}}
\n{E_{\hbar,\ve}}_{L_{t,x}^{\infty}}\bbn{\int_{\R} \xi^{m-1} f_{\hbar,\ve}d\xi}_{L_{t}^{\infty}L_{x}^{1}}\\
\lesssim& \hbar\n{\vp}_{L_{t}^{1}L_{x}^{\infty}}.
\end{align*}
In a similar way, for $A_{3}$ we use the uniform estimate \eqref{equ:higer,momentum of wigner,Lmax} and the $L_{t}^{\infty}L_{x}^{1}$ bound for $\mathrm{R}_{\hbar,\ve}^{(m)}$ to obtain
\begin{align*}
|A_{3}|\leq \hbar\n{\vp}_{L_{t}^{1}L_{x}^{\infty}}\bbn{\hbar^{\al-2}\pa_{x}^{\al-3} \int_{\R} \xi^{j} f_{\hbar,\ve}d\xi}_{L_{t,x}^{\infty}}
\n{\mathrm{R}_{\hbar,\ve}^{(m)}}_{L_{t}^{\infty}L_{x}^{1}}
\lesssim& \hbar\n{\vp}_{L_{t}^{1}L_{x}^{\infty}}.
\end{align*}
In the same manner, we have
\begin{align*}
|A_{4}|\lesssim \hbar\n{\pa_{x}\vp}_{L_{t}^{1}L_{x}^{\infty}}.
\end{align*}

Next, we get into the analysis of the term $A_{1}$. Using integration by parts in the time variable, we get
\begin{align*}
A_{1}=A_{11}+A_{12},
\end{align*}
where
\begin{align*}
A_{11}=&-\hbar^{\al-1}\int_{\Omega_{T}}\vp \lrs{\pa_{x}^{\al-3} \pa_{t}\int_{\R} \xi^{j} f_{\hbar,\ve}d\xi}  \lrs{ \int_{\R} \xi^{m} f_{\hbar,\ve}d\xi} dxdt,\\
A_{12}=&-\hbar^{\al-1}\int_{\Omega_{T}}\pa_{t}\vp \lrs{\pa_{x}^{\al-3} \int_{\R} \xi^{j} f_{\hbar,\ve}d\xi}  \lrs{ \int_{\R} \xi^{m} f_{\hbar,\ve}d\xi} dxdt.
\end{align*}
 As the term $A_{12}$ can be treated in a similar way as $A_{2}$, we have
 \begin{align*}
 |A_{12}|\lesssim \hbar \n{\pa_{t}\vp}_{L_{t}^{1}L_{x}^{\infty}}.
 \end{align*}

 For the term $A_{11}$, we use again the moment equation \eqref{equ:momentum,f} to get
 \begin{align*}
 A_{11}=A_{111}+A_{112}+A_{113},
 \end{align*}
 where
 \begin{align*}
 A_{111}=&M_{\vp}^{(k,\al,j+1)}=\hbar^{\al-1}\int_{\Omega_{T}}\vp \lrs{\pa_{x}^{\al-2}\int_{\R} \xi^{j+1} f_{\hbar,\ve}d\xi}  \lrs{ \int_{\R} \xi^{m} f_{\hbar,\ve}d\xi} dxdt,\\
 A_{112}=&j\hbar^{\al-1}\int_{\Omega_{T}}\vp \lrs{\pa_{x}^{\al-3} \lrs{E_{\hbar,\ve}\int_{\R} \xi^{j+1} f_{\hbar,\ve}d\xi}}  \lrs{ \int_{\R} \xi^{m} f_{\hbar,\ve}d\xi} dxdt,\\
 A_{113}=&\hbar^{\al-1}\int_{\Omega_{T}}\vp \lrs{\pa_{x}^{\al-3} \mathrm{R}_{\hbar,\ve}^{(j)} }  \lrs{ \int_{\R} \xi^{m} f_{\hbar,\ve}d\xi} dxdt.
 \end{align*}

  For the term $A_{112}$, by the uniform estimates \eqref{equ:higer,momentum of wigner,Lmax} and \eqref{equ:weighted,eletron part}, we obtain
\begin{align*}
|A_{112}|\lesssim \hbar\n{\vp}_{L_{t}^{1}L_{x}^{\infty}}.
\end{align*}

For the term $A_{113}$, using Leibniz rule, the uniform estimates \eqref{equ:higer,momentum of wigner,Lmax} and \eqref{equ:weighted,eletron part}, we have
\begin{align*}
\n{\hbar^{\al-3}\pa_{x}^{\al-3}\mathrm{R}_{\hbar,\ve}^{(j)}}_{L_{x}^{1}}\lesssim C(j,\al,t),
\end{align*}
and hence obtain
\begin{align*}
|A_{113}|\leq \hbar\n{\vp}_{L_{t}^{1}L_{x}^{\infty}}\n{\hbar^{\al-3}\pa_{x}^{\al-3}\mathrm{R}_{\hbar,\ve}^{(j)}}_{L_{t}^{\infty}L_{x}^{1}}
\bbn{\hbar \int_{\R} \xi^{m} f_{\hbar,\ve}d\xi}_{L_{t,x}^{\infty}}\lesssim \hbar\n{\vp}_{L_{t}^{1}L_{x}^{\infty}}.
\end{align*}

To sum up, we finally arrive at
\begin{align}\label{equ:induction formula,Mj}
M_{\vp}^{(k,\al,j)}=M_{\vp}^{(k,\al,j+1)}+\mathcal{O}(\hbar),
\end{align}
with $|\mathcal{O}(\hbar)|\lesssim \hbar \lrs{\n{\nabla_{t,x} \vp}_{L_{t}^{1}L_{x}^{\infty}}+\n{\vp}_{L_{t}^{1}L_{x}^{\infty}}}$.

When $k=2l+1$, $\al=2n+1$, $l\geq n\geq 1$,
iteratively using \eqref{equ:induction formula,Mj} and integration by parts, we have
\begin{align*}
M_{\vp}^{(k,\al,0)}=&M_{\vp}^{(2l+1,2n+1,l-n)}+\mathcal{O}(\hbar)\\
=&\hbar^{\al-1}\int_{\Omega_{T}}\vp \lrs{\pa_{x}^{2n-1} \int_{\R} \xi^{l-n} f_{\hbar,\ve}d\xi}  \lrs{ \int_{\R} \xi^{l-n} f_{\hbar,\ve}d\xi} dxdt+\mathcal{O}(\hbar)\\
=&(-1)\hbar^{\al-1}\int_{\Omega_{T}}\vp \lrs{\pa_{x}^{2n-2} \int_{\R} \xi^{l-n} f_{\hbar,\ve}d\xi}  \lrs{ \pa_{x}\int_{\R} \xi^{l-n} f_{\hbar,\ve}d\xi} dxdt+\mathcal{O}(\hbar)\\
=&...\\
=&(-1)^{n-1}\hbar^{\al-1}\int_{\Omega_{T}}\vp \lrs{\pa_{x}^{n} \int_{\R} \xi^{l-n} f_{\hbar,\ve}d\xi}  \lrs{ \pa_{x}^{n-1}\int_{\R} \xi^{l-n} f_{\hbar,\ve}d\xi} dxdt+\mathcal{O}(\hbar).
\end{align*}
Noticing that $\int_{\R} \xi^{l-n} f_{\hbar,\ve}d\xi$ is real-valued, we hence have
\begin{align}\label{equ:remainder term,I1}
M_{\vp}^{(k,\al,0)}=&(-1)^{n-1}\frac{\hbar^{\al-1}}{2}\int_{\Omega_{T}}\vp \pa_{x}\lrs{\pa_{x}^{n-1} \int_{\R} \xi^{l-n} f_{\hbar,\ve}d\xi}^{2}  dxdt+\mathcal{O}(\hbar)\\
=&(-1)^{n-1}\frac{\hbar^{\al-1}}{2}\int_{\Omega_{T}}\pa_{x}\vp \lrs{\pa_{x}^{n-1} \int_{\R} \xi^{l-n} f_{\hbar,\ve}d\xi}^{2}  dxdt+\mathcal{O}(\hbar)\notag\\
\leq& \n{\pa_{x}\vp}_{L_{t}^{1}L_{x}^{\infty}}\bbn{\hbar^{n}\pa_{x}^{n-1} \int_{\R} \xi^{l-n} f_{\hbar,\ve}d\xi}_{L_{t}^{\infty}L_{x}^{1}}^{2}
+\mathcal{O}(\hbar)\notag\\
\leq& \mathcal{O}(\hbar),\notag
\end{align}
where in the last inequality we have used the uniform estimate \eqref{equ:higer,momentum of wigner,Lmax}.
Putting together the estimates \eqref{equ:remainder term,I1,I2}, \eqref{equ:remainder term,I2}, and \eqref{equ:remainder term,I1}, we thus prove
\eqref{equ:remainder,estimate}--\eqref{equ:remainder,error} for the case $k=2l+1$.

When $k=2l$, $\al=2n+1$, $l\geq n\geq 1$, repeating the proof of the case $k=2l+1$, we also have
\begin{align*}
&M_{\vp}^{(k,\al,0)}\\
=&M_{\vp}^{(2l,2n+1,l-n)}+\mathcal{O}(\hbar)\\
=&\hbar^{\al-1}\int_{\Omega_{T}}\vp \lrs{\pa_{x}^{2n-1} \int_{\R} \xi^{l-n} f_{\hbar,\ve}d\xi}  \lrs{ \int_{\R} \xi^{l-n-1} f_{\hbar,\ve}d\xi} dxdt+\mathcal{O}(\hbar)\\
=&(-1)\hbar^{\al-1}\int_{\Omega_{T}}\vp \lrs{\pa_{x}^{2n-2} \int_{\R} \xi^{l-n} f_{\hbar,\ve}d\xi}  \lrs{ \pa_{x}\int_{\R} \xi^{l-n-1} f_{\hbar,\ve}d\xi} dxdt+\mathcal{O}(\hbar)\\
=&...\\
=&(-1)^{n-1}\hbar^{\al-1}\int_{\Omega_{T}}\vp \lrs{\pa_{x}^{n} \int_{\R} \xi^{l-n} f_{\hbar,\ve}d\xi}  \lrs{ \pa_{x}^{n-1}\int_{\R} \xi^{l-n-1} f_{\hbar,\ve}d\xi} dxdt+\mathcal{O}(\hbar).
\end{align*}
By the moment equation \eqref{equ:momentum,f}, in a similar way in which we obtain \eqref{equ:M,k,j,expand}, we have
\begin{align*}
M_{\vp}^{(k,\al,0)}
=&(-1)^{n}\frac{\hbar^{\al-1}}{2}\int_{\Omega_{T}}\vp \pa_{t}\lrs{\pa_{x}^{n-1} \int_{\R} \xi^{l-n-1} f_{\hbar,\ve}d\xi}^{2}  dxdt+\mathcal{O}(\hbar)\\
=&(-1)^{n-1}\frac{\hbar^{\al-1}}{2}\int_{\Omega_{T}}\pa_{t}\vp \lrs{\pa_{x}^{n-1} \int_{\R} \xi^{l-n-1} f_{\hbar,\ve}d\xi}^{2}  dxdt+\mathcal{O}(\hbar)\\
\leq& \n{\pa_{t}\vp}_{L_{t}^{1}L_{x}^{\infty}}\bbn{\hbar^{n}\pa_{x}^{n-1} \int_{\R} \xi^{l-n-1} f_{\hbar,\ve}d\xi}_{L_{t}^{\infty}L_{x}^{1}}^{2}
+\mathcal{O}(\hbar)\notag\\
\leq& \mathcal{O}(\hbar).
\end{align*}
Therefore, putting together the estimates \eqref{equ:remainder term,I1,I2}, \eqref{equ:remainder term,I2}, and \eqref{equ:remainder term,I1}, we have completed the proof of \eqref{equ:remainder,estimate}--\eqref{equ:remainder,error} for the case $k=2l$.

Doing the summation over $\al$ in \eqref{equ:remainder,estimate}, we immediately obtain \eqref{equ:rate,remainder}.
For \eqref{equ:convergence,remainder}, we apply an approximation argument and rewrite
\begin{align*}
\int_{\Omega_{T}}\vp \mathrm{R}_{\hbar,\ve}^{(k)} dxdt=B_{1}+B_{2}+B_{3},
\end{align*}
where
\begin{align*}
B_{1}=&\int_{\Omega_{T}}\vp \lrs{1-\chi(\frac{x}{R})} \mathrm{R}_{\hbar,\ve}^{(k)} dxdt,\\
B_{2}=&\int_{\Omega_{T}}\lrs{\vp \chi(\frac{x}{R})-\vp_{n,R}} \mathrm{R}_{\hbar,\ve}^{(k)} dxdt,\\
B_{3}=&\int_{\Omega_{T}}\vp_{n,R} \mathrm{R}_{\hbar,\ve}^{(k)} dxdt,
\end{align*}
and
\begin{align*}
\vp_{n,R}\in C_{c}^{\infty}(\Omega_{T}),\quad \lim_{n\to \infty}\bn{\vp \chi(\frac{x}{R})-\vp_{n,R}}_{L_{t}^{1}([0,T];L_{x}^{\infty})}=0.
\end{align*}
By the weighted uniform estimate \eqref{equ:uniform estimate,remainder} on the remainder term, we have
\begin{align*}
|B_{1}|\leq& \frac{1}{R}\n{\vp}_{L_{t}^{1}([0,T];L_{x}^{\infty})}\n{\lra{x}\mathrm{R}_{\hbar,\ve}^{(k)}}_{L_{t}^{\infty}([0,T];L_{x}^{1})}\lesssim \frac{1}{R},\\
|B_{2}|\leq& \bn{\vp \chi(\frac{x}{R})-\vp_{n,R}}_{L_{t}^{1}([0,T];L_{x}^{\infty})}
\n{\mathrm{R}_{\hbar,\ve}^{(k)}}_{L_{t}^{\infty}([0,T];L_{x}^{1})}\lesssim\bn{\vp \chi(\frac{x}{R})-\vp_{n,R}}_{L_{t}^{1}([0,T];L_{x}^{\infty})}.
\end{align*}
Hence, using \eqref{equ:rate,remainder} for $\vp_{n,R}\in C_{c}^{\infty}(\Omega_{T})$, we obtain
\begin{align*}
&\lim_{(\hbar,\ve)\to(0,0)}\bbabs{\int_{\Omega_{T}}\vp \mathrm{R}_{\hbar,\ve}^{(k)} dxdt}\\
=& \lim_{R\to \infty}\lim_{n\to \infty}\lim_{(\hbar,\ve)\to(0,0)}\lrs{B_{1}+B_{2}+B_{3}}\\
\lesssim &\lim_{R\to \infty}\lim_{n\to \infty} \lrs{\frac{1}{R}+\bn{\vp \chi(\frac{x}{R})-\vp_{n,R}}_{L_{t}^{1}([0,T];L_{x}^{\infty})}}=0,
\end{align*}
which completes the proof of \eqref{equ:convergence,remainder}.

\end{proof}

\subsection{Convergence of the Nonlinear Term for $k=1,2$}\label{section:Convergence of the Nonlinear Term,k=1,2}
As a preliminary part, we prove the convergence of the nonlinear term for the $k=1,2$ case based on the weighted uniform estimates in Section \ref{sec:Uniform Estimates}.
We first provide estimates on $E_{\hbar,\ve}(t,x)$ and study its limit.
\begin{lemma}
There holds that
\begin{align}
\pa_{x}E_{\hbar,\ve}=&\rho_{\hbar,\ve}-\pa_{x}^{2}U_{\ve}*\rho_{\hbar,\ve},\label{equ:equation,E,h,x}\\
\pa_{t}E_{\hbar,\ve}=&-\int_{\R}\xi f_{\hbar,\ve}d\xi+\pa_{x}^{2}U_{\ve}*\int_{\R}\xi f_{\hbar,\ve}d\xi.\label{equ:equation,E,h,t}
\end{align}
where $U_{\ve}(x)=\frac{1}{2}|x|-V_{\ve}(x)$.
We have the uniform estimates on $E_{\hbar,\ve}(t,x)$ that
\begin{align}
\n{E_{\hbar,\ve}}_{L_{t}^{\infty}([0,T];L_{x}^{\infty})}\leq& C(T),\label{equ:uniform bound,E,h}\\
\n{\pa_{x}E_{\hbar,\ve}}_{L_{t}^{\infty}([0,T];L_{x}^{1})}\leq& C(T),\label{equ:uniform bound,E,h,x}\\
\n{\pa_{t}E_{\hbar,\ve}}_{L_{t}^{\infty}([0,T];L_{x}^{1})}\leq& C(T).\label{equ:uniform bound,E,h,t}
\end{align}
Moreover, for $p\in [1,\infty)$ we have the strong convergence that
\begin{align}\label{equ:strong convergence,E,h}
E_{\hbar,\ve}(t,x)\to E(t,x):=\frac{1}{2}\int_{\R} \frac{x-y}{|x-y|} \lrs{\int_{\R}f(t,y,\xi)d\xi}dy,\quad L_{loc}^{p}(\Omega_{T}),
\end{align}
where
\begin{align}\label{equ:BV,E}
E(t,x)\in BV\cap L^{\infty}(\Omega_{T}).
\end{align}
Moreover, the limit function $E(t,x)$ satisfies
\begin{align}\label{equ:limit,E,x,t}
\pa_{x}E=&\int_{\R}fd\xi,\quad \pa_{t}E=\int_{\R}\xi fd\xi,
\end{align}
in the sense of measures.
\end{lemma}
\begin{proof}
Equations \eqref{equ:equation,E,h,x} and \eqref{equ:equation,E,h,t} follow from a direct calculation using the moment equation \eqref{equ:momentum,f}. Estimate \eqref{equ:uniform bound,E,h} follows from the uniform estimate \eqref{equ:weighted,eletron part}.
For \eqref{equ:uniform bound,E,h,x}, noting that
\begin{align*}
\pa_{x}^{2}V_{\ve}=\delta(x)e^{-\ve|x|}-2\ve e^{-\ve|x|}+\ve^{2}|x|e^{-\ve |x|},
\end{align*}
we use Young's inequality to get
\begin{align*}
\n{\pa_{x}E_{\hbar,\ve}}_{L_{t}^{\infty}([0,T];L_{x}^{1})}=&\n{(\pa_{x}^{2}V_{\ve})*\rho_{\hbar,\ve}}_{L_{t}^{\infty}([0,T];L_{x}^{1})}\\
\leq& \n{\rho_{\hbar,\ve}}_{L_{t}^{\infty}([0,T];L_{x}^{1})}\lrs{1+\n{\ve e^{-\ve|x|}}_{L_{x}^{1}}+\n{\ve^{2}|x|e^{-\ve|x|}}_{L_{x}^{1}}}\leq C(T).
\end{align*}
For \eqref{equ:uniform bound,E,h,t}, by the moment equation \eqref{equ:momentum,f} we rewrite
\begin{align*}
\pa_{t}E_{\hbar,\ve}=V_{\ve}*\pa_{t}\rho_{\hbar,\ve}=-\pa_{x}^{2}V_{\ve}*\int_{\R}\xi f_{\hbar,\ve}d\xi.
\end{align*}
Via the same way in which we obtain \eqref{equ:uniform bound,E,h,x}, we arrive at \eqref{equ:uniform bound,E,h,t}.

By the uniform estimates \eqref{equ:uniform bound,E,h}--\eqref{equ:uniform bound,E,h,t} and $L^{p}$ compactness criteria, there is a subsequence of $\lr{E_{\hbar,\ve}(t,x)}$, which we still denote by $\lr{E_{\hbar,\ve}(t,x)}$, and some function
$$E(t,x)\in BV\cap L^{\infty}(\Omega_{T}),$$
such that
\begin{align*}
E_{\hbar,\ve}(t,x)\stackrel{L_{loc}^{p}(\Omega_{T})}{\longrightarrow} E(t,x),\quad p\in [1,\infty).
\end{align*}
To obtain the explicit formula of $E(t,x)$, we consider
\begin{align*}
\int_{\Omega_{T}}E_{\hbar,\ve} \vp dx dt=&\int_{\Omega_{T}}\frac{1}{2}\lrs{\frac{x}{|x|}*\rho_{\hbar,\ve}} \vp dx dt+\int_{\Omega_{T}}\lrs{\pa_{x}U_{\ve}*\rho_{\hbar,\ve}} \vp dx dt\\
=&\int_{\Omega_{T}}\frac{1}{2}\rho_{\hbar,\ve} \lrs{\frac{x}{|x|}*\vp} dx dt+\int_{\Omega_{T}}\lrs{\pa_{x}U_{\ve}*\rho_{\hbar,\ve}} \vp dx dt.
\end{align*}
On the one hand, by the pointwise estimate \eqref{equ:U,pointwise estimate} that
$|\pa_{x}U_{\ve}*\rho_{\hbar,\ve}|\lesssim \ve \lra{x}$ and the weighted uniform estimate \eqref{equ:weighted estimate,mass},
we have
\begin{align*}
\bbabs{\int_{\Omega_{T}}\lrs{\pa_{x}U_{\ve}*\rho_{\hbar,\ve}} \vp dx dt}\lesssim \ve\n{\lra{x}\rho_{\hbar,\ve}}_{L_{t}^{\infty}L_{x}^{1}} \n{\lra{x}\vp}_{L_{t}^{1}L_{x}^{1}}\to 0.
\end{align*}
On the other hand, due to the fact that $\frac{x}{|x|}*\vp\in L_{t}^{1}([0,T];C_{b}(\R))$, we use the narrow convergence \eqref{equ:convergence,momentum,f,h} and hence obtain
\begin{align*}
\lim_{(\hbar,\ve)\to(0,0)}\int_{\Omega_{T}}E_{\hbar,\ve} \vp dx dt=\frac{1}{2}\int_{\Omega_{T}}\lrs{\int_{\R}f(t,x,\xi)d\xi} \lrs{\frac{x}{|x|}*\vp} dx dt,
\end{align*}
which implies formula \eqref{equ:strong convergence,E,h}. In the same manner, we also attain \eqref{equ:limit,E,x,t} and hence complete the proof.

\end{proof}

Now, we are able to prove the following convergence.
\begin{lemma}\label{lemma:convergence,nonlinear,momentum,k=1,2}
For $\vp\in L_{t}^{1}([0,T];C_{b}(\R))$, we have
\begin{align}
\lim_{(\hbar,\ve)\to(0,0)}\int_{\Omega_{T}} \vp  E_{\hbar,\ve} \lrs{\int_{\R} f_{\hbar,\ve}d\xi}dxdt=&\int_{\Omega_{T}} \vp \ol{E}\pa_{x}Edxdt,\label{equ:convergence,nonlinear,momentum,k=1}\\
\lim_{(\hbar,\ve)\to(0,0)} \int_{\Omega_{T}} \vp  E_{\hbar,\ve} \lrs{\int_{\R}\xi f_{\hbar,\ve}d\xi}dxdt=&\int_{\Omega_{T}} \vp \ol{E}\pa_{t}Edxdt.
\label{equ:convergence,nonlinear,momentum,k=2}
\end{align}
\end{lemma}
\begin{proof}
It suffices to prove \eqref{equ:convergence,nonlinear,momentum,k=1}, as \eqref{equ:convergence,nonlinear,momentum,k=2} follows similarly.
First, we prove that \eqref{equ:convergence,nonlinear,momentum,k=1} holds for $\vp\in C_{c}^{\infty}(\Omega_{T})$. By \eqref{equ:equation,E,h,x}, we rewrite
\begin{align*}
&\int_{\Omega_{T}} \vp  E_{\hbar,\ve} \lrs{\int_{\R} f_{\hbar,\ve}d\xi}dxdt\\
=&\int_{\Omega_{T}} \vp  E_{\hbar,\ve} \lrs{\pa_{x}E_{\hbar,\ve}} dxdt+
\int_{\Omega_{T}} \vp  E_{\hbar,\ve}  \lrs{\pa_{x}^{2}U_{\ve}*\rho_{\hbar,\ve}} dxdt\\
:=&A_{\hbar,\ve}+B_{\hbar,\ve}.
\end{align*}

For the first term $A_{\hbar,\ve}$, by property $(1)$ of BV functions in Appendix \ref{section:Basic Properties of BV Functions}, we have
\begin{align*}
\lim_{(\hbar,\ve)\to (0,0)}A_{\hbar,\ve}=&-\lim_{(\hbar,\ve)\to (0,0)}\frac{1}{2}\int_{\Omega_{T}} (\pa_{x}\vp)  \lrs{E_{\hbar,\ve}}^{2} dxdt\\
=&-\frac{1}{2}\int_{\Omega_{T}} (\pa_{x}\vp)  \lrs{E}^{2} dxdt\\
=&\int_{\Omega_{T}} \vp  \ol{E}\lrs{\pa_{x}E} dxdt.
\end{align*}

For the second term $B_{\hbar,\ve}$,
\begin{align*}
B_{\hbar,\ve}=&\int_{\Omega_{T}} \vp  E_{\hbar,\ve}  \lrs{\pa_{x}^{2}U_{\ve}*\rho_{\hbar,\ve}} dxdt\\
=&-\int_{\Omega_{T}} \lrs{\pa_{x}\vp}  E_{\hbar,\ve}  \lrs{\pa_{x}U_{\ve}*\rho_{\hbar,\ve}} dxdt
-\int_{\Omega_{T}} \vp  \lrs{\pa_{x}E_{\hbar,\ve}}  \lrs{\pa_{x}U_{\ve}*\rho_{\hbar,\ve}} dxdt\\
\leq& \ve \n{\lra{x}\pa_{x}\vp}_{L_{x}^{1}}\n{E_{\hbar,\ve}}_{L_{x}^{\infty}} +\ve \n{\lra{x}\vp}_{L_{x}^{\infty}}
\n{\pa_{x}E_{\hbar,\ve}}_{L_{x}^{1}}\to 0.
\end{align*}
Hence, we complete the proof of \eqref{equ:convergence,nonlinear,momentum,k=1} for $\vp\in C_{c}^{\infty}(\Omega_{T})$. Furthermore,
by the uniform bound \eqref{equ:uniform bound,E,h} and the weighted uniform estimate \eqref{equ:weighted,momentum of wigner}, we get the weighted estimate that
\begin{align}\label{equ:uniform estimate,product form}
\bbn{\lra{x}E_{\hbar,\ve} \lrs{\int_{\R} f_{\hbar,\ve}d\xi}}_{L_{t}^{\infty}L_{x}^{1}}\leq\n{E_{\hbar,\ve}}_{L_{t}^{\infty}L_{x}^{\infty}}
\bbn{\lra{x}\int_{\R} f_{\hbar,\ve}d\xi}_{L_{t}^{\infty}L_{x}^{1}}\leq C(T).
\end{align}
Via the same way in which we obtain \eqref{equ:convergence,remainder} by an approximation argument, we arrive at \eqref{equ:convergence,nonlinear,momentum,k=1} for $\vp\in L_{t}^{1}([0,T];C_{b}(\R))$ and hence complete the proof.

\end{proof}

\subsection{Convergence of the Nonlinear Term for $k\geq 3$}\label{section:Convergence of the Nonlinear Term,k>=3}
In this section, we prove the convergence of the nonlinear term for the general $k\geq 3$ case.
\begin{lemma}\label{lemma:convergence,nonlinear,momentum,k}
Let $T>0$ and $k\geq 3$. For $\vp\in L_{t}^{1}([0,T];C_{b}(\R))$, there holds that
\begin{align}\label{equ:convergence,nonlinear,momentum}
\lim _{(\hbar,\ve)\to (0,0)} \int_{\Omega_T} \vp E_{\hbar,\ve} \lrs{\int_{\mathbb{R}} \xi^{k-1} f_{\hbar,\ve} d \xi} d x d t=\int_{\Omega_T} \vp \ol{E} \lrs{\int_{\mathbb{R}} \xi^{k-1} f(t,dx,d\xi)} dt.
\end{align}
\end{lemma}
As we have proven the base $k=1,2$ case in Lemma \ref{lemma:convergence,nonlinear,momentum,k=1,2},
we take an induction argument to prove Lemma \ref{lemma:convergence,nonlinear,momentum,k}, whose proof is postponed to the end of the section.

\textbf{Induction hypothesis}:
For $l\leq k-1$, $\vp\in L_{t}^{1}([0,T];C_{b}(\R))$, there holds that
\begin{align}\label{equ:induction hypothesis,l}
\lim_{(\hbar,\ve)\to(0,0)} \int_{\Omega_{T}} \vp  E_{\hbar,\ve} \lrs{\int_{\R}\xi^{l-1} f_{\hbar,\ve}d\xi}dxdt=\int_{\Omega_{T}} \vp \ol{E}\lrs{\int_{\R}\xi^{l-1} f d\xi}dxdt.
\end{align}

Before getting into the proof,
we consider the integral function of the moment function that
\begin{align}
M_{\hbar,\ve}^{(m)}(t,x):=\int_{-\infty}^{x}\int_{\R}\xi^{m}f_{\hbar,\ve}(t,y,\xi)d\xi dy,
\end{align}
which plays a similar role as $E_{\hbar,\ve}(t,x)$. We set up the uniform estimates for $M_{\hbar,\ve}^{(m)}(t,x)$ and study its limit function, which is important to the convergence of the nonlinear term.
\begin{lemma}
Let $0\leq m\leq k-1$.
The function $M_{\hbar,\ve}^{(m)}(t,x)$ satisfies
\begin{align}\label{equ:M,m,h,equation}
\pa_{t}M_{\hbar,\ve}^{(m)}
+\pa_{x}M_{\hbar,\ve}^{(m+1)}+m\int_{-\infty}^{x}E_{\hbar,\ve}\int_{\R}\xi^{m-1}f_{\hbar,\ve}d\xi dy+\int_{-\infty}^{x}\mathrm{R}_{\hbar,\ve}^{(m)}dy=0,
\end{align}
and enjoys the uniform estimates that
\begin{align}
&\n{M_{\hbar,\ve}^{(m)}}_{L_{t}^{\infty}([0,T];L_{x}^{\infty})}\leq C(T),\label{equ:uniform bound,M,h}\\
&\n{\pa_{x}M_{\hbar,\ve}^{(m)}}_{L_{t}^{\infty}([0,T];L_{x}^{1})}\leq C(m,T),\label{equ:uniform bound,M,h,x}\\
&\n{\pa_{t}M_{\hbar,\ve}^{(m)}}_{L_{t}^{\infty}([0,T];L_{x}^{1})}\leq C(m,T).\label{equ:uniform bound,M,h,t}
\end{align}

Moreover, for $p\in [1,\infty)$ we have the strong convergence that
\begin{align}\label{equ:strong convergence,M,h}
M_{\hbar,\ve}^{(m)}(t,x)\to M^{(m)}(t,x):=\int_{-\infty}^{x}\int_{\R}\xi^{m}f(t,y,\xi)d\xi dy, \quad \text{in $L_{loc}^{p}(\Omega_{T})$}.
\end{align}
Finally, under the induction hypothesis \eqref{equ:induction hypothesis,l}, the limit function satisfies
\begin{align}\label{equ:M,equation}
\pa_{t}M^{(m)}+\pa_{x}M^{(m+1)}+
m\int_{-\infty}^{x}\ol{E}\int_{\R}\xi^{m-1}fd\xi dy=0,
\end{align}
in the sense of measures.

\end{lemma}
\begin{proof}
Equation \eqref{equ:M,m,h,equation} follows from the moment equation \eqref{equ:momentum,f}.
By the weighted uniform estimates \eqref{equ:weighted,momentum of wigner}--\eqref{equ:weighted,eletron part} and the uniform bound \eqref{equ:uniform estimate,remainder} on the remainder term, we have \eqref{equ:uniform bound,M,h}--\eqref{equ:uniform bound,M,h,t}.
In the same way in which we obtain \eqref{equ:strong convergence,E,h}, we get \eqref{equ:strong convergence,M,h}.

Next, we prove \eqref{equ:M,equation}. For the linear part, we have
\begin{align*}
&\lim_{(\hbar,\ve)\to (0,0)}\int_{\Omega_{T}}\lrs{\pa_{t}M_{\hbar,\ve}^{(m)}+\pa_{x}M_{\hbar,\ve}^{(m+1)}}\vp dx dt=\int_{\Omega_{T}}-M^{(m)}\pa_{t}\vp -M^{(m+1)}\pa_{x}\vp dx dt.
\end{align*}
For the nonlinear part, we use the induction hypothesis to get
\begin{align*}
&\lim_{(\hbar,\ve)\to(0,0)}m\int_{\Omega_{T}}\lrs{\int_{-\infty}^{x}E_{\hbar,\ve}\int_{\R}\xi^{m-1}f_{\hbar,\ve}d\xi dy}\vp dxdt\\
=&\lim_{(\hbar,\ve)\to(0,0)}m\int_{\Omega_{T}}\lrs{E_{\hbar,\ve}\int_{\R}\xi^{m-1}f_{\hbar,\ve}d\xi }\lrs{\int_{y}^{\infty}\vp dx} dydt\\
=&m\int_{\Omega_{T}}\lrs{\ol{E}\int_{\R}\xi^{m-1}fd\xi }\lrs{\int_{y}^{\infty}\vp dx} dydt\\
=&\int_{\Omega_{T}}\lrs{\int_{-\infty}^{x}\ol{E}\int_{\R}\xi^{m-1}fd\xi dy}\vp dxdt.
\end{align*}
For the remainder term, due to the fact that $\int_{y}^{\infty}\vp dx\in L_{t}^{1}([0,T];C_{b}(\R))$, we can use \eqref{equ:convergence,remainder} in Lemma \ref{lemma:remainder,vanishing} to get
\begin{align*}
\lim_{(\hbar,\ve)\to(0,0)}\int_{\Omega_{T}}\lrs{\int_{-\infty}^{x}\mathrm{R}_{\hbar,\ve}^{(m)}dy}\vp dxdt
=&\lim_{(\hbar,\ve)\to(0,0)}\int_{\Omega_{T}}\mathrm{R}_{\hbar,\ve}^{(m)} \lrs{\int_{y}^{\infty}\vp dx} dydt=0.
\end{align*}
Hence, by formula \eqref{equ:M,m,h,equation}, we complete the proof of \eqref{equ:M,equation}.
\end{proof}

The following lemma shows that the limit function $M^{(m)}$ satisfies an induction equation. This is the key to reduce the order of the weight function $\xi^{k}$ so that one can make use of the induction hypothesis.
\begin{lemma}\label{lemma:induction equation,limit case}
Let $0\leq j\leq k-1$, $0\leq m\leq k-1$. Under the induction hypothesis \eqref{equ:induction hypothesis,l}, for $\vp\in C_{c}^{\infty}(\Omega_{T})$, we have
\begin{align}\label{equ:induction equation,M}
&\int_{\Omega_{T}}\vp \ol{M}^{(j)}
\lrs{\pa_{x}M^{(m+1)}}dxdt\\
=&\int_{\Omega_{T}}\vp \ol{M}^{(j+1)}
\lrs{ \pa_{x}M^{(m)}}dxdt+I_{1}^{(j,m)}+I_{2}^{(j,m)}+I_{3}^{(j,m)}+I_{4}^{(j,m)},\notag
\end{align}
where $\ol{M}^{(j)}$ is the Vol$^\prime$pert's symmetric average defined in \eqref{equ:volpert average} and
\begin{align*}
I_{1}^{(j,m)}=&\int_{\Omega_{T}} (\pa_{t}\vp)M^{(j)}M^{(m)}dxdt,\\
I_{2}^{(j,m)}=&-m\int_{\Omega_{T}} \vp\lrs{\int_{-\infty}^{x}\ol{E}\int_{\R}\xi^{j-1}fd\xi dy}M^{(m)}dxdt,\\
I_{3}^{(j,m)}=&\int_{\Omega_{T}}\lrs{\pa_{x}\vp} M^{(j+1)}
M^{(m)}dxdt,\\
I_{4}^{(j,m)}=&-j\int_{\Omega_{T}} \vp M^{(j)} \lrs{\int_{-\infty}^{x}\ol{E}\int_{\R}\xi^{m-1}fd\xi dy}dxdt.
\end{align*}
\end{lemma}
\begin{proof}
We consider the test function of the form
$$\lrs{\vp M^{(j)}}*\eta_{\sigma}\in C_{c}^{\infty}(\Omega_{T}),$$
where $\vp\in C_{c}^{\infty}(\Omega_{T})$ and $\eta_{\sigma}(t,x)=\sigma^{-2}\eta(t/\sigma,x/\sigma)$ is a smooth mollifier and approximation of the identity.
Putting the test function into the limit equation \eqref{equ:M,equation}, we obtain
\begin{align*}
&\int_{\Omega_{T}}\lrs{ \lrs{\vp M^{(j)}}*\eta_{\sigma}} \lrs{\pa_{x}M^{(m+1)}} dxdt\\
=&\int_{\Omega_{T}} \lrs{\pa_{t}\lrs{\vp M^{(j)}}*\eta_{\sigma}}M^{(m)} dxdt\\
&-m\int_{\Omega_{T}} \lrs{ \lrs{\vp M^{(j)}}*\eta_{\sigma}} \lrs{\int_{-\infty}^{x}\ol{E}\int_{\R}\xi^{m-1}fd\xi dy}dxdt\\
=&
\int_{\Omega_{T}} \vp \lrs{\pa_{t}M^{(j)}} (M^{(m)}*\eta_{\sigma})dxdt+\int_{\Omega_{T}} \lrs{\pa_{t}\vp} M^{(j)} (M^{(m)}*\eta_{\sigma})dxdt\\
&-m\int_{\Omega_{T}} \lrs{ \lrs{\vp M^{(j)}}*\eta_{\sigma}} \lrs{\int_{-\infty}^{x}\ol{E}\int_{\R}\xi^{m-1}fd\xi dy}dxdt\\
:=&\int_{\Omega_{T}} \vp \lrs{\pa_{t}M^{(j)}} (M^{(m)}*\eta_{\sigma})dxdt+I_{1,\sigma}^{(j,m)}+I_{2,\sigma}^{(j,m)},
\end{align*}
where
\begin{align*}
I_{1,\sigma}^{(j,m)}=&\int_{\Omega_{T}} \lrs{\pa_{t}\vp} M^{(j)} (M^{(m)}*\eta_{\sigma})dxdt,\\
I_{2,\sigma}^{(j,m)}=&-m\int_{\Omega_{T}} \lrs{ \lrs{\vp M^{(j)}}*\eta_{\sigma}} \lrs{\int_{-\infty}^{x}\ol{E}\int_{\R}\xi^{m-1}fd\xi dy}dxdt.
\end{align*}
Using again \eqref{equ:M,equation} for $\pa_{t}M^{(j)}$, we expand
\begin{align*}
&\int_{\Omega_{T}} \vp \lrs{\pa_{t}M^{(j)}} (M^{(m)}*\eta_{\sigma})dxdt\\
=&-\int_{\Omega_{T}} \vp \lrs{\pa_{x}M^{(j+1)}} (M^{(m)}*\eta_{\sigma})dxdt\\
&-j\int_{\Omega_{T}} \vp \lrs{\int_{-\infty}^{x}\ol{E}\int_{\R}\xi^{m-1}fd\xi dy} (M^{(m)}*\eta_{\sigma})dxdt\\
=&\int_{\Omega_{T}} \vp M^{(j+1)} (\pa_{x}M^{(m)}*\eta_{\sigma})dxdt+\int_{\Omega_{T}} \lrs{\pa_{x}\vp} M^{(j+1)} (M^{(m)}*\eta_{\sigma})dxdt\\
&-j\int_{\Omega_{T}} \vp \lrs{\int_{-\infty}^{x}\ol{E}\int_{\R}\xi^{m-1}fd\xi dy} (M^{(m)}*\eta_{\sigma})dxdt\\
:=&\int_{\Omega_{T}} \lrs{\lrs{\vp M^{(j+1)}}*\eta_{\sigma}} \pa_{x}M^{(m)}dxdt+I_{3,\sigma}^{(j,m)}+I_{4,\sigma}^{(j,m)},
\end{align*}
where
\begin{align*}
I_{3,\sigma}^{(j,m)}=&\int_{\Omega_{T}} \lrs{\pa_{x}\vp} M^{(j+1)} (M^{(m)}*\eta_{\sigma})dxdt,\\
I_{4,\sigma}^{(j,m)}=&-j\int_{\Omega_{T}} \vp \lrs{\int_{-\infty}^{x}\ol{E}\int_{\R}\xi^{m-1}fd\xi dy} (M^{(m)}*\eta_{\sigma})dxdt.
\end{align*}
Therefore, we arrive at
\begin{align}\label{equ:induction equation,M,sigma}
&\int_{\Omega_{T}}\lrs{ \lrs{\vp M^{(j)}}*\eta_{\sigma}} \lrs{\pa_{x}M^{(m+1)}} dxdt\\
=&\int_{\Omega_{T}} \lrs{\lrs{\vp M^{(j+1)}}*\eta_{\sigma}} \pa_{x}M^{(m)}dxdt+I_{1,\sigma}^{(j,m)}
+I_{2,\sigma}^{(j,m)}+I_{3,\sigma}^{(j,m)}+I_{4,\sigma}^{(j,m)}.\notag
\end{align}
By the dominated convergence theorem, we have
\begin{align*}
\lim_{\sigma\to 0}I_{i,\sigma}^{(j,m)}=I_{i}^{(j,m)},\quad i=1,2,3,4.
\end{align*}
By the properties (5)-(6) of BV functions at the Appendix \ref{section:Basic Properties of BV Functions}, we have
\begin{align*}
\lim_{\sigma\to 0}\int_{\Omega_{T}}\lrs{ \lrs{\vp M^{(j)}}*\eta_{\sigma}} \lrs{\pa_{x}M^{(m+1)}} dxdt=&\int_{\Omega_{T}}\vp \ol{M}^{(j)}
\lrs{\pa_{x}M^{(m+1)}}dxdt,\\
\lim_{\sigma\to 0}\int_{\Omega_{T}}\lrs{ \lrs{\vp M^{(j+1)}}*\eta_{\sigma}} \lrs{\pa_{x}M^{(m)}} dxdt=&\int_{\Omega_{T}}\vp \ol{M}^{(j+1)}
\lrs{\pa_{x}M^{(m)}}dxdt.
\end{align*}
Sending $\sigma\to 0$ in \eqref{equ:induction equation,M,sigma}, we complete the proof of \eqref{equ:induction equation,M}.

\end{proof}
Next, we prove that
the function $M_{\hbar,\ve}^{(m)}(t,x)$, which is similar to its limit function $M^{(m)}(t,x)$, also has an induction structure.
\begin{lemma}\label{lemma:induction equation,M,h}
Let $0\leq j\leq k-1$, $0\leq m\leq k-1$. Under the induction hypothesis \eqref{equ:induction hypothesis,l}, we have
\begin{align}\label{equ:induction equation,M,h}
&\lim_{(\hbar,\ve)\to(0,0)}\int_{\Omega_{T}}\vp M_{\hbar,\ve}^{(j)}
 \lrs{\pa_{x}M_{\hbar,\ve}^{(m+1)}}dxdt\\
=&\lim_{(\hbar,\ve)\to(0,0)}\int_{\Omega_{T}}\vp M_{\hbar,\ve}^{(j+1)}
 \lrs{\pa_{x}M_{\hbar,\ve}^{(m)}}dxdt+I_{1}^{(j,m)}+I_{2}^{(j,m)}+I_{3}^{(j,m)}+I_{4}^{(j,m)}.\notag
\end{align}

\end{lemma}
\begin{proof}
Using equation \eqref{equ:M,m,h,equation} for $\pa_{x}M_{\hbar,\ve}^{(m+1)}$, we get
\begin{align*}
\int_{\Omega_{T}}\vp M_{\hbar,\ve}^{(j)}
\lrs{ \pa_{x}M_{\hbar,\ve}^{(m+1)} }dxdt
=&A_{\hbar,\ve,1}^{(j,m)}+A_{\hbar,\ve,2}^{(j,m)}+A_{\hbar,\ve,3}^{(j,m)},
\end{align*}
where
\begin{align*}
A_{\hbar,\ve,1}^{(j,m)}=&-\int_{\Omega_{T}}\vp M_{\hbar,\ve}^{(j)}
\lrs{\pa_{t}M_{\hbar,\ve}^{(m)}},\\
A_{\hbar,\ve,2}^{(j,m)}=&-m\int_{\Omega_{T}}\vp M_{\hbar,\ve}^{(j)}
\lrs{\int_{-\infty}^{x}E_{\hbar,\ve}\int_{\R}\xi^{m-1}f_{\hbar,\ve}d\xi dy}dxdt,\\
A_{\hbar,\ve,3}^{(j,m)}=&-\int_{\Omega_{T}}\vp M_{\hbar,\ve}^{(j)}
\lrs{\int_{-\infty}^{x} \mathrm{R}_{\hbar,\ve}^{(m)}(y)dy}dxdt.
\end{align*}
Using again equation \eqref{equ:M,m,h,equation} for $\pa_{t}M_{\hbar,\ve}^{(j)}$, we expand
\begin{align*}
A_{\hbar,\ve,1}^{(j,m)}=&\int_{\Omega_{T}}(\pa_{t}\vp) M_{\hbar,\ve}^{(j)}
M_{\hbar,\ve}^{(m)}dxdt+\int_{\Omega_{T}}\vp \lrs{\pa_{t}M_{\hbar,\ve}^{(j)}}
M_{\hbar,\ve}^{(m)}dxdt\\
=&A_{\hbar,\ve,10}^{(j,m)}+A_{\hbar,\ve,11}^{(j,m)}+A_{\hbar,\ve,12}^{(j,m)}+
A_{\hbar,\ve,13}^{(j,m)}+A_{\hbar,\ve,14}^{(j,m)},
\end{align*}
where
\begin{align*}
A_{\hbar,\ve,10}^{(j,m)}=&\int_{\Omega_{T}}(\pa_{t}\vp) M_{\hbar,\ve}^{(j)}
M_{\hbar,\ve}^{(m)}dxdt,\\
A_{\hbar,\ve,11}^{(j,m)}=&
\int_{\Omega_{T}}\vp  M_{\hbar,\ve}^{(j+1)}
\lrs{\pa_{x}M_{\hbar,\ve}^{(m)}}dxdt,\\
A_{\hbar,\ve,12}^{(j,m)}=&
\int_{\Omega_{T}}\lrs{\pa_{x}\vp}  M_{\hbar,\ve}^{(j+1)}
M_{\hbar,\ve}^{(m)}dxdt,\\
A_{\hbar,\ve,13}^{(j,m)}=&-j
\int_{\Omega_{T}}\vp \lrs{\int_{-\infty}^{x}E_{\hbar,\ve}\int_{\R}\xi^{j-1}f_{\hbar,\ve}d\xi dy}
M_{\hbar,\ve}^{(m)}dxdt,\\
A_{\hbar,\ve,14}^{(j,m)}=&-
\int_{\Omega_{T}}\vp \lrs{\int_{-\infty}^{x}\mathrm{R}_{\hbar,\ve}^{(m)}dy}
M_{\hbar,\ve}^{(m)}dxdt.
\end{align*}

Therefore, we arrive at
\begin{align*}
&\int_{\Omega_{T}}\vp M_{\hbar,\ve}^{(j)}
 \lrs{\pa_{x}M_{\hbar,\ve}^{(m+1)}}dxdt\\
=&\int_{\Omega_{T}}\vp M_{\hbar,\ve}^{(j+1)}
 \lrs{\pa_{x}M_{\hbar,\ve}^{(m)}}dxdt+A_{\hbar,\ve,2}^{(j,m)}+A_{\hbar,\ve,3}^{(j,m)}+A_{\hbar,\ve,10}^{(j,m)}
+ A_{\hbar,\ve,12}^{(j,m)}+A_{\hbar,\ve,13}^{(j,m)}+A_{\hbar,\ve,14}^{(j,m)}.
\end{align*}
We are left to prove that
\begin{align*}
\lim_{(\hbar,\ve)\to(0,0)}A_{\hbar,\ve,2}^{(j,m)}=&-m\int_{\Omega_{T}}\vp M^{(j)}
\lrs{\int_{-\infty}^{x}\ol{E}\int_{\R}\xi^{m-1}fd\xi dy}dxdt=I_{2}^{(j,m)},\\
\lim_{(\hbar,\ve)\to(0,0)}A_{\hbar,\ve,3}^{(j,m)}=&0,\\
\lim_{(\hbar,\ve)\to(0,0)}A_{\hbar,\ve,10}^{(j,m)}=&\int_{\Omega_{T}} (\pa_{t}\vp)M^{(j)}M^{(m)}dxdt=I_{1}^{(j,m)},\\
\lim_{(\hbar,\ve)\to(0,0)}A_{\hbar,\ve,12}^{(j,m)}=&\int_{\Omega_{T}}\lrs{\pa_{x}\vp} M^{(j+1)}
M^{(m)}dxdt=I_{3}^{(j,m)},\\
\lim_{(\hbar,\ve)\to(0,0)}A_{\hbar,\ve,13}^{(j,m)}=&-j\int_{\Omega_{T}} \vp M^{(j)} \lrs{\int_{-\infty}^{x}\ol{E}\int_{\R}\xi^{m-1}fd\xi dy}dxdt=I_{4}^{(j,m)},\\
\lim_{(\hbar,\ve)\to(0,0)}A_{\hbar,\ve,14}^{(j,m)}=&0.
\end{align*}
It suffices to prove the limits for $A_{\hbar,\ve,2}^{(j,m)}$, $A_{\hbar,\ve,3}^{(j,m)}$, and $A_{\hbar,\ve,10}^{(j,m)}$, as the others can be dealt with in a similar way.

For $A_{\hbar,\ve,2}^{(j,m)}$, we rewrite
\begin{align*}
A_{\hbar,\ve,2}^{(j,m)}=&-m\int_{\Omega_{T}}\lrs{\int_{y}^{\infty}\vp M_{\hbar,\ve}^{(j)}dx}
\lrs{E_{\hbar,\ve}\int_{\R}\xi^{m-1}f_{\hbar,\ve}d\xi}dydt.
\end{align*}
On the one hand, we have the $L_{x}^{\infty}$ convergence that
\begin{align*}
\bbn{\int_{y}^{\infty}\vp M_{\hbar,\ve}^{(j)}dx-\int_{y}^{\infty}\vp M^{(j)}dx}_{L_{x}^{\infty}}\leq \n{\vp}_{L_{x}^{2}}
\n{M_{\hbar,\ve}^{(j)}-M^{(j)}}_{L_{x,loc}^{2}}\to 0.
\end{align*}
On the other hand, by the induction hypothesis for $l\leq k-1$, we have
\begin{align*}
\lim_{(\hbar,\ve)\to(0,0)}\int_{\Omega_{T}}\vp E_{\hbar,\ve}\lrs{\int_{\R}\xi^{m-1}f_{\hbar,\ve}d\xi}dxdt =
\int_{\Omega_{T}}\vp \ol{E}\lrs{\int_{\R}\xi^{m-1}fd\xi}dxdt,
\end{align*}
for $\vp\in L_{t}^{1}([0,T];C_{b}(\R))$. Therefore, we obtain
\begin{align*}
\lim_{(\hbar,\ve)\to (0,0)}A_{\hbar,\ve,2}^{(j,m)}= &-m\int_{\Omega_{T}}\lrs{\int_{y}^{\infty}\vp M^{(j)}dx}
\lrs{\ol{E}\int_{\R}\xi^{m-1}fd\xi}dydt\\
=&-m\int_{\Omega_{T}}\vp M^{(j)}
\lrs{\int_{-\infty}^{x}\ol{E}\int_{\R}\xi^{m-1}fd\xi dy}dxdt\\
=&I_{2}^{(j,m)}.
\end{align*}

For $A_{\hbar,\ve,3}^{(j,m)}$, we rewrite
\begin{align*}
A_{\hbar,\ve,3}^{(j,m)}=&-\int_{\Omega_{T}}\lrs{\int_{y}^{+\infty}\vp M_{\hbar,\ve}^{(j)}dx}
 \mathrm{R}_{\hbar,\ve}^{(m)} dydt\\
 =&-\int_{\Omega_{T}}\lrs{1-\chi(\frac{y}{R})+\chi(\frac{y}{R})}\lrs{\int_{y}^{+\infty}\vp M_{\hbar,\ve}^{(j)}dx}
 \mathrm{R}_{\hbar,\ve}^{(m)} dydt.
\end{align*}
By the quantitative estimate \eqref{equ:rate,remainder} in Lemma \ref{lemma:remainder,vanishing} and the weighted uniform bound \eqref{equ:uniform estimate,remainder} on the remainder term, we have
\begin{align*}
\babs{A_{\hbar,\ve,3}^{(j,m)}}\lesssim& \hbar \bbn{\nabla_{t,y} \lrs{\chi(\frac{y}{R})\int_{y}^{+\infty}\vp M_{\hbar,\ve}^{(j)}dx}}_{L_{t}^{1}L_{x}^{\infty}}+(\hbar+\ve) \bbn{\chi(\frac{y}{R})\int_{y}^{+\infty}\vp M_{\hbar,\ve}^{(j)}dx}_{L_{t}^{1}L_{x}^{\infty}}\\
&+\frac{1}{R}\bbn{\int_{y}^{+\infty}\vp M_{\hbar,\ve}^{(j)}dx}_{L_{t}^{1}L_{x}^{\infty}}\n{\lra{x}\mathrm{R}_{\hbar,\ve}^{(m)}}_{L_{t}^{\infty}L_{x}^{1}}\\
\leq& \hbar \lrs{\n{\vp}_{L_{t}^{1}L_{x}^{\infty}}+\n{\pa_{t}\vp}_{L_{t}^{1}L_{x}^{1}}}\lrs{\n{M_{\hbar,\ve}^{(j)}}_{L_{t}^{\infty}L_{x}^{\infty}}+
\n{\pa_{t}M_{\hbar,\ve}^{(j)}}_{L_{t}^{\infty}L_{x}^{1}}}\\
&+\lrs{\frac{\hbar}{R}+\hbar+\ve+\frac{1}{R}}
\n{\vp}_{L_{t}^{1}L_{x}^{\infty}}\n{M_{\hbar,\ve}^{(j)}}_{L_{t}^{\infty}L_{x}^{\infty}}\\
\lesssim& \hbar+\ve+\frac{1}{R}\to 0,
\end{align*}
where in the last inequality we have used the uniform bounds \eqref{equ:uniform bound,M,h}--\eqref{equ:uniform bound,M,h,t} on $M_{\hbar,\ve}^{(j)}$.

For $A_{\hbar,\ve,10}^{(j,m)}$, noting that
\begin{align*}
M_{\hbar,\ve}^{(j)}\stackrel{L_{loc}^{2}}{\longrightarrow}M^{(j)},
\end{align*}
we immediately get
\begin{align*}
\lim_{(\hbar,\ve)\to (0,0)}A_{\hbar,\ve,10}^{(j,m)}=&\lim_{(\hbar,\ve)\to (0,0)}\int_{\Omega_{T}}(\pa_{t}\vp) M_{\hbar,\ve}^{(j)}
M_{\hbar,\ve}^{(m)}dxdt\\
=&\int_{\Omega_{T}} (\pa_{t}\vp)M^{(j)}M^{(m)}dxdt=I_{1}^{(j,m)}.
\end{align*}
Hence, we complete the proof of Lemma \ref{lemma:induction equation,M,h}.

\end{proof}

Now, we are able to prove the moment convergence of the nonlinear term, which is Lemma \ref{lemma:convergence,nonlinear,momentum,k}, the last remaining part of the proof of Lemma \ref{lemma:momentum convergence,vp}.
\begin{proof}[\textbf{Proof of Lemma \ref{lemma:convergence,nonlinear,momentum,k}}]
By the uniform bound \eqref{equ:uniform bound,E,h} and the weighted uniform estimate \eqref{equ:weighted,momentum of wigner}, we have
\begin{align*}
\bbn{\lra{x}E_{\hbar,\ve} \lrs{\int_{\mathbb{R}} \xi^{k-1} f_{\hbar,\ve} d \xi} d x d t}_{L_{t}^{\infty}L_{x}^{1})}
\leq \n{E_{\hbar,\ve}}_{L_{t}^{\infty}L_{x}^{1}}\bbn{\lra{x}\int_{\mathbb{R}} \xi^{k-1} f_{\hbar,\ve} d \xi}_{L_{t}^{\infty}L_{x}^{1}}\leq C(k).
\end{align*}
Following the same process as the $k=1,2$ case in Lemma \ref{lemma:convergence,nonlinear,momentum,k=1,2}, it suffices to prove
\begin{align}\label{equ:convergence,nonlinear,momentum,proof}
\lim _{(\hbar,\ve)\to (0,0)} \int_{\Omega_T} \vp E_{\hbar,\ve} \lrs{\int_{\mathbb{R}} \xi^{k-1} f_{\hbar,\ve} d \xi} d x d t=\int_{\Omega_T} \vp \ol{E} \lrs{\int_{\mathbb{R}} \xi^{k-1} f(t,dx,d\xi)} dt,
\end{align}
for $\vp\in C_{c}^{\infty}(\Omega_{T})$.

First, we get into the analysis of the term on the left hand side of \eqref{equ:convergence,nonlinear,momentum,proof}.
Noting that
\begin{align*}
E_{\hbar,\ve}=&\pa_{x}V_{\ve}*\rho_{\hbar,\ve}=\pa_{x}\lrs{\frac{|x|}{2}+U_{\ve}(x)}*\rho_{\hbar,\ve},\\
\pa_{x}\lrs{\frac{|x|}{2}*\rho_{\hbar,\ve}}=&\lrs{\frac{x}{2|x|}*\rho_{\hbar,\ve}}=\int_{-\infty}^{x}\rho_{\hbar,\ve}(y)dy-\frac{1}{2}=M_{\hbar,\ve}^{(0)}-\frac{1}{2},
\end{align*}
 we rewrite
\begin{align*}
\int_{\Omega_T} \vp E_{\hbar,\ve} \lrs{\int_{\mathbb{R}} \xi^{k-1} f_{\hbar,\ve} d \xi} d x d t=A_{1}+A_{2}+A_{3},
\end{align*}
where
\begin{align*}
A_{1}=&\int_{\Omega_{T}}\vp M_{\hbar,\ve}^{(0)}\lrs{\pa_{x}M_{\hbar,\ve}^{(k-1)}}dxdt,\\
A_{2}=&-\frac{1}{2}\int_{\Omega_{T}}\vp \lrs{\int_{\mathbb{R}} \xi^{k-1} f_{\hbar,\ve} d \xi} dxdt,\\
A_{3}=&\int_{\Omega_T} \vp \lrs{\pa_{x}U_{\ve}*\rho_{\hbar,\ve}} \lrs{\int_{\mathbb{R}} \xi^{k-1} f_{\hbar,\ve} d \xi} d x d t.
\end{align*}

For term $A_{3}$, using  pointwise estimate \eqref{equ:U,pointwise estimate} that $|\pa_{x}U_{\ve}*\rho_{\hbar,\ve}|\lesssim \ve\lra{x}$ and the weighted estimate \eqref{equ:weighted,momentum of wigner}, we get
\begin{align*}
|A_{3}|=&\bbabs{\int_{\Omega_T} \vp \lrs{\pa_{x}U_{\ve}*\rho_{\hbar,\ve}} \lrs{\int_{\mathbb{R}} \xi^{k-1} f_{\hbar,\ve} d \xi} d x d t}\\
\leq& \ve\n{\vp}_{L_{t}^{1}L_{x}^{\infty}}\bbn{\lra{x}\int_{\mathbb{R}} \xi^{k-1} f_{\hbar,\ve} d \xi}_{L_{t}^{\infty}L_{x}^{1}}\to 0.
\end{align*}
Therefore, we obtain
\begin{align}
&\lim _{(\hbar,\ve)\to (0,0)} \int_{\Omega_T} \vp E_{\hbar,\ve} \lrs{\int_{\mathbb{R}} \xi^{k-1} f_{\hbar,\ve} d \xi} d x d t \label{equ:nonlinear,h,re}\\
=&\lim_{(\hbar,\ve)\to (0,0)}\int_{\Omega_{T}}\vp M_{\hbar,\ve}^{(0)}\lrs{\pa_{x}M_{\hbar,\ve}^{(k-1)}}dxdt-
\frac{1}{2}\int_{\Omega_{T}}\vp \lrs{\int_{\mathbb{R}} \xi^{k-1} f d \xi} dxdt
.\notag
\end{align}

On the other hand,
by \eqref{equ:strong convergence,E,h} and conservation of mass \eqref{equ:convservation,mass} in Lemma \ref{lemma:conservation laws}, we have
\begin{align*}
E(t,x)=&\int_{\R} \frac{x-y}{2|x-y|} \lrs{\int_{\R}f(t,y,\xi)d\xi}dy\\
=&\int_{-\infty}^{x}\int_{\R}f(t,y,\xi)d\xi dy-\frac{1}{2}=M^{(0)}(t,x)-\frac{1}{2},
\end{align*}
and hence obtain
\begin{align}
&\int_{\Omega_T} \vp \ol{E} \lrs{\int_{\mathbb{R}} \xi^{k-1} f(t,dx,d\xi)} dt\label{equ:nonlinear,limit,re}\\
=&\int_{\Omega_{T}}\vp \ol{M}^{(0)}\lrs{\pa_{x}M^{(k-1)}}dxdt-\frac{1}{2}\int_{\Omega_{T}}\vp \lrs{\int_{\mathbb{R}} \xi^{k-1} f d \xi} dxdt.\notag
\end{align}

Comparing \eqref{equ:nonlinear,h,re} with \eqref{equ:nonlinear,limit,re}, to conclude \eqref{equ:convergence,nonlinear,momentum,proof}, we are left to prove
\begin{align}\label{equ:convergence,order,0}
\lim_{(\hbar,\ve)\to (0,0)}\int_{\Omega_{T}}\vp M_{\hbar,\ve}^{(0)}\lrs{\pa_{x}M_{\hbar,\ve}^{(k-1)}}dxdt=
\int_{\Omega_{T}}\vp \ol{M}^{(0)}\lrs{\pa_{x}M^{(k-1)}}dxdt.
\end{align}
By Lemma \ref{lemma:induction equation,limit case} and Lemma \ref{lemma:induction equation,M,h}, the equality \eqref{equ:convergence,order,0} is equivalent to
\begin{align}\label{equ:convergence,order,1}
\lim_{(\hbar,\ve)\to (0,0)}\int_{\Omega_{T}}\vp M_{\hbar,\ve}^{(1)}\lrs{\pa_{x}M_{\hbar,\ve}^{(k-2)}}dxdt=
\int_{\Omega_{T}}\vp \ol{M}^{(1)}\lrs{\pa_{x}M^{(k-2)}}dxdt.
\end{align}
Iteratively using Lemma \ref{lemma:induction equation,limit case} and Lemma \ref{lemma:induction equation,M,h}, we are left to prove
\begin{align}
\lim_{(\hbar,\ve)\to (0,0)}\int_{\Omega_{T}}\vp M_{\hbar,\ve}^{(n)}\lrs{\pa_{x}M_{\hbar,\ve}^{(n)}}dxdt=&
\int_{\Omega_{T}}\vp \ol{M}^{(n)}\lrs{\pa_{x}M^{(n)}}dxdt,\quad k=2n+1,\label{equ:convergence,order,n}\\
\lim_{(\hbar,\ve)\to (0,0)}\int_{\Omega_{T}}\vp M_{\hbar,\ve}^{(n)}\lrs{\pa_{x}M_{\hbar,\ve}^{(n+1)}}dxdt=&
\int_{\Omega_{T}}\vp \ol{M}^{(n)}\lrs{\pa_{x}M^{(n+1)}}dxdt,\quad k=2n.\label{equ:convergence,order,n+1}
\end{align}

For \eqref{equ:convergence,order,n}, by integration by parts we get
\begin{align*}
\lim_{(\hbar,\ve)\to (0,0)}\int_{\Omega_{T}}\vp M_{\hbar,\ve}^{(n)}\lrs{\pa_{x}M_{\hbar,\ve}^{(n)}}dxdt=&
-\frac{1}{2}\lim_{(\hbar,\ve)\to (0,0)}\int_{\Omega_{T}}\lrs{\pa_{x}\vp} \lrs{M_{\hbar,\ve}^{(n)}}^{2}dxdt\\
=&-\frac{1}{2}\int_{\Omega_{T}}\lrs{\pa_{x}\vp} \lrs{M^{(n)}}^{2}dxdt\\
=&\int_{\Omega_{T}}\vp \ol{M}^{(n)}\lrs{\pa_{x}M^{(n)}}dxdt,
\end{align*}
where in the last equality we have used the fact that $\pa_{x}(M^{(n)})^{2}=2\ol{M}^{(n)}\pa_{x}M^{(n)}$.

For \eqref{equ:convergence,order,n+1}, by the equations \eqref{equ:M,m,h,equation} and \eqref{equ:M,equation}, it suffices to prove
\begin{align*}
\lim_{(\hbar,\ve)\to (0,0)}\int_{\Omega_{T}}\vp M_{\hbar,\ve}^{(n)}\lrs{\pa_{t}M_{\hbar,\ve}^{(n)}}dxdt=&
\int_{\Omega_{T}}\vp \ol{M}^{(n)}\lrs{\pa_{t}M^{(n)}}dxdt.
\end{align*}
This can be done in the same way in which we obtain \eqref{equ:convergence,order,n}. Hence, we complete the proof of Lemma \ref{lemma:convergence,nonlinear,momentum,k}.
\end{proof}

\section{Full Convergence to the Vlasov-Poisson Equation}\label{section:Full Convergence to the Vlasov-Poisson Equation}
In the section, we prove the limit measure $f(t,dx,d\xi)$ satisfies the Vlasov-Poisson equation in the weak sense.
Let
\begin{align}
\mu:=\pa_{t}f+\xi\pa_{x}f-\pa_{\xi}(\ol{E}f).
\end{align}
We use the following lemma to conclude the full convergence to the Vlasov-Poisson equation, that is, $\mu(t,x,\xi)=0$ in the sense of distributions.
\begin{lemma}[{\hspace{-0.05em}\cite[$p.620$]{ZZM02}}]\label{lemma:uniqueness,limit measure}
Let $\Omega_{T}=(0,T)\times \R$, $\delta$ be an arbitrary positive constant. Assume $f(t,dx,d\xi)$ satisfies the following conditions.
\begin{enumerate}
\item[$(1)$]  Exponential decay:
\begin{align}
\iint_{\Omega_{T}}\int_{\R}e^{\delta|\xi|}f(t,dx,d\xi)dt\leq C_{\delta}.
\end{align}
\item[$(2)$] For all test functions of the form $\phi(t,x,\xi)=\vp(t,x)\xi^{m}$, $\vp(t,x)\in C_{c}^{\infty}(\Omega_{T})$, there holds that
\begin{align}
\iint_{\Omega_{T}}\int_{\R}\phi d\mu(t,x,\xi)=0.
\end{align}
\end{enumerate}
Then $\mu(t,x,\xi)=0$ in the sense of distributions.
\end{lemma}

By the moment convergence in Lemma \ref{lemma:momentum convergence,vp}, we have verified the condition $(2)$ in Lemma \ref{lemma:uniqueness,limit measure}. Therefore, we are left to prove the exponential decay condition.

\begin{lemma}
Let $T>0$. There holds that
\begin{align}
\iint_{\R^{2}}\xi^{2k}f(t,dx,d\xi)\leq C^{2k}(2k)^{2k}e^{t},\quad \f t\in[0,T].
\end{align}
In particular, there exists a positive constant $\delta$ such that
\begin{align}\label{equ:exponential decay}
\iint_{\R^{2}} e^{\delta|\xi|}f(t,dx,d\xi)\leq C_{\delta}e^{t},\quad \f t\in[0,T].
\end{align}
\end{lemma}
\begin{proof}
Recalling the moment equation \eqref{equ:momentum,f} that
\begin{align*}
\pa_{t}\int_{\R}\xi^{m}f_{\hbar,\ve}d\xi+\pa_{x}\int_{\R}\xi^{m+1}f_{\hbar,\ve}d\xi+
mE_{\hbar,\ve}\int_{\R}\xi^{m-1}f_{\hbar,\ve}d\xi+\mathrm{R}_{\hbar,\ve}^{(m)}=0,
\end{align*}
we have
\begin{align*}
&\iint_{\R^{2}}\xi^{2k}f_{\hbar,\ve}(t,x,\xi)d\xi dx\\
=&\iint_{\R^{2}}\xi^{2k}f_{\hbar,\ve}(0,x,\xi) d\xi dx+
2k\int_{\Omega_{t}}E_{\hbar,\ve}\lrs{\int_{\R}\xi^{2k-1}f_{\hbar,\ve}d\xi} dx d\tau+\int_{\Omega_{t}}\mathrm{R}_{\hbar,\ve}^{(2k)}dx d\tau.
\end{align*}
By the narrow convergence in Lemma \ref{lemma:convergence}, Lemma \ref{lemma:remainder,vanishing}, and Lemma \ref{lemma:convergence,nonlinear,momentum,k}, taking $\vp(\tau,x)=1$ and letting $(\hbar,\ve)\to (0,0)$, we obtain
\begin{align*}
\iint_{\R^{2}}\xi^{2k}f(t,dx,d\xi)
=&\iint_{\R^{2}}\xi^{2k}f(0,dx,d\xi)+
2k\int_{\Omega_{t}}\ol{E}\int_{\R}\xi^{2k-1}f(\tau,dx,d\xi)d\tau.
\end{align*}

For the initial data, we have
\begin{align*}
\iint_{\R^{2}}\xi^{2k}f(0,dx,d\xi)
=&\lim_{\hbar\to 0}\iint_{\R^{2}}\xi^{2k}f_{\hbar}(0,x,\xi)d\xi dx\\
=&\lim_{\hbar\to 0}\frac{\hbar^{2k}}{2^{2k}}\sum_{\al=0}^{2k}\binom{2k}{\al}(-1)^{2k-\al} \int_{\R}D_{x}^{\al}\psi_{\hbar}^{\mathrm{in}}\ol{D_{x}^{2k-\al}\psi_{\hbar}^{\mathrm{in}}}dx\\
\leq&\sup_{\hbar}\frac{1}{2^{2k}}\sum_{\al=0}^{2k}\binom{2k}{\al}\n{\hbar^{\al}\pa_{x}^{\al}\psi_{\hbar}^{\mathrm{in}}}_{L_{x}^{2}}
\n{\hbar^{2k-\al}\pa_{x}^{2k-\al}\psi_{\hbar}^{\mathrm{in}}}_{L_{x}^{2}}\\
\leq& C^{2k}(2k)^{2k},
\end{align*}
where in the last inequality we have used the initial condition \eqref{equ:uniform bounds,initial data} that $$\n{\hbar^{\al}\pa_{x}^{\al}\psi_{\hbar}^{\mathrm{in}}}_{L_{x}^{2}}\leq C^{\al}\al^{\al}.$$

For the nonlinear part, using that $2k|\xi|^{2k-1}\leq (2k)^{2k}+\xi^{2k}$, we get
\begin{align*}
&2k\bbabs{\int_{\Omega_{t}}\ol{E}\int_{\R}\xi^{2k-1}f(\tau,dx,d\xi)d\tau}\\
\leq& \n{\ol{E}}_{L_{t,x}^{\infty}}\int_{0}^{t} \iint_{\R^{2}}((2k)^{2k}+\xi^{2k})f(\tau,dx,d\xi)d\tau\\
\leq& (2k)^{2k}T+\int_{0}^{t} \iint_{\R^{2}}\xi^{2k}f(\tau,dx,d\xi)d\tau,
\end{align*}
where in the last inequality we have used that $\n{\ol{E}}_{L_{t,x}^{\infty}}\leq 1$ in \eqref{equ:BV,E}.
Thus, we arrive at
\begin{align*}
\iint_{\R^{2}}\xi^{2k}f(t,dx,d\xi)\leq C^{2k}(2k)^{2k}+T(2k)^{2k}+\int_{0}^{t} \iint_{\R^{2}}\xi^{2k}f(\tau,dx,d\xi)d\tau.
\end{align*}
Then by Gronwall's inequality, we get
\begin{align}
\iint_{\R^{2}}\xi^{2k}f(t,dx,d\xi)\leq (C^{2k}+T)(2k)^{2k}e^{t}.
\end{align}

For the exponential decay \eqref{equ:exponential decay}, provided that $C\delta<1$, we have
\begin{align*}
\iint_{\R^{2}}e^{\delta|\xi|}f(t,dx,d\xi)\leq& \iint_{\R^{2}}\lrs{e^{\delta|\xi|}+e^{-\delta|\xi|}}f(t,dx,d\xi)\\
=& 2\sum_{k=0}^{\infty}\frac{\delta^{2k}}{2k!}\iint_{\R^{2}}\xi^{2k}f(t,dx,d\xi)\\
\leq&2e^{t}\sum_{k=0}^{\infty}\frac{\delta^{2k}(C^{2k}+T)(2k)^{2k}}{2k!}<\infty.
\end{align*}
\end{proof}

\appendix
\section{Measure Solutions to the Vlasov-Poisson Equation}\label{section:The Measure Solution to the Vlasov-Poisson Equation}
Let us recall the definition of weak measure solutions from \cite{ZM94} by Zheng and Majda.
\begin{definition}\label{def:weak solution,vp}
A pair $(E(t,x),f(t,x,\xi))$ of a function and a bounded  non-negative Radon measure is called
a weak solution to the Vlasov-Poisson equation \eqref{equ:vlasov-poisson,1d} if for any $T>0$ there hold
\begin{enumerate}
\item $E(t,x)\in (BV\cap L^{\infty})(\Omega_{T})$, where $\Omega_{T}=(0,T)\times \R$;
\item $f(t,x,\xi)\in L^{\infty}(0,\infty;\mathcal{M}^{+}(\R^{2}))$;
\item $E(t,x)=\frac{x}{|x|}*\int_{\R}f(t,x,\xi)d\xi$ a.e.;
\item $\f \phi\in C_{c}^{\infty}((0,T)\times \R^{2})$,
\begin{align*}
\int_{0}^{T}\iint_{\R^{2}}\lrs{\pa_{t}\phi} f+\lrs{\pa_{x}\phi} \xi fdx d\xi dt-
\int_{0}^{T}\int_{\R}\ol{E}\int_{\R}\lrs{\pa_{\xi}\phi}f(d\xi)dxdt=0.
\end{align*}
\item $f\in C^{0,1}([0,T);H^{-L}(\R^{2}))$ for some $L>0$.
\end{enumerate}

The term $\ol{E}(t,x)$ in the above definition is the Vol$^\prime$pert's symmetric average:
\begin{align}\label{equ:volpert average}
\ol{E}(t,x)=
\left\{
\begin{aligned}
&E(t,x)&\quad \text{if $E(t,x)$ is approximately continuous at $(t,x)$,}\\
&\frac{E_{l}(t,x)+E_{r}(t,x)}{2}&\quad \text{if $E(t,x)$ has a jump at $(t,x)$.}
\end{aligned}
\right.
\end{align}
where $E_{l}(t,x)$ and $E_{r}(t,x)$ denote, respectively, the left and right limits of $E(t,x)$ at a discontinuity line at $(t,x)$.
\end{definition}
\section{Basic Properties of Bounded Variation Functions}\label{section:Basic Properties of BV Functions}
We provide some basic properties of BV functions which are used in the paper. For more details, see for instance
\cite{Vol67}, or \cite{ZZM02,ZM94}.

Let $\Omega$ be a Borel measurable subset of $\R^{2}$.
\begin{enumerate}
\item If $E\in BV(\Omega)\cap L^{\infty}(\Omega)$, then
\begin{align*}
E^{2}\in BV(\Omega),\quad \nabla E^{2}=2\ol{E} \nabla E
\end{align*}
in the sense of measures.
\item If $u,v\in BV(\Omega)$, then $\ol{u}$ is almost everywhere defined and measurable with respect to $\nabla v$. Furthermore, $\ol{u}$ is integrable with respect to $\nabla v$ if $u$ is bounded.
\item If $u,v\in BV(\Omega)$, $\ol{u}$ is locally integrable with respect to $\nabla v$ and $\ol{v}$ is locally integrable with respect to
$\nabla u$. Then $uv\in BV(\Omega)$ and
\begin{align*}
\nabla(uv)=\ol{u}\nabla v+\ol{v}\nabla u.
\end{align*}
\item $\ol{\vp E}=\vp \ol{E}$ if $\vp\in C^{1}(\Omega)$.
\item Let $u\in BV(\Omega)\cap L^{\infty}(\Omega)$, and $\eta_{\sigma}$ be an approximation of the identity. Then
\begin{align}
u*\eta_{\sigma}\to \ol{u}\quad  \mathcal{H}^{1}-a.e.
\end{align}
as $\sigma\to 0$. Here, $\mathcal{H}^{1}$ denotes the one-dimensional Hausdorff measure.
\item $\nabla u$ is absolutely continuous with respect to $\mathcal{H}^{1}$ for any $u\in BV(\Omega)$.
\item $\ol{E}$ is $\mathcal{H}^{1}$-a.e. defined for any $E\in BV$.
\end{enumerate}

\noindent \textbf{Acknowledgements}
X. Chen was supported in part by U.S. NSF grant DMS-2406620.
P. Zhang was supported in part by National Key R$\&$D Program of China under Grant 2021YFA1000800 and NSF of China under Grants 12288201, 12031006.
Z. Zhang was supported in part by National Key R\&D Program of China under Grant 2023YFA1008801 and  NSF of China under Grant 12288101.

\bibliographystyle{abbrv}
\bibliography{references}

\begin{thebibliography}{10}

\bibitem{AGT07}
R.~Adami, F.~Golse, and A.~Teta.
\newblock Rigorous derivation of the cubic {NLS} in dimension one.
\newblock {\em J. Stat. Phys.}, 127(6):1193--1220, 2007.

\bibitem{akhiezer2017plasma}
A.~I. Akhiezer, I.~Akhiezer, and R.~V. Polovin.
\newblock {\em Plasma Electrodynamics: Linear Theory}, volume~1.
\newblock Elsevier, 2017.

\bibitem{BEGMH02}
C.~Bardos, L.~Erd\H{o}s, F.~Golse, N.~Mauser, and H.-T. Yau.
\newblock Derivation of the {S}chr\"{o}dinger-{P}oisson equation from the
  quantum {$N$}-body problem.
\newblock {\em C. R. Math. Acad. Sci. Paris}, 334(6):515--520, 2002.

\bibitem{BG24}
I.~Ben~Porat and F.~Golse.
\newblock {Pickl's proof of the quantum mean-field limit and quantum
  Klimontovich solutions}.
\newblock {\em Lett. Math. Phys.}, 114(2):51, 2024.

\bibitem{BOS15}
N.~Benedikter, G.~de~Oliveira, and B.~Schlein.
\newblock Quantitative derivation of the {G}ross-{P}itaevskii equation.
\newblock {\em Comm. Pure Appl. Math.}, 68(8):1399--1482, 2015.

\bibitem{BPSS16}
N.~Benedikter, M.~Porta, C.~Saffirio, and B.~Schlein.
\newblock From the {H}artree dynamics to the {V}lasov equation.
\newblock {\em Arch. Ration. Mech. Anal.}, 221(1):273--334, 2016.

\bibitem{BS19}
C.~Brennecke and B.~Schlein.
\newblock Gross-{P}itaevskii dynamics for {B}ose-{E}instein condensates.
\newblock {\em Anal. PDE}, 12(6):1513--1596, 2019.

\bibitem{Car08}
R.~Carles.
\newblock {\em Semi-classical analysis for nonlinear {S}chr\"{o}dinger
  equations}.
\newblock World Scientific Publishing Co. Pte. Ltd., Hackensack, NJ, 2008.

\bibitem{CP11}
T.~Chen and N.~Pavlovi\'{c}.
\newblock The quintic {NLS} as the mean field limit of a boson gas with
  three-body interactions.
\newblock {\em J. Funct. Anal.}, 260(4):959--997, 2011.

\bibitem{CP14}
T.~Chen and N.~Pavlovi\'{c}.
\newblock Derivation of the cubic {NLS} and {G}ross-{P}itaevskii hierarchy from
  manybody dynamics in {$d=3$} based on spacetime norms.
\newblock {\em Ann. Henri Poincar\'{e}}, 15(3):543--588, 2014.

\bibitem{CH16on}
X.~Chen and J.~Holmer.
\newblock On the {K}lainerman-{M}achedon conjecture for the quantum {BBGKY}
  hierarchy with self-interaction.
\newblock {\em J. Eur. Math. Soc. (JEMS)}, 18(6):1161--1200, 2016.

\bibitem{CH19}
X.~Chen and J.~Holmer.
\newblock The derivation of the {$\Bbb T^3$} energy-critical {NLS} from quantum
  many-body dynamics.
\newblock {\em Invent. Math.}, 217(2):433--547, 2019.

\bibitem{CH22quantitative}
X.~Chen and J.~Holmer.
\newblock Quantitative derivation and scattering of the 3{D} cubic {NLS} in the
  energy space.
\newblock {\em Ann. PDE}, 8(2):Paper No. 11, 39, 2022.

\bibitem{CSWZ24}
X.~Chen, S.~Shen, J.~Wu, and Z.~Zhang.
\newblock The derivation of the compressible {E}uler equation from quantum
  many-body dynamics.
\newblock {\em Peking Math. J.}, 7(1):35--90, 2024.

\bibitem{CSZ23}
X.~Chen, S.~Shen, and Z.~Zhang.
\newblock Quantitative derivation of the {Euler-Poisson} equation from quantum
  many-body dynamics.
\newblock {\em Peking Math. J., https://doi.org/10.1007/s42543-023-00065-5},
  2023.

\bibitem{CS04}
Y.~C. Chen, C.~E. Simien, S.~Laha, P.~Gupta, Y.~N. Martinez, P.~G. Mickelson,
  S.~B. Nagel, and T.~C. Killian.
\newblock Electron screening and kinetic-energy oscillations in a strongly
  coupled plasma.
\newblock {\em Phys. Rev. Lett.}, 93:265003, Dec 2004.

\bibitem{DP88solutions}
R.~DiPerna and P.-L. Lions.
\newblock Solutions globales d'\'{e}quations du type {V}lasov-{P}oisson.
\newblock {\em C. R. Acad. Sci. Paris S\'{e}r. I Math.}, 307(12):655--658,
  1988.

\bibitem{DP88}
R.~J. DiPerna and P.-L. Lions.
\newblock Global weak solutions of kinetic equations.
\newblock {\em Rend. Sem. Mat. Univ. Politec. Torino}, 46(3):259--288 (1990),
  1988.

\bibitem{DP89}
R.~J. DiPerna and P.-L. Lions.
\newblock On the {C}auchy problem for {B}oltzmann equations: global existence
  and weak stability.
\newblock {\em Ann. of Math. (2)}, 130(2):321--366, 1989.

\bibitem{EESY04}
A.~Elgart, L.~Erd\H{o}s, B.~Schlein, and H.-T. Yau.
\newblock Nonlinear {H}artree equation as the mean field limit of weakly
  coupled fermions.
\newblock {\em J. Math. Pures Appl. (9)}, 83(10):1241--1273, 2004.

\bibitem{EESY06}
A.~Elgart, L.~Erd\H{o}s, B.~Schlein, and H.-T. Yau.
\newblock Gross-{P}itaevskii equation as the mean field limit of weakly coupled
  bosons.
\newblock {\em Arch. Ration. Mech. Anal.}, 179(2):265--283, 2006.

\bibitem{ES07}
A.~Elgart and B.~Schlein.
\newblock Mean field dynamics of boson stars.
\newblock {\em Comm. Pure Appl. Math.}, 60(4):500--545, 2007.

\bibitem{ESY06}
L.~Erd\H{o}s, B.~Schlein, and H.-T. Yau.
\newblock Derivation of the {G}ross-{P}itaevskii hierarchy for the dynamics of
  {B}ose-{E}instein condensate.
\newblock {\em Comm. Pure Appl. Math.}, 59(12):1659--1741, 2006.

\bibitem{ESY07}
L.~Erd\H{o}s, B.~Schlein, and H.-T. Yau.
\newblock Derivation of the cubic non-linear {S}chr\"{o}dinger equation from
  quantum dynamics of many-body systems.
\newblock {\em Invent. Math.}, 167(3):515--614, 2007.

\bibitem{ESY09}
L.~Erd\H{o}s, B.~Schlein, and H.-T. Yau.
\newblock Rigorous derivation of the {G}ross-{P}itaevskii equation with a large
  interaction potential.
\newblock {\em J. Amer. Math. Soc.}, 22(4):1099--1156, 2009.

\bibitem{ESY10}
L.~Erd\H{o}s, B.~Schlein, and H.-T. Yau.
\newblock Derivation of the {G}ross-{P}itaevskii equation for the dynamics of
  {B}ose-{E}instein condensate.
\newblock {\em Ann. of Math. (2)}, 172(1):291--370, 2010.

\bibitem{EY01}
L.~Erd\H{o}s and H.-T. Yau.
\newblock Derivation of the nonlinear {S}chr\"{o}dinger equation from a many
  body {C}oulomb system.
\newblock {\em Adv. Theor. Math. Phys.}, 5(6):1169--1205, 2001.

\bibitem{GMP16}
F.~Golse, C.~Mouhot, and T.~Paul.
\newblock On the mean field and classical limits of quantum mechanics.
\newblock {\em Comm. Math. Phys.}, 343(1):165--205, 2016.

\bibitem{GP17}
F.~Golse and T.~Paul.
\newblock The {S}chr\"{o}dinger equation in the mean-field and semiclassical
  regime.
\newblock {\em Arch. Ration. Mech. Anal.}, 223(1):57--94, 2017.

\bibitem{GP22}
F.~Golse and T.~Paul.
\newblock Mean-field and classical limit for the {$N$}-body quantum dynamics
  with {C}oulomb interaction.
\newblock {\em Comm. Pure Appl. Math.}, 75(6):1332--1376, 2022.

\bibitem{GMP03}
S.~Graffi, A.~Martinez, and M.~Pulvirenti.
\newblock Mean-field approximation of quantum systems and classical limit.
\newblock {\em Math. Models Methods Appl. Sci.}, 13(1):59--73, 2003.

\bibitem{GM13}
M.~Grillakis and M.~Machedon.
\newblock Pair excitations and the mean field approximation of interacting
  bosons, {I}.
\newblock {\em Comm. Math. Phys.}, 324(2):601--636, 2013.

\bibitem{HSE03}
J.~Heyd, G.~E. Scuseria, and M.~Ernzerhof.
\newblock {Hybrid functionals based on a screened Coulomb potential}.
\newblock {\em J. Chem. Phys}, 118(18):8207--8215, 2003.

\bibitem{HC22}
J.~Hofierka, B.~Cunningham, C.~M. Rawlins, C.~H. Patterson, and D.~G. Green.
\newblock Many-body theory of positron binding to polyatomic molecules.
\newblock {\em Nature}, 606(7915):688--693, 2022.

\bibitem{KSS11}
K.~Kirkpatrick, B.~Schlein, and G.~Staffilani.
\newblock Derivation of the two-dimensional nonlinear {S}chr\"{o}dinger
  equation from many body quantum dynamics.
\newblock {\em Amer. J. Math.}, 133(1):91--130, 2011.

\bibitem{KM08}
S.~Klainerman and M.~Machedon.
\newblock On the uniqueness of solutions to the {G}ross-{P}itaevskii hierarchy.
\newblock {\em Comm. Math. Phys.}, 279(1):169--185, 2008.

\bibitem{LS23}
L.~Lafleche and C.~Saffirio.
\newblock Strong semiclassical limits from {H}artree and {H}artree-{F}ock to
  {V}lasov-{P}oisson equations.
\newblock {\em Anal. PDE}, 16(4):891--926, 2023.

\bibitem{LP93}
P.-L. Lions and T.~Paul.
\newblock Sur les mesures de {W}igner.
\newblock {\em Rev. Mat. Iberoamericana}, 9(3):553--618, 1993.

\bibitem{LP91}
P.-L. Lions and B.~Perthame.
\newblock Propagation of moments and regularity for the {$3$}-dimensional
  {V}lasov-{P}oisson system.
\newblock {\em Invent. Math.}, 105(2):415--430, 1991.

\bibitem{Liu21}
X.~Liu, Z.~Wang, K.~Watanabe, T.~Taniguchi, O.~Vafek, and J.~Li.
\newblock Tuning electron correlation in magic-angle twisted bilayer graphene
  using {C}oulomb screening.
\newblock {\em Science}, 371(6535):1261--1265, 2021.

\bibitem{MM93}
P.~A. Markowich and N.~J. Mauser.
\newblock The classical limit of a self-consistent quantum-{V}lasov equation in
  {$3$}{D}.
\newblock {\em Math. Models Methods Appl. Sci.}, 3(1):109--124, 1993.

\bibitem{Pic11}
P.~Pickl.
\newblock A simple derivation of mean field limits for quantum systems.
\newblock {\em Lett. Math. Phys.}, 97(2):151--164, 2011.

\bibitem{RS09}
I.~Rodnianski and B.~Schlein.
\newblock Quantum fluctuations and rate of convergence towards mean field
  dynamics.
\newblock {\em Comm. Math. Phys.}, 291(1):31--61, 2009.

\bibitem{RGH70}
F.~J. Rogers, H.~C. Graboske, and D.~J. Harwood.
\newblock Bound eigenstates of the static screened {C}oulomb potential.
\newblock {\em Phys. Rev. A}, 1:1577--1586, Jun 1970.

\bibitem{Ser20}
S.~Serfaty.
\newblock Mean field limit for {C}oulomb-type flows.
\newblock {\em Duke Math. J.}, 169(15):2887--2935, 2020.
\newblock With an appendix by Mitia Duerinckx and Serfaty.

\bibitem{Vol67}
A.~I. Vol$^\prime$pert.
\newblock Spaces {${\rm BV}$} and quasilinear equations.
\newblock {\em Mat. Sb. (N.S.)}, 73(115):255--302, 1967.

\bibitem{WS11}
T.~O. Wehling, E.~\ifmmode \mbox{\c{S}}\else \c{S}\fi{}a\ifmmode
  \mbox{\c{s}}\else \c{s}\fi{}\ifmmode \imath \else \i
  \fi{}o\ifmmode~\breve{g}\else \u{g}\fi{}lu, C.~Friedrich, A.~I. Lichtenstein,
  M.~I. Katsnelson, and S.~Bl\"ugel.
\newblock Strength of effective {C}oulomb interactions in graphene and
  graphite.
\newblock {\em Phys. Rev. Lett.}, 106:236805, Jun 2011.

\bibitem{Zha02}
P.~Zhang.
\newblock Wigner measure and the semiclassical limit of
  {S}chr\"{o}dinger-{P}oisson equations.
\newblock {\em SIAM J. Math. Anal.}, 34(3):700--718, 2002.

\bibitem{Zha08}
P.~Zhang.
\newblock {\em Wigner measure and semiclassical limits of nonlinear
  {S}chr\"{o}dinger equations}, volume~17 of {\em Courant Lecture Notes in
  Mathematics}.
\newblock Courant Institute of Mathematical Sciences, New York; American
  Mathematical Society, Providence, RI, 2008.

\bibitem{ZZM02}
P.~Zhang, Y.~Zheng, and N.~J. Mauser.
\newblock The limit from the {S}chr\"{o}dinger-{P}oisson to the
  {V}lasov-{P}oisson equations with general data in one dimension.
\newblock {\em Comm. Pure Appl. Math.}, 55(5):582--632, 2002.

\bibitem{ZM94}
Y.~Zheng and A.~Majda.
\newblock Existence of global weak solutions to one-component
  {V}lasov-{P}oisson and {F}okker-{P}lanck-{P}oisson systems in one space
  dimension with measures as initial data.
\newblock {\em Comm. Pure Appl. Math.}, 47(10):1365--1401, 1994.

\end{thebibliography}

\end{document}